\documentclass[11pt,reqno]{amsart}

\usepackage{amsfonts, amssymb, amsthm, amsmath, bbm, wasysym, enumerate, amsxtra, latexsym, fullpage, amscd, graphics, epic, youngtab, sansmath, mathtools, multicol, url, hyperref, bbm, blindtext, wasysym, comment, datetime, stmaryrd, tabu} %
\usepackage{mathdots} 
\usepackage{arydshln} 
\usepackage{scalerel} 
\usepackage[shortlabels]{enumitem}
\usepackage[all,cmtip]{xy}
\usepackage[dvipsnames]{xcolor}
\usepackage[makeroom]{cancel}  
\usepackage{blkarray} 
\usepackage{youngtab}
\usepackage{tikz}
\usetikzlibrary{cd,arrows}

\usepgflibrary{shapes.geometric}

\tikzstyle{V}=[draw, fill =black, circle, inner sep=0pt, minimum size=1.5pt]
\tikzstyle{C}=[draw, fill =white, circle, inner sep=0pt, minimum size=1.5pt]
\tikzstyle{over}=[draw=white,double=black,line width=2pt, double distance=.5pt]
\numberwithin{equation}{section}
\usepackage[margin = 0.9in]{geometry}
\theoremstyle{definition}

\newtheorem*{thmA}{Theorem A}
\newtheorem*{thmB}{Theorem B}

\newtheorem{theorem}{Theorem}[section]
\newtheorem{alg}[theorem]{Algorithm}
\newtheorem{thm}[theorem]{Theorem}
\newtheorem{lemma}[theorem]{Lemma} 
\newtheorem{prop}[theorem]{Proposition} 
\newtheorem{proposition}[theorem]{Proposition}
\newtheorem{cor}[theorem]{Corollary}
\newtheorem{definition}[theorem]{Definition}
\newtheorem{remark}[theorem]{Remark}

\newtheorem{ex}[theorem]{Example}

\newtheorem{rmk}[theorem]{Remark}

\def\<{\langle}
\def\>{\rangle}

\newcommand{\BD}{{B}^\textup{D}}

\newcommand{\cB}{\mathcal{B}}
\newcommand{\cBD}{\mathcal{B}^\textup{D}}\newcommand{\CC}{\mathbb{C}}

\newcommand{\cN}{\mathcal{N}}
\newcommand{\cND}{\mathcal{N}^{\textup{D}}}

\newcommand{\CP}{\mathbb{P}_\mathbb{C}}

\newcommand{\cS}{\mathcal{S}}
\newcommand{\cSD}{\mathcal{S}^{\textup{D}}}

\newcommand{\da}{\dot{a}}
\newcommand{\db}{\dot{b}}
\newcommand{\dc}{\dot{c}}

\newcommand{\fg}{\mathfrak{g}}
\newcommand{\fgl}{\mathfrak{gl}} 
\newcommand{\tforall}{~\textup{ for all }~}
 
\newcommand{\fsl}{\mathfrak{sl}} 
\newcommand{\fso}{\mathfrak{so}} 
\newcommand{\fT}{\mathfrak{T}}

\newcommand{\GL}{\mathrm{GL}}

\newcommand{\inv}{^{-1}}
 
\newcommand{\II}{\textup{II}}
\newcommand{\III}{\textup{III}}

\newcommand{\ld}{\lambda}

\newcommand{\otw}{\textup{otherwise}}

\newcommand{\rk}{\mathop{\textup{rank}}}

\renewcommand{\SS}{\mathbb{S}}

\newcommand{\tA}{\widetilde{A}}
\newcommand{\tB}{\widetilde{B}}
\newcommand{\tcN}{\widetilde{\cN}}
\newcommand{\tcND}{\widetilde{\cN}^{\textup{D}}}
\newcommand{\tcS}{\widetilde{\cS}}
\newcommand{\tcSD}{\widetilde{\cS}^{\textup{D}}}
\newcommand{\td}{\widetilde{d}}
\newcommand{\tD}{\widetilde{D}}
\newcommand{\tDe}{\widetilde{\Delta}}

\newcommand{\tGa}{\widetilde{\Gamma}}

\newcommand{\tif}{\textup{if }}
\newcommand{\tphi}{\widetilde{\varphi}}

\newcommand{\TT}{\mathbb{T}}
\newcommand{\tv}{\widetilde{v}}
\newcommand{\tV}{\widetilde{V}}

\newcommand{\ZZ}{\mathbb{Z}}
\renewcommand{\Im}{\mathop{\text{Im}}}
 
\allowdisplaybreaks 
\newcommand{\cupa}{\begin{tikzpicture}[baseline={(0,-.3)}, scale = 0.8]
\draw(-.25,0) -- (1.25,0) -- (1.25,-.9) -- (-.25,-.9) -- cycle;
\begin{footnotesize}
\node at (0,.2) {$1$};
\node at (.5,.2) {$2$};
\node at (1,.2) {$3$};
\end{footnotesize}
\draw[thick] (0,0) .. controls +(0,-.5) and +(0,-.5) .. +(.5,0);

\draw[thick] (1,0) -- +(0,-.9);
\end{tikzpicture}
}

\newcommand{\cupb}{\begin{tikzpicture}[baseline={(0,-.3)}, scale = 0.8]
\draw(-.25,0) -- (1.25,0) -- (1.25,-.9) -- (-.25,-.9) -- cycle;
\begin{footnotesize}
\node at (0,.2) {$1$};
\node at (.5,.2) {$2$};
\node at (1,.2) {$3$};
\end{footnotesize}
\draw[thick] (0.5,0) .. controls +(0,-.5) and +(0,-.5) .. +(.5,0);

\draw[thick] (0,0) -- +(0,-.9);
\end{tikzpicture}
}

\newcommand{\cupc}{\begin{tikzpicture}[baseline={(0,-.3)}, scale = 0.8]
\draw (-.25,0) -- (1.25,0) -- (1.25,-.9) -- (-.25,-.9) -- cycle;
\begin{footnotesize}
\node at (0,.2) {$1$};
\node at (.5,.2) {$2$};
\node at (1,.2) {$3$};
\end{footnotesize}
\draw[thick] (0,0) -- +(0,-.9);
\draw[thick] (0.5,0) -- +(0,-.9);
\draw[thick] (1,0) -- +(0,-.9);
\end{tikzpicture}
}

\newcommand{\cupaa}{\begin{tikzpicture}[baseline={(0,-.3)}, scale = 0.8]
\draw (-.25,0) -- (1.75,0) -- (1.75,-.9) -- (-.25,-.9) -- cycle;
\begin{footnotesize}
\node at (0,.2) {$1$};
\node at (.5,.2) {$2$};
\node at (1,.2) {$3$};
\node at (1.5,.2) {$4$};
\end{footnotesize}
\draw[thick] (0.5,0) .. controls +(0,-.5) and +(0,-.5) .. +(.5,0);
\draw[thick] (0,0) .. controls +(0,-1) and +(0,-1) .. +(1.5,0);
\end{tikzpicture}
}

\newcommand{\cupab}{\begin{tikzpicture}[baseline={(0,-.3)}, scale = 0.8]
\draw (-.25,0) -- (1.75,0) -- (1.75,-.9) -- (-.25,-.9) -- cycle;
\begin{footnotesize}
\node at (0,.2) {$1$};
\node at (.5,.2) {$2$};
\node at (1,.2) {$3$};
\node at (1.5,.2) {$4$};
\end{footnotesize}
\draw[thick] (0,0) .. controls +(0,-.5) and +(0,-.5) .. +(.5,0);
\draw[thick] (1,0) .. controls +(0,-.5) and +(0,-.5) .. +(.5,0);
\end{tikzpicture}
}

\newcommand{\cupfa}{\begin{tikzpicture}[baseline={(0,-.3)}, scale = 0.8]
\draw (-.25,0) -- (1.75,0) -- (1.75,-.9) -- (-.25,-.9) -- cycle;
\begin{footnotesize}
\node at (0,.2) {$1$};
\node at (.5,.2) {$2$};
\node at (1,.2) {$3$};
\node at (1.5,.2) {$4$};
\end{footnotesize}
\draw[thick] (0,0) .. controls +(0,-.5) and +(0,-.5) .. +(.5,0);
\draw (1,0) -- (1,-.9);
\draw (1.5,0) -- (1.5,-.9);
\end{tikzpicture}
}

\newcommand{\cupfb}{\begin{tikzpicture}[baseline={(0,-.3)}, scale = 0.8]
\draw (-.25,0) -- (1.75,0) -- (1.75,-.9) -- (-.25,-.9) -- cycle;
\begin{footnotesize}
\node at (0,.2) {$1$};
\node at (.5,.2) {$2$};
\node at (1,.2) {$3$};
\node at (1.5,.2) {$4$};
\end{footnotesize}
\draw[thick] (0.5,0) .. controls +(0,-.5) and +(0,-.5) .. +(.5,0);
\draw (0,0) -- (0,-.9);
\draw (1.5,0) -- (1.5,-.9);
\end{tikzpicture}
}

\newcommand{\cupfc}{\begin{tikzpicture}[baseline={(0,-.3)}, scale = 0.8]
\draw (-.25,0) -- (1.75,0) -- (1.75,-.9) -- (-.25,-.9) -- cycle;
\begin{footnotesize}
\node at (0,.2) {$1$};
\node at (.5,.2) {$2$};
\node at (1,.2) {$3$};
\node at (1.5,.2) {$4$};
\end{footnotesize}
\draw[thick] (1,0) .. controls +(0,-.5) and +(0,-.5) .. +(.5,0);
\draw (0,0) -- (0,-.9);
\draw (0.5,0) -- (0.5,-.9);
\end{tikzpicture}
}

\newcommand{\mcup}[2]{
\overset{
\hspace{.7mm}
{}_{#1}
\hspace{.9mm}
{}_{#2}}{
\cup\raisebox{-.5mm}{\hspace{-1.85mm}\scaleobj{.4}{\blacksquare}}
	} 
}

\newcommand{\mray}[1]{
\overset{{}_{#1}}{|\raisebox{.4mm}{\hspace{-1.2mm}\scaleobj{0.4}{\blacksquare}}} 
}
\newcommand{\mcupb}{
\begin{tikzpicture}[baseline={(0,-.3)}, scale = 0.8]
\draw (2.75,0) -- (1.25,0)  -- (1.25,-1.3) --  (2.75,-1.3);
\draw[dotted] (2.75,-1.3) -- (2.75,0);
\begin{footnotesize}
\node at (1.5,.2) {$1$};
\node at (2,.2) {$2$};
\node at (2.5,.2) {$3$};
\end{footnotesize}
\draw[thick] (1.5,0) -- (1.5, -1.3);
\draw[thick] (2,0) .. controls +(0,-.5) and +(0,-.5) .. +(.5,0);
\end{tikzpicture}
}
\newcommand{\fcupdm}{
\begin{tikzpicture}[baseline={(0,-.3)}, scale = 0.8]
\draw (-.25,0) -- (2.75,0) -- (2.75,-1.3) -- (-.25,-1.3) -- cycle;
\begin{footnotesize}
\node at (0,.2) {$0$};
\node at (.5,.2) {$1$};
\node at (1,.2) {$2$};
\node at (1.5,.2) {$3$};
\node at (2,.2) {$4$};
\node at (2.5,.2) {$5$};
\node at (1, -.65) {$\blacksquare$};
\node at (1.5, -.65) {$\blacksquare$};
\end{footnotesize}
\draw[thick] (0,0) .. controls +(0,-.5) and +(0,-.5) .. +(.5,0);
\draw[thick] (1,0) -- (1, -1.3);
\draw[thick] (1.5,0) -- (1.5, -1.3);
\draw[thick] (2,0) .. controls +(0,-.5) and +(0,-.5) .. +(.5,0);
\end{tikzpicture}
}
\newcommand{\mcupa}{
\begin{tikzpicture}[baseline={(0,-.3)}, scale = 0.8]
\draw (2.75,0) -- (1.25,0)  -- (1.25,-1.3) --  (2.75,-1.3);
\draw[dotted] (2.75,-1.3) -- (2.75,0);
\begin{footnotesize}
\node at (1.5,.2) {$1$};
\node at (2,.2) {$2$};
\node at (2.5,.2) {$3$};
\end{footnotesize}
\draw[thick] (2.5,0) -- (2.5, -1.3);
\draw[thick] (1.5,0) .. controls +(0,-.5) and +(0,-.5) .. +(.5,0);
\end{tikzpicture}
}

\newcommand{\fcupfp}{
\begin{tikzpicture}[baseline={(0,-.3)}, scale = 0.8]
\draw (-.25,0) -- (2.75,0) -- (2.75,-1.3) -- (-.25,-1.3) -- cycle;
\begin{footnotesize}
\node at (0,.2) {$1$};
\node at (.5,.2) {$2$};
\node at (1,.2) {$3$};
\node at (1.5,.2) {$4$};
\node at (2,.2) {$5$};
\node at (2.5,.2) {$6$};
\end{footnotesize}
\draw[thick] (0.5,0) .. controls +(0,-.5) and +(0,-.5) .. +(.5,0);
\draw[thick] (0,0) -- (0, -1.3);
\draw[thick] (2.5,0) -- (2.5, -1.3);
\draw[thick] (1.5,0) .. controls +(0,-.5) and +(0,-.5) .. +(.5,0);
\end{tikzpicture}
}

\newcommand{\fcupfm}{
\begin{tikzpicture}[baseline={(0,-.3)}, scale = 0.8]
\draw (-.25,0) -- (2.75,0) -- (2.75,-1.3) -- (-.25,-1.3) -- cycle;
\begin{footnotesize}
\node at (0,.2) {$1$};
\node at (.5,.2) {$2$};
\node at (1,.2) {$3$};
\node at (1.5,.2) {$4$};
\node at (2,.2) {$5$};
\node at (2.5,.2) {$6$};
\node at (0, -.65) {$\blacksquare$};
\node at (2.5, -.65) {$\blacksquare$};
\end{footnotesize}
\draw[thick] (0.5,0) .. controls +(0,-.5) and +(0,-.5) .. +(.5,0);
\draw[thick] (0,0) -- (0, -1.3);
\draw[thick] (2.5,0) -- (2.5, -1.3);
\draw[thick] (1.5,0) .. controls +(0,-.5) and +(0,-.5) .. +(.5,0);
\end{tikzpicture}
}

\newcommand{\mcupc}{
\begin{tikzpicture}[baseline={(0,-.3)}, scale = 0.8]
\draw (2.75,0) -- (1.25,0)  -- (1.25,-1.3) --  (2.75,-1.3);
\draw[dotted] (2.75,-1.3) -- (2.75,0);
\begin{footnotesize}
\node at (1.5,.2) {$1$};
\node at (1.5, -.65) {$\blacksquare$};
\node at (2,.2) {$2$};
\node at (2.25, -.35) {$\blacksquare$};
\node at (2.5,.2) {$3$};
\end{footnotesize}
\draw[thick] (1.5,0) -- (1.5, -1.3);
\draw[thick] (2,0) .. controls +(0,-.5) and +(0,-.5) .. +(.5,0);
\end{tikzpicture}
}
\newcommand{\fcupc}{
\begin{tikzpicture}[baseline={(0,-.3)}, scale = 0.8]
\draw (-.25,0) -- (2.75,0) -- (2.75,-1.3) -- (-.25,-1.3) -- cycle;
\begin{footnotesize}
\node at (0,.2) {$1$};
\node at (.5,.2) {$2$};
\node at (1,.2) {$3$};
\node at (1.5,.2) {$4$};
\node at (2,.2) {$5$};
\node at (2.5,.2) {$6$};
\end{footnotesize}
\draw[thick] (0,0) .. controls +(0,-1.5) and +(0,-1.5) .. +(2.5,0);
\draw[thick] (0.5,0) .. controls +(0,-1) and +(0,-1) .. +(1.5,0);
\draw[thick] (1,0) .. controls +(0,-.5) and +(0,-.5) .. +(.5,0);
\end{tikzpicture}
}
\newcommand{\mcupf}{
\begin{tikzpicture}[baseline={(0,-.3)}, scale = 0.8]
\draw (2.75,0) -- (1.25,0)  -- (1.25,-1.3) --  (2.75,-1.3);
\draw[dotted] (2.75,-1.3) -- (2.75,0);
\begin{footnotesize}
\node at (1.5,.2) {$1$};
\node at (1.5, -.65) {$\blacksquare$};
\node at (2,.2) {$2$};
\node at (2.5,.2) {$3$};
\end{footnotesize}
\draw[thick] (1.5,0) -- (1.5, -1.3);
\draw[thick] (2,0) .. controls +(0,-.5) and +(0,-.5) .. +(.5,0);
\end{tikzpicture}
}
\newcommand{\fcupd}{
\begin{tikzpicture}[baseline={(0,-.3)}, scale = 0.8]
\draw (-.25,0) -- (2.75,0) -- (2.75,-1.3) -- (-.25,-1.3) -- cycle;
\begin{footnotesize}
\node at (0,.2) {$1$};
\node at (.5,.2) {$2$};
\node at (1,.2) {$3$};
\node at (1.5,.2) {$4$};
\node at (2,.2) {$5$};
\node at (2.5,.2) {$6$};
\end{footnotesize}
\draw[thick] (0,0) .. controls +(0,-.5) and +(0,-.5) .. +(.5,0);
\draw[thick] (1,0) .. controls +(0,-.5) and +(0,-.5) .. +(.5,0);
\draw[thick] (2,0) .. controls +(0,-.5) and +(0,-.5) .. +(.5,0);
\end{tikzpicture}
}

\newcommand{\fcupdp}{
\begin{tikzpicture}[baseline={(0,-.3)}, scale = 0.8]
\draw (-.25,0) -- (2.75,0) -- (2.75,-1.3) -- (-.25,-1.3) -- cycle;
\begin{footnotesize}
\node at (0,.2) {$1$};
\node at (.5,.2) {$2$};
\node at (1,.2) {$3$};
\node at (1.5,.2) {$4$};
\node at (2,.2) {$5$};
\node at (2.5,.2) {$6$};
\end{footnotesize}
\draw[thick] (0,0) .. controls +(0,-.5) and +(0,-.5) .. +(.5,0);
\draw[thick] (1,0) -- (1, -1.3);
\draw[thick] (1.5,0) -- (1.5, -1.3);
\draw[thick] (2,0) .. controls +(0,-.5) and +(0,-.5) .. +(.5,0);
\end{tikzpicture}
}

\newcommand{\mcupe}{
\begin{tikzpicture}[baseline={(0,-.3)}, scale = 0.8]
\draw (2.75,0) -- (1.25,0)  -- (1.25,-1.3) --  (2.75,-1.3);
\draw[dotted] (2.75,-1.3) -- (2.75,0);
\begin{footnotesize}
\node at (1.5,.2) {$1$};
\node at (2,.2) {$2$};
\node at (2.25, -.35) {$\blacksquare$};
\node at (2.5,.2) {$3$};
\end{footnotesize}
\draw[thick] (1.5,0) -- (1.5, -1.3);
\draw[thick] (2,0) .. controls +(0,-.5) and +(0,-.5) .. +(.5,0);
\end{tikzpicture}
}
\newcommand{\fcupcm}{
\begin{tikzpicture}[baseline={(0,-.3)}, scale = 0.8]
\draw (-.25,0) -- (2.75,0) -- (2.75,-1.3) -- (-.25,-1.3) -- cycle;
\begin{footnotesize}
\node at (0,.2) {$1$};
\node at (.5,.2) {$2$};
\node at (1,.2) {$3$};
\node at (1.5,.2) {$4$};
\node at (2,.2) {$5$};
\node at (2.5,.2) {$6$};
\node at (.75, -.35) {$\blacksquare$};
\node at (1.75, -.35) {$\blacksquare$};
\end{footnotesize}
\draw[thick] (.5,0) .. controls +(0,-.5) and +(0,-.5) .. +(.5,0);
\draw[thick] (1.5,0) .. controls +(0,-.5) and +(0,-.5) .. +(.5,0);
\draw[thick] (0,0) .. controls +(0,-1.5) and +(0,-1.5) .. +(2.5,0);
\end{tikzpicture}
}

\newcommand{\fcupcmp}{
\begin{tikzpicture}[baseline={(0,-.3)}, scale = 0.8]
\draw (-.25,0) -- (2.75,0) -- (2.75,-1.3) -- (-.25,-1.3) -- cycle;
\begin{footnotesize}
\node at (0,.2) {$1$};
\node at (.5,.2) {$2$};
\node at (1,.2) {$3$};
\node at (1.5,.2) {$4$};
\node at (2,.2) {$5$};
\node at (2.5,.2) {$6$};
\node at (.75, -.35) {$\blacksquare$};
\node at (1.75, -.35) {$\blacksquare$};
\end{footnotesize}
\draw[thick] (.5,0) .. controls +(0,-.5) and +(0,-.5) .. +(.5,0);
\draw[thick] (1.5,0) .. controls +(0,-.5) and +(0,-.5) .. +(.5,0);
\draw[thick] (0,0) -- (0, -1.3);
\draw[thick] (2.5,0) -- (2.5, -1.3);
\end{tikzpicture}
}
\newcommand{\fcupcmm}{
\begin{tikzpicture}[baseline={(0,-.3)}, scale = 0.8]
\draw (-.25,0) -- (2.75,0) -- (2.75,-1.3) -- (-.25,-1.3) -- cycle;
\begin{footnotesize}
\node at (0,.2) {$1$};
\node at (.5,.2) {$2$};
\node at (1,.2) {$3$};
\node at (1.5,.2) {$4$};
\node at (2,.2) {$5$};
\node at (2.5,.2) {$6$};
\node at (.75, -.35) {$\blacksquare$};
\node at (1.75, -.35) {$\blacksquare$};
\node at (0, -.65) {$\blacksquare$};
\node at (2.5, -.65) {$\blacksquare$};
\end{footnotesize}
\draw[thick] (.5,0) .. controls +(0,-.5) and +(0,-.5) .. +(.5,0);
\draw[thick] (1.5,0) .. controls +(0,-.5) and +(0,-.5) .. +(.5,0);
\draw[thick] (0,0) -- (0, -1.3);
\draw[thick] (2.5,0) -- (2.5, -1.3);
\end{tikzpicture}
}

\newcommand{\mcupd}{
\begin{tikzpicture}[baseline={(0,-.3)}, scale = 0.8]
\draw (2.75,0) -- (1.25,0)  -- (1.25,-1.3) --  (2.75,-1.3);
\draw[dotted] (2.75,-1.3) -- (2.75,0);
\begin{footnotesize}
\node at (1.5,.2) {$1$};
\node at (2.5, -.65) {$\blacksquare$};
\node at (2,.2) {$2$};
\node at (2.5,.2) {$3$};
\end{footnotesize}
\draw[thick] (2.5,0) -- (2.5, -1.3);
\draw[thick] (1.5,0) .. controls +(0,-.5) and +(0,-.5) .. +(.5,0);
\end{tikzpicture}
}
\newcommand{\fcupf}{
\begin{tikzpicture}[baseline={(0,-.3)}, scale = 0.8]
\draw (-.25,0) -- (2.75,0) -- (2.75,-1.3) -- (-.25,-1.3) -- cycle;
\begin{footnotesize}
\node at (0,.2) {$1$};
\node at (.5,.2) {$2$};
\node at (1,.2) {$3$};
\node at (1.5,.2) {$4$};
\node at (2,.2) {$5$};
\node at (2.5,.2) {$6$};
\end{footnotesize}
\draw[thick] (.5,0) .. controls +(0,-.5) and +(0,-.5) .. +(.5,0);
\draw[thick] (1.5,0) .. controls +(0,-.5) and +(0,-.5) .. +(.5,0);
\draw[thick] (0,0) .. controls +(0,-1.5) and +(0,-1.5) .. +(2.5,0);
\end{tikzpicture}
}

\newcommand{\mcupaa}{
\begin{tikzpicture}[baseline={(0,-.3)}, scale = 0.8]
\draw (1.75,0) -- (0.75,0) -- (0.75,-.9) -- (1.75,-.9);
\draw[dotted] (1.75,-.9) -- (1.75,0);\begin{footnotesize}
\node at (1.5,.2) {$2$};
\node at (1,.2) {$1$};
\node at (1.25, -.35) {$\blacksquare$};
\end{footnotesize}
\draw[thick] (1,0) .. controls +(0,-.5) and +(0,-.5) .. +(.5,0);
\end{tikzpicture}
}

\newcommand{\mcupab}{
\begin{tikzpicture}[baseline={(0,-.3)}, scale = 0.8]
\draw (1.75,0) -- (0.75,0) -- (0.75,-.9) -- (1.75,-.9);
\draw[dotted] (1.75,-.9) -- (1.75,0);\begin{footnotesize}
\node at (1.5,.2) {$2$};
\node at (1,.2) {$1$};
\end{footnotesize}
\draw[thick] (1,0) .. controls +(0,-.5) and +(0,-.5) .. +(.5,0);
\end{tikzpicture}
}

\newcommand{\cupaaa}{
\begin{tikzpicture}[baseline={(0,-.3)}, scale = 0.8]
\draw (1.75,0) -- (0.75,0) -- (0.75,-.6) -- (1.75,-.6);
\draw (1.75,-.6) -- (1.75,0);\begin{footnotesize}
\node at (1.5,.2) {$2$};
\node at (1,.2) {$1$};
\end{footnotesize}
\draw[thick] (1,0) .. controls +(0,-.5) and +(0,-.5) .. +(.5,0);
\end{tikzpicture}
}

\newcommand{\mcupaaa}{
\begin{tikzpicture}[baseline={(0,-.3)}, scale = 0.8]
\draw (1.25,0) -- (0.75,0) -- (0.75,-.6) -- (1.25,-.6);
\draw[dotted] (1.25,-.6) -- (1.25,0);\begin{footnotesize}
\node at (1,.2) {$1$};
\end{footnotesize}
\draw[thick] (1,0) -- (1, -.6);
\end{tikzpicture}
}

\newcommand{\mcupaab}{
\begin{tikzpicture}[baseline={(0,-.3)}, scale = 0.8]
\draw (1.25,0) -- (0.75,0) -- (0.75,-.6) -- (1.25,-.6);
\draw[dotted] (1.25,-.6) -- (1.25,0);\begin{footnotesize}
\node at (1,.2) {$1$};
\node at (1, -.3) {$\blacksquare$};
\end{footnotesize}
\draw[thick] (1,0) -- (1, -.6);
\end{tikzpicture}
}

\newcommand{\icupa}{
\begin{tikzpicture}[baseline={(0,-.3)}, scale = 0.8]
\draw (2.25,0) -- (0.75,0) -- (0.75,-.6) -- (2.25,-.6);
\draw[dotted] (2.25,-.6) -- (2.25,0);\begin{footnotesize}
\node at (1,.2) {$1$};
\node at (1.75, -.3) {$a_2$};
\end{footnotesize}
\draw[thick] (1,0) -- (1, -.6);
\end{tikzpicture}
}

\newcommand{\icupb}{
\begin{tikzpicture}[baseline={(0,-.3)}, scale = 0.8]
\draw (2.25,0) -- (0.75,0) -- (0.75,-.6) -- (2.25,-.6);
\draw[dotted] (2.25,-.6) -- (2.25,0);\begin{footnotesize}
\node at (1, -.3) {$\blacksquare$};
\node at (1,.2) {$1$};
\node at (1.75, -.3) {$a_2$};
\end{footnotesize}
\draw[thick] (1,0) -- (1, -.6);
\end{tikzpicture}
}

\newcommand{\icupc}{
\begin{tikzpicture}[baseline={(0,-.3)}, scale = 0.8]
\draw (3.25,0) -- (0.75,0) -- (0.75,-.9) -- (3.25,-.9);
\draw[dotted] (3.25,-.9) -- (3.25,0);\begin{footnotesize}
\node at (2.5,.2) {$2t$};
\node at (1,.2) {$1$};
\node at (1.75, -.7) {$\blacksquare$};
\node at (1.75, -.3) {$a_1$};
\node at (2.8, -.3) {$a_2$};
\end{footnotesize}
\draw[thick] (1,0) .. controls +(0,-1) and +(0,-1) .. +(1.5,0);
\end{tikzpicture}
}

\newcommand{\icupd}{
\begin{tikzpicture}[baseline={(0,-.3)}, scale = 0.8]
\draw (3.25,0) -- (0.75,0) -- (0.75,-.9) -- (3.25,-.9);
\draw[dotted] (3.25,-.9) -- (3.25,0);\begin{footnotesize}
\node at (2.5,.2) {$2t$};
\node at (1,.2) {$1$};
\node at (1.75, -.3) {$a_1$};
\node at (2.8, -.3) {$a_2$};
\end{footnotesize}
\draw[thick] (1,0) .. controls +(0,-1) and +(0,-1) .. +(1.5,0);
\end{tikzpicture}
}

\newenvironment{eq}{\begin{equation}}{\end{equation}}

\title{Irreducible components of two-row Springer fibers and Nakajima quiver varieties}
\author[M.S. Im, C. Lai, and A. Wilbert]{Mee Seong Im,  Chun-Ju Lai, and Arik Wilbert}
\address{Department of Mathematical Sciences, United States Military Academy, West Point, NY 10996}
    \email{meeseongim@gmail.com (Im)}
\address{Department of Mathematics, University of Georgia, Athens, GA 30602}
    \email{cjlai@uga.edu (Lai)}
\address{Department of Mathematics, University of Georgia, Athens, GA 30602}
\email{arik.wilbert@uga.edu (Wilbert)}

\begin{document}

\begin{abstract}
We give an explicit description of the irreducible components of two-row Springer fibers in type A as closed subvarieties in certain Nakajima quiver varieties in terms of quiver representations. 
By taking invariants under a variety automorphism, we obtain an explicit algebraic description of the irreducible components of two-row Springer fibers of classical type.
As a consequence, we discover relations on isotropic flags that describe the irreducible components.

\end{abstract}
 
\maketitle 

\section{Introduction}
\subsection{Background and summary} \label{sec:background}
Quiver varieties were used by 
Nakajima in~\cite{Na94, Nak98} to provide a geometric construction of the universal enveloping algebra for symmetrizable Kac-Moody Lie algebras altogether with their integrable highest weight modules.  
It was shown by Nakajima that the cotangent bundle of partial flag varieties can be realized as a quiver variety.
He also conjectured that the Slodowy varieties, i.e., resolutions of slices to the adjoint orbits in the nilpotent cone, can be realized as quiver varieties. 
This conjecture was proved by Maffei,~\cite[Theorem~8]{Maf05}, 
thereby establishing a precise connection between quiver varieties and flag varieties. 

Aside from quiver varieties, an important geometric object in our article is the Springer fiber, which plays a crucial role in the geometric construction of representations of Weyl groups (cf. \cite{Spr76, Spr78}). 
In general, Springer fibers are singular and decompose into many irreducible components.
  
The first goal of this article is to study the irreducible components of two-row Springer fibers in type A from the point of view of quiver varieties using Maffei's isomorphism. More precisely, since Springer fibers are naturally contained in the Slodowy variety, it makes sense to describe the image of the Springer fiber as a subvariety of the quiver variety. 
We  describe  quiver representations which represent points of the entire Slodowy variety containing any given two-row Springer fiber  under Maffei's isomorphism.
As an application, we achieve our goal (see Theorem~\ref{thm:main1}). Our proof relies on an explicit description of the irreducible components in terms of flags obtained by Stroppel--Webster, \cite{SW12}, based on an earlier work by Fung, \cite{Fun03}. 

Moreover, Maffei's isomorphism (which in general is give by an implicit solution of a system of equations in terms of matrices) actually becomes explicit (cf. Lemma~\ref{lem:hell}) on the Slodowy variety, which allows for a translation of the results of Stroppel--Webster to the quiver variety, see Proposition~\ref{prop:hellA}. 
It would be interesting to investigate if the relations on the quiver variety of the irreducible components can be generalized to nilpotent endomorphisms with more than two Jordan blocks. For these nilpotent endomorphisms, obtaining an understanding of the precise geometric and combinatorial structure of the irreducible components remains an open problem.

The second goal\footnote{We would like to thank Dongkwan Kim for pointing out that the second goal of this paper can be obtained more efficiently without using the fixed-point subvarieties. This manuscript (not intended for publication in a journal) has therefore been split  into the two preprints arXiv:2009.08778 and arXiv:2011.13138.} of this article is to generalize the results in type A to all classical types. Our focus will be on two-row Springer fibers of type D. In fact, any two-row Springer fiber of type C is isomorphic to a two-row Springer fiber of type D,~\cite{Wil18, Li19}. Moreover, by the classification of nilpotent orbits, \cite{Wi37, Ger61}, there are no two-row Springer fibers of type B. Hence, it is enough to treat the type D case. Work of Henderson--Licata~\cite{HL14} and Li~\cite{Li19} shows that the type D Slodowy variety can be realized as a fixed-point subvariety of a type A quiver variety under a suitable variety automorphism. These fixed-point subvarieties play an important role in developing the geometric representation theory of symmetric pairs, \cite{Li19}. 

Our second main result (see Theorems~\ref{thm:mainDQ} and \ref{thm:main3}) is to explicitly compute the fixed points contained in the subvarieties of the quiver variety corresponding to the irreducible components in type A. By sending the result through Maffei's isomorphism we obtain explicit algebraic relations that describe the irreducible components of the type D two-row Springer fiber in terms of isotropic flags. 
Finally, we show that these subvarieties are indeed irreducible components using an iterated $\CP^1$-bundle argument.
This generalizes the results by Stroppel--Webster to all other types. 
Before our manuscript, for two-row Springer fibers of type D, 
the only explicit algebro-geometric construction of components required an intricate inductive procedure, see \cite[\S 6]{ES16}, based on \cite{Spa82, vL89}.
Other than that, only topological models were available, \cite{ES16, Wil18}.

There have been an extensive study of Lagrangian subvarieties in quiver varieties (see e.g., \cite{Lu91, Lu98, Na94, Nak98, Sai02, KS19}). That is, such subvarieties index crystal bases of certain irreducible representations of simple Lie algebras using a conormal bundle approach or torus fixed-point approach. 
In particular, they can also be identified with Springer fibers via Maffei's isomorphism. Our results however make preceding constructions explicit. 


\subsection{Type A: Two-row Springer fibers and quiver varieties}
Let $\mu: \tcN \to \cN$ be the Springer resolution for nilpotent cone $\cN$ of $\fgl_n(\CC)$. For any element $x \in \cN$, one can associate a Slodowy slice $\cS_x$, a Slodowy variety $\tcS_x$, and a Springer fiber $\cB_x$ with the relations below:
\[
\xymatrix@-1.25pc{
\mathcal{B}_x = \mu^{-1}(x) \ar[dd] \ar@{^{(}->}[rr]
& & \widetilde{\mathcal{S}}_x =\mu^{-1}(\mathcal{S}_x)  \ar[dd]^{\mu|_{\widetilde{\mathcal{S}}_x}} \ar@{^{(}->}[rr] & &\widetilde{\mathcal{N}} \ar[dd]^{\mu.}
\\
& & & & \\ 
\{ x\}\ar@{^{(}->}[rr] &  & \mathcal{S}_x \ar@{^{(}->}[rr]&  &\mathcal{N} 
}
\] 
It was conjectured by Nakajima in \cite{Na94} and then proved by Maffei \cite[Theorem~8]{Maf05} that there is an isomorphism $\tphi: M(d,v) \to \tcS_x$ between the Slodowy variety and a certain Nakajima quiver variety $M(d,v)$, which is realized by a geometric invariant theory (GIT) quotient of a certain space of (type A) quiver representations $\Lambda^+(d,v)$, which consists of collections $((A_i)_i, (B_i)_i, (\Gamma_i)_i, (\Delta_i)_i)$ of linear maps. In particular, it can be realized via certain stability and admissibility conditions. 

For a nilpotent endomorphism $x$ of Jordan type $(n-k,k)$, it is well-known (cf. \cite[\S~7]{Fun03}) that the irreducible components $\{K^a\}_a$ of  the Springer fiber $\cB_x$ are parametrized by the so-called cup diagrams.
The purpose of this article is to give an 
explicit description (of the irreducible components) of the Springer fibers via the embedding into the corresponding Nakajima quiver varieties. 
Their relations are depicted as below:
\[
\begin{tikzcd}
\Lambda^+(d,v) \ar[r,"p"]& M(d,v) \ar[r,"\tphi", "\simeq"'] & \tcS_x
\\
& \tphi\inv(\cB_x) \ar[r, "\simeq"'] \ar[u, hookrightarrow]& \cB_x. \ar[u, hookrightarrow]
\\
& \tphi\inv(K^a) \ar[r, "\simeq"'] \ar[u, hookrightarrow]& K^a \ar[u, hookrightarrow]
&
\end{tikzcd}
\]
Each irreducible component $K^a$ consists of pairs $(x, F_\bullet)$ for some complete flag $F_\bullet = (0 \subsetneq F_1 \subsetneq \ldots \subsetneq F_n = \CC^n)$ subject to certain relations imposed from the configurations of cups and rays in a cup diagram.

Our first result is to give explicit and new relations on the quiver representation side that correspond to the cup/ray relations in \cite[Proposition~7]{SW12} as below:
\[
\begin{split}
\textup{Cup relation for } \overset{{}_i\hspace{.9mm}{}_j}{\cup} 
&
\quad
\Leftrightarrow 
\quad
F_j = 
x^{-\frac{1}{2}(j-i+1)} F_{i-1}
\\
&
\quad
\Leftrightarrow 
\quad
\ker B_{i-1} B_{i} \cdots B_{\frac{j+i-3}{2}}=\ker A_{j-1} A_{j-2} \cdots A_{\frac{j+i-1}{2}},
\\
\textup{Ray relation for } \overset{{}_i}{|} 
&
\quad
\Leftrightarrow 
\quad
F_i = x^{-\frac{1}{2}(i-\rho(i))}(x^{n-k-\rho(i)} F_n ) 
\\
&
\quad
\Leftrightarrow 
\quad
\begin{cases}
\:\: B_iA_i = 0
&\tif c(i) \geq 1, 
\\
B_{i} B_{i+1} \cdots B_{n-k-1}\Gamma_{n-k} = 0
&\tif c(i) = 0,
\end{cases}
\end{split}
\]
where $\rho(i) \in \ZZ_{>0}$ counts the number of rays (including itself) to the left of $i$ , and $c(i) = \frac{i-\rho(i)}{2}$ is the total number of cups to the left of $i$. 
In other words, we have constructed a subvariety $M^a \subset M(d,v)$ 
inside the quiver variety using the above relations, whose points correspond to exactly one irreducible component of the Springer fiber, and we prove: 
\begin{thmA}[Theorem~\ref{thm:main1}]
For any cup diagram $a$, the Maffei--Nakajima isomorphism $\tphi:M(d,v)\to \tcS_x$ between quiver variety and Slodowy variety restricts to an isomorphism
$M^a \to K^a$ between irreducible components.

\end{thmA}
\subsection{Springer fibers of classical type}
For any type $\Phi = {}$B, C or D, let $\mu_\Phi: \tcN^{\Phi} \to \cN^\Phi$ be the Springer resolution for the nilpotent cone $\cN^\Phi$ of the Lie algebra $\fg^\Phi$ of type $\Phi$. 
For each $x \in \cN^\Phi \subset \cN$, one associates a Slodowy slice $\cS^\Phi_x$, a Slodowy variety $\tcS^\Phi_x$, and a Springer fiber $\cB^\Phi_x$, which are related as follows:
\[ 
\xymatrix@-1.25pc{
\cB^\Phi_x = \mu_\Phi\inv(x) \ar@{^{(}->}[rr]  \ar[dd] & &  
\tcS^\Phi_x = \mu_\Phi\inv(\cS^\Phi_x)  \ar@{^{(}->}[rr] \ar[dd] & &  
\tcN^\Phi  \ar[dd]^{\mu_{\Phi}.}\\ 
& &    & &  \\ 
\{ x\} \ar@{^{(}->}[rr] & & \cS^\Phi_x \ar@{^{(}->}[rr] & & 
\cN^\Phi \\ 
}
\] 
In the following, we would like to study the components of Springer fibers $\cB^\Phi_{x}$ associated with a nilpotent endomorphism $x \in \cN^\Phi$ of Jordan type $(n-k,k)$, using a generalized cup diagram approach.

As mentioned in Section~\ref{sec:background}, it suffices to study the type D Springer fibers of two Jordan blocks.
The type D cup diagrams are similar to the type A ones but with two new ingredients -- marked cups and marked rays, which arise naturally in the process of folding a centro-symmetric type A cup diagram (see \cite{ES15, LS13}).
For example,
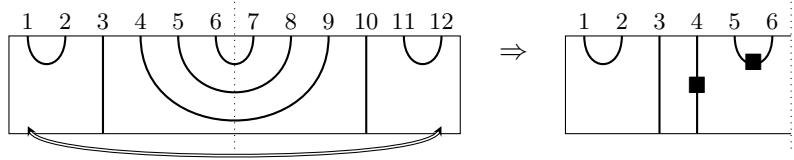
\[
\begin{tikzpicture}[baseline={(0,-.3)}]
\draw (-.25,0) -- (5.75,0) -- (5.75,-1.3) -- (-.25,-1.3) -- cycle;
\draw[dotted] (2.75,-1.5) -- (2.75,.5);
\draw[>=stealth, <->, double] (0,-1.2) .. controls +(.25,-.5) and +(-.25,-.5) .. (5.5,-1.2);
\begin{footnotesize}
\node at (0,.2) {$1$};
\node at (.5,.2) {$2$};
\node at (1,.2) {$3$};
\node at (1.5,.2) {$4$};
\node at (2,.2) {$5$};
\node at (2.5,.2) {$6$};
\node at (3,.2) {$7$};
\node at (3.5,.2) {$8$};
\node at (4,.2) {$9$};
\node at (4.5,.2) {$10$};
\node at (5,.2) {$11$};
\node at (5.5,.2) {$12$};
\end{footnotesize}
\draw[thick] (1,0) -- (1, -1.3);
\draw[thick] (4.5,0) -- (4.5, -1.3);
\draw[thick] (0,0) .. controls +(0,-.5) and +(0,-.5) .. +(.5,0);
\draw[thick] (1.5,0) .. controls +(0,-1.5) and +(0,-1.5) .. +(2.5,0);
\draw[thick] (2,0) .. controls +(0,-1) and +(0,-1) .. +(1.5,0);
\draw[thick] (2.5,0) .. controls +(0,-.5) and +(0,-.5) .. +(.5,0);
\draw[thick] (5,0) .. controls +(0,-.5) and +(0,-.5) .. +(.5,0);
\end{tikzpicture}
\quad
\Rightarrow
\quad
\begin{tikzpicture}[baseline={(0,-.3)}]
\draw (5.75,0) -- (2.75,0) -- (2.75,-1.3) -- (5.75,-1.3);
\draw[dotted] (5.75,-1.5) -- (5.75,.5);
\begin{footnotesize}
\node at (3,.2) {$1$};
\node at (5.25, -.35) {$\blacksquare$};
\node at (3.5,.2) {$2$};
\node at (4,.2) {$3$};
\node at (4.5, -.65) {$\blacksquare$};
\node at (4.5,.2) {$4$};
\node at (5,.2) {$5$};
\node at (5.5,.2) {$6$};
\end{footnotesize}
\draw[thick] (4.5,0) -- (4.5, -1.3);
\draw[thick] (4,0) -- (4, -1.3);
\draw[thick] (3,0) .. controls +(0,-.5) and +(0,-.5) .. +(.5,0);
\draw[thick] (5,0) .. controls +(0,-.5) and +(0,-.5) .. +(.5,0);
\end{tikzpicture}
\]
Here the cups and rays may carry a mark (i.e., $\blacksquare$) subject to some conditions.
Unlike in type A, no explicit relations describing the flags in a given component were known for marked diagrams.
In this paper, we discover explicit relations on the flags described by the cups and rays for a marked diagram. 
In order to achieve this, we study the process of taking fixed-points under a certain automorphism on the quiver varieties, which naturally corresponds to a folding of a type A cup diagram.

\subsection{Involutive quiver varieties}
For a Nakajima quiver variety $M(d,v)$ of type A, an automorphism $\theta$ on the type A Dynkin diagram induces a 
variety automorphism $\Theta$ which also depends on some non-degenerate bilinear form.
These fixed-point subvarieties $M(d,v)^\Theta$ appeared in \cite{HL14}, and were generalized in \cite{Li19} (see \cite[Section 4]{Li19} for a more general construction of such subvarieties).
It is shown in \cite{HL14} that under the Maffei-Nakajima isomorphism, $M(d,v)^\Theta$ encodes the type D Slodowy varieties.

We can now apply this folding technique to our subvariety $M^a  \subset  M(d,v)$ and then obtain a fixed-point subvariety $(M^a)^\Theta \subset M(d,v)^\Theta$ associated with a type D cup diagram.

\begin{thmB}[Theorem~\ref{thm:mainDQ}]
For any type A cup diagram $a$ that is centro-symmetric, the Maffei-Nakajima isomorphism $\tphi$ 
restricts to an isomorphism $ (M^a)^\Theta \to \sqcup_{\da} K^{\da}$,
where $\da$ runs over all type D cup diagrams which unfold to $a$ in the sense of Definition~\ref{def:unfold}.
Moreover, 
\[
M(d,v)^\Theta \simeq \bigcup_{\da} K^{\da}.
\]
\end{thmB}

As an application, we obtain the relations on the isotopic flags imposed by marked cups and marked rays:
\[
\begin{split}
\textup{Marked cup relation for } 
\mcup{i}{j}
&
\quad
\Leftrightarrow 
\quad
x^{\lfloor\frac{i}{2}\rfloor}F_i = x^{\lfloor\frac{j+1}{2}\rfloor}F_j,
\\
\textup{Marked ray relation for } 
\mray{i}
&
\quad
\Rightarrow 
\quad 
F_i=
\begin{cases}
\qquad \quad \langle e_1,\ldots,e_{\frac{1}{2}(i+1)},f_1,\ldots,f_{\frac{1}{2}(i-1)}\rangle
&\tif n=2k, 
\\
\langle e_1,\ldots,e_{i-c(i)-1},f_1,\ldots,f_{c(i)}, {f_{c(i)+1}+e_{i-c(i)}}\rangle
&\tif n > 2k.
\end{cases}
\end{split}
\]

\subsection*{Acknowledgment}
M.S.I. would like to thank the University of Georgia for organizing the Southeast Lie Theory Workshop X, where our collaborative research group initially came together. 
The authors also thank the AGANT\footnote{Algebraic Geometry, Algebra, and Number Theory.} group at the University of Georgia for supporting this project. 
This project was initiated as a part of the Summer Collaborators Program 2019 at the School of Mathematics at the Institute for Advanced Study (IAS). The authors thank the IAS for providing an excellent working environment.
We are grateful for useful discussions with Jieru Zhu during the early stages of this project. 
We thank Anthony Henderson, Yiqiang Li, and Hiraku Nakajima for helpful clarifications, and Catharina Stroppel for helpful comments on earlier drafts of our manuscript. 
A.W. was also partially supported by the ARC Discover Grant DP160104912 ``Subtle Symmetries and the Refined Monster''.

\section{Springer fibers and Nakajima quiver varieties of type A}
\subsection{Springer fibers and Slodowy varieties} 
Fix integers $n, k$ such that $n \geq 1$ and $n-k \geq k \geq 0$.
Let $\mathcal{N}$ be the variety of nilpotent elements in $\mathfrak{gl}_n(\CC)$. Let $G=GL_n(\CC)$. 
One may parametrize the $G$-orbits in $\mathcal{N}$ using partitions of $n$ by associating a nilpotent endomorphism to the list of the dimensions of the Jordan blocks. 

Denote the complete flag variety in $\CC^n$ by
\eq
\cB = \{ F_\bullet = (0 = F_0 \subset F_1 \subset \ldots \subset F_n = \CC^n)\mid\dim F_i = i \tforall 1\leq i\leq n\}.
\endeq
We denote the Springer resolution by 
$\mu: \tcN \to \cN, (u, F_\bullet) \mapsto u$,
where
\eq
\tcN = T^*\cB \cong \{(u,F_\bullet)\in \cN \times \cB \mid u (F_i) \subseteq F_{i-1} \tforall i \}.
\endeq
We denote the Springer fiber of $x\in \cN$ by
\eq
\cB_x = \mu\inv(x).
\endeq
For each nilpotent element $x\in \mathcal{N}$,
we denote the
Slodowy transversal slice 
of (the $G$-orbit of) 
$x$ by the variety 
\eq
\mathcal{S}_x =\{ u\in \mathcal{N}\mid [u-x,y]=0\}, 
\endeq
where $(x,y,h)$ is a $\fsl_2$-triple in $\cN$ (here $(0,0,0)$ is also considered an $\mathfrak{sl}_2$-triple).  
We then denote the Slodowy variety associated to $x$ by
\eq
\tcS_{x}=\mu^{-1}(\mathcal{S}_{x}) = \{ (u, F_\bullet) \in \cN\times \cB \mid [u-x,y]=0,  \:\:   u (F_i) \subseteq F_{i-1} \textup{ for all }i \}.
\endeq  
In particular, $x \in \cS_x$ and hence $\cB_x = \mu\inv(x) \subset \mu\inv(\cS_x) = \tcS_x$.

From now on, let $n, k$ be non-negative integers such that $n-k \geq k$.
If $x$ is of Jordan type $(n-k,k)$, we write  
\eq
\cB_{n-k,k}:= \cB_{x},
\quad
\cS_{n-k,k} := \cS_{x},
\quad
\tcS_{n-k,k} := \tcS_{x}.
\endeq
\subsection{Irreducible components of Springer fibers of two-rows}
A {\em cup diagram} is a non-intersecting arrangement of cups and rays below a horizontal axis connecting a subset of $n$ vertices on the axis, and we identify any cup diagram $a$ with the collection of sets of endpoints of cups as below:
\eq\label{eq:cup}
a \equiv \{ \{i_t, j_t\} \subset \{1, \ldots, n\}~|~ 1\leq t \leq k\} 
\quad
\textup{for some}
\quad
k \leq \left\lfloor\frac{n}{2}\right\rfloor.
\endeq
We use a rectangular region that contains all the cups to represent the cup diagram. 
Note that a ray is now a through-strand in this presentation but we still call it a ray.

The irreducible components of the Springer fiber $\mathcal B_{n-k,k}$  can be labeled by the set $B_{n-k,k}$ of all cup diagram on $n$ vertices with $k$ cups and $n-2k$ rays.
For example, when $n=3$ we have
\eq
B_{3,0} = \left\{ \cupc \right\},
\quad
B_{2,1} = \left\{  \cupa~,~ \cupb \right\}.
\endeq
We denote by $K^a$ the irreducible component in $\mathcal B_{n-k,k}$ associated to the cup diagram $a\in B_{n-k,k}$.
\begin{proposition}\label{prop:known_results_about_irred_comp}
Let $x \in \cN$ with Jordan type $(n-k,k)$.
We fix a basis $\{e_i, f_j\mid 1\leq i \leq n-k, 1\leq j \leq k\}$ of $\CC^n$ such that on which $x$ acts by
\[ 
f_k \mapsto f_{k-1} \mapsto \ldots \mapsto f_1 \mapsto 0,
\quad
e_{n-k} \mapsto e_{n-k-1} \mapsto \ldots \mapsto e_1 \mapsto 0.
\]
\begin{enumerate}[(a)]
\item There exists a bijection between the irreducible components of the Springer fiber $\cB_{n-k,k}$ and the set $B_{n-k,k}$ of cup diagrams on $n$ vertices with $k$ cups.
\item The irreducible component $K^a\subseteq \mathcal B_{n-k,k}$  consists of the pairs $(x, F_\bullet) \in\mathcal B_{n-k,k}$ such that
\[
\begin{split}
&\eqref{eq:cup_rel}\textup{ holds for all }(i,j) = (i_t,j_t), 1\leq t \leq k,
\\
&\eqref{eq:ray_rel}\textup{ holds for any } i \not\in \{i_t, j_t ~|~ 1\leq t \leq k\}.
\end{split}
\]
Here the cup relation is given by
\eq\label{eq:cup_rel}
F_j=x^{-\frac{1}{2}(j-i+1)}F_{i-1},
\endeq
where $x^{-1}$ denotes the preimage of a space under the endomorphism $x$;
while the ray relation is given by  
\eq\label{eq:ray_rel}
F_i=F_{i-1} \oplus \langle e_{\frac{1}{2}(i+\rho(i))}\rangle,
\endeq 
where
 $\rho(i)\in\mathbb Z_{>0}$ be the number of rays to the left of $i$ (including itself).
\end{enumerate}
\end{proposition}
\begin{proof}
See \cite{Spa76} and \cite{Var79} for part (a).
For part (b), we refer to \cite[Proposition~7]{SW12} (see also \cite[Theorem~5.2]{Fun03}).
\end{proof}
\rmk
The ray relation \eqref{eq:ray_rel} is equivalent to that
$F_i = x^{-\frac{1}{2}(i-\rho(i))}(x^{n-k-\rho(i)} F_n )$.
\endrmk

\subsection{Quiver representations}
Now we follow \cite{Maf05} to realize Nakajima's quiver variety 
as equivalence classes of semistable orbits of a certain quiver representation space, 
which is equivalent to the usual proj construction in GIT. 
Let $S(d,v)$ be the quiver representation space of type $A$ (see \cite[Defn.~5]{Maf05}) with respect to dimension vectors $d = (\dim D_i)_i, v = (\dim V_i)_i$.
In this article we do not need the most general $S(d,v)$, and hence we will be focusing on certain special cases which we will elaborate below.
 
For $k < n-k$, let $S_{n-k,k}$ be the quiver representation space in Figure~\ref{figure:Sn-kk}.
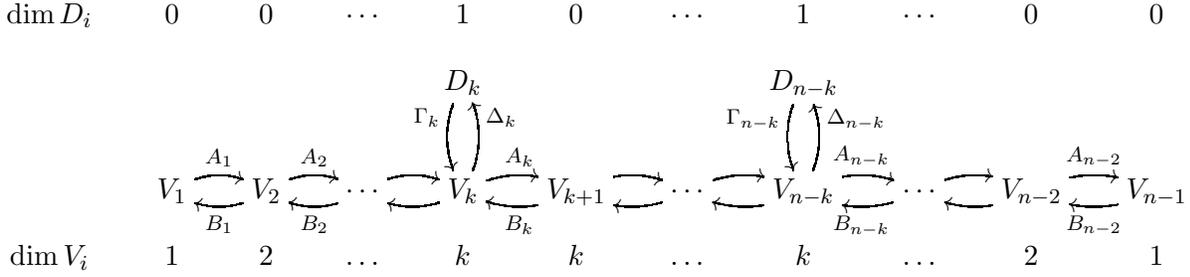
\begin{figure}[ht!]
\caption{Quiver representations in $S_{n-k,k}$ and dimension vectors.}
 \label{figure:Sn-kk}
\[ 
\xymatrix@C=18pt@R=9pt{
\dim D_i&0&0&\cdots &1&0&\cdots &1&\cdots &0&0
\\
&
& &  &   
\ar@/_/[dd]_<(0.2){\Gamma_k}
D_k
& &  &  
\ar@/_/[dd]_<(0.2){\Gamma_{n-k}}
D_{n-k}
 & & &&
\\ & & & & & & &  && 
\\ 
&
{V_1}  
	\ar@/^/[r]^{A_1} 
	& 
	\ar@/^/[l]^{B_1}
{V_2}
	\ar@/^/[r]^{A_2}
	&
	\ar@/^/[l]^{B_2}
	\cdots   
	\ar@/^/[r]
	& 
	\ar@/^/[l]
{V_k}  
	\ar@/_/[uu]_>(0.8){\Delta_k} 
	\ar@/^/[r]^{A_{k}}
	&   
{V_{k+1}}  
	\ar@/^/[l]^{B_k}
	\ar@/^/[r]
	&
	\ar@/^/[l]
	\cdots
	\ar@/^/[r]
	&  
{V_{n-k}} 
	\ar@/_/[uu]_>(0.8){\Delta_{n-k}}
	\ar@/^/[r]^{A_{n-k}}
	\ar@/^/[l]
	& 
	\ar@/^/[l]^{B_{n-k}}
	\cdots
	\ar@/^/[r]
	&
	\ar@/^/[l]
{V_{n-2}}  
	\ar@/^/[r]^{A_{n-2}}
	& 
{V_{n-1}} 
\ar@/^/[l]^{B_{n-2}}  
\\
\dim V_i & 1&2&\ldots&k&k&\ldots&k&\ldots&2&1
}
\] 
\end{figure}
\rmk\label{rmk:0}
Throughout this article, any space with an ineligible subscript is understood as a zero space (e.g., $V_0 = \{0\}, V_{n} =\{0\}$). Any linear map with an ineligible subscript is understood as a zero map (e.g., $A_{n-1}:V_{n-1}\to \{0\}, B_0:V_1 \to \{0\}$).
\endrmk

In other words, $S_{n-k,k}$ can be identified as the space of quadruples $(A,B,\Gamma, \Delta)$ of collections of linear maps  of the following form:
\[
\left(A = (A_i:V_i\to V_{i+1})_{i=1}^{n-2}, B= (B_i:V_{i+1}\to V_i)_{i=1}^{n-2}, \Gamma = (\Gamma_i:D_i\to V_i)_{i=1}^{n-1}, \Delta = (\Delta_i:V_i\to D_i)_{i=1}^{n-1}\right),
\] 
with dimension vectors $d = (\dim D_i)_i$ and $v = (\dim V_i)_i$ given by
\[
\dim V_i = \begin{cases}
\quad i &\tif i \leq k, \\
\quad k &\tif k\leq i \leq n-k, \\
n-i &\tif i \geq n-k,
\end{cases}
\qquad
\dim D_i = \begin{cases}
1 &\tif i=k, n-k, \\
0 &\textup{otherwise}.
\end{cases}  
\]
For the special case $S_{k,k}$ when $n =2k$, we use instead the quiver representations of the form in Figure~\ref{figure:Skk} below.

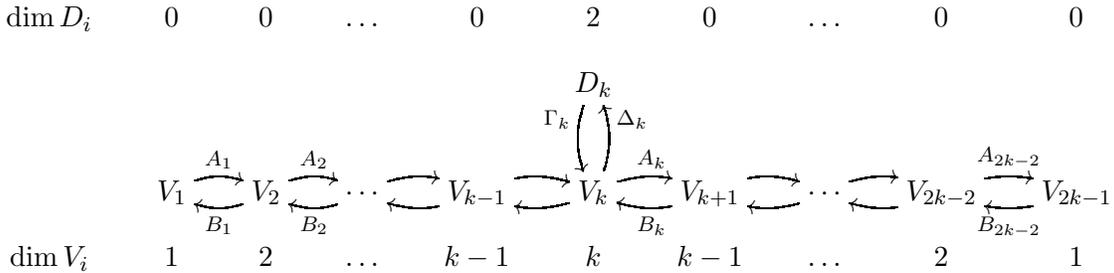
\begin{figure}[ht!]
\caption{Quiver representations in $S_{k,k}$ and the dimension vectors.}
 \label{figure:Skk}
\[ 
\xymatrix@C=18pt@R=9pt{
\dim D_i& 0&0&\ldots&0&2&0&\ldots&0&0
\\
&
& &  &   &
\ar@/_/[dd]_<(0.2){\Gamma_k}
D_k
& &  &  &
\\ & & & & & & &  && 
\\ 
&
{V_1}  
	\ar@/^/[r]^{A_1} 
	& 
	\ar@/^/[l]^{B_1}
{V_2}
	\ar@/^/[r]^{A_2}
	&
	\ar@/^/[l]^{B_2}
	\cdots
	\ar@/^/[r]
	&   
	\ar@/^/[l]
{V_{k-1}} 
	\ar@/^/[r]
	& 
	\ar@/^/[l]
{V_k}  
	\ar@/_/[uu]_>(0.8){\Delta_k} 
	\ar@/^/[r]^{A_{k}}
	&   
	\ar@/^/[l]^{B_k}
{V_{k+1}}  
	\ar@/^/[r]
	& 
	\ar@/^/[l]
	\cdots
	\ar@/^/[r]
	&
	\ar@/^/[l]
{V_{2k-2}}  
	\ar@/^/[r]^{A_{2k-2}}
	& 
{V_{2k-1}} 
\ar@/^/[l]^{B_{2k-2}}  
\\
\dim V_i & 1&2&\ldots&k-1&k&k-1&\ldots&2&1
}
\] 
\end{figure}

Following Nakajima, an element $(A,B,\Gamma,\Delta) \in S(d,v)$ is called {\em admissible} if the Atiyah-Drinfeld-Hitchin-Manin (ADHM) equations are satisfied. Equivalently, for all $1\leq i\leq n-1$,
\eq\label{eq:L1}
B_i A_i =  A_{i-1} B_{i-1} + \Gamma_i \Delta_i.
\endeq
An admissible element is called {\em stable} if, for each collection $U = (U_i \subseteq V_i)_i$ of subspaces satisfying that
\eq
\Im \: \Gamma_i \subseteq U_i, 
\quad
A_i(U_i) \subseteq U_{i+1}, 
\quad
B_i(U_{i+1}) \subseteq U_i 
\quad\quad 
\tforall i,
\endeq
it follows that $U_i = V_i$ for all $i$.
We will use the following equivalent notion of stability due to Maffei:
\begin{lemma}[{\cite[Lemmas~14, 2)]{Maf05}}]
An admissible element $(A,B,\Gamma,\Delta) \in S(d,v)$ is stable if and only if, for all $1\leq i\leq n-1$,
\eq\label{eq:L2}
\Im \: A_{i-1} + \sum_{j\geq i} \Im \: \Gamma_{j\to i} = V_i,
\endeq
where it is understood that $A_0 =0$, and that $\Gamma_{j\to i}$, for all $i, j$, is the natural composition from $D_j$ to $V_i$, i.e.,
\eq\label{def:Gaij}
\Gamma_{j\to i} = 
\begin{cases} 
B_i  \ldots B_{j-1} \Gamma_j &\tif j \geq i;
\\
A_{i-1} \ldots A_j \Gamma_j &\tif j \leq i.
\end{cases}
\endeq
\end{lemma}
Denote the subspace (which we call the {\em stable locus}) of $S(d,v)$ consisting of elements that are admissible (i.e., the ADHM equations \eqref{eq:L1} are satisfied) and stable by 
\eq\label{def:L+}
\Lambda^+(d,v) = \{(A,B,\Gamma,\Delta)  \in S(d,v) \mid \eqref{eq:L1}, \eqref{eq:L2}\},
\endeq
We denote by $\Lambda_{n-k,k}$ as the set of admissible representations in $S_{n-k,k}$, and $\Lambda^+_{n-k,k}$ as its stable locus.
\subsection{Nakajima quiver varieties}
Let $V = \prod_i V_i$, $D = \prod_i D_i$.
Now we define on any quiver representation $(A,B,\Gamma, \Delta)$  an action of $\GL(V) = \prod_i \GL(V_i)$ 
by
\eq\label{def:GL action}
\begin{aligned}
&g\cdot (A,B,\Gamma, \Delta) = ((g_{i+1}A_i g_i\inv)_i, (g_i B_i g_{i+1}\inv)_i, (g_i\Gamma_i)_i, (\Delta_ig_i\inv)_i), & g=(g_i)_i \in \GL(V).
\end{aligned}
\endeq
We denote the Nakajima quiver variety as stable $\GL(V)$-orbits on $S(d,v)$ satisfying pre-projective conditions, i.e., 
\eq\label{def:M}
M(d,v) := \Lambda^+(d,v)/\GL(V).
\endeq
Denote also by $M_{n-k,k}$ the Nakajima quiver variety for $S_{n-k,k}$.
We also call the projection onto the moduli space of the $\GL(V)$-orbits by
\eq\label{def:p}
p_{d,v}: \Lambda^+(d,v) \to M(d,v).
\endeq
Denote also by $p_{n-k,k}$ for the projection onto $M_{n-k,k}$.

It is first proved in \cite[Thm.~7.2]{Na94} that there is an explicit isomorphism $M(d,v) \to \tcS_x$ for certain $d,v,x$  using a different stability condition.
Here we recall a variant due to Maffei that suits our need. 
\begin{prop}[{\cite[Lemma~15]{Maf05}}]\label{prop:MS}
If $\dim D_i =0$ for all $i$ unless $i=1$, then the assignment below defines an isomorphism $\tphi = \tphi(d,v): M(d, v) \simeq \tcS_x$:
\eq
p_{d,v}(A, B, \Gamma, \Delta) 
\mapsto 
(\Delta_1\Gamma_1, (0 \subset \ker\Gamma_1 
\subset \ker \Gamma_{1\to2} 
\subset
\ldots \subset \ker \Gamma_{1\to n})),
\endeq
where $x = \Delta_1 \Gamma_1$. 
\end{prop}
In general, Proposition~\ref{prop:MS} does not apply to all $M_{n-k,k}$ for $n \geq 3$. 
Our next step is to describe an explicit isomorphism $\tphi_{n-k,k}: M_{n-k,k} \simeq \tcS_{n-k,k}$ due to Maffei in Proposition~\ref{prop:MS2}.
\subsection{Maffei's isomophism}
Following \cite{Maf05}, we utilize a modified quiver representation space $\widetilde{S}_{n-k,k}$ as in Figure~\ref{figure:tS} below, for each $S_{n-k,k}$:

\begin{figure}[ht!]
\caption{Modified quiver representations in $\widetilde{S}_{n-k,k}$.}
 \label{figure:tS}
\[ 
\xymatrix@C=18pt@R=9pt{
\dim \tD_i& n&0&\ldots&0&0
\\
&\ar@/_/[dd]_<(0.2){\tGa_1}
\tD_1
& &  &  &
\\ & & & & & & &  && 
\\ 
&
{\tV_1}  
	\ar@/_/[uu]_>(0.8){\tDe_1} 
	\ar@/^/[r]^{\tA_{1}}
	&   
	\ar@/^/[l]^{\tB_1}
{\tV_{2}}  
	\ar@/^/[r]
	& 
	\ar@/^/[l]
	\cdots
	\ar@/^/[r]
	&
	\ar@/^/[l]
{\tV_{n-2}}  
	\ar@/^/[r]^{\tA_{n-2}}
	& 
{\tV_{n-1}} 
\ar@/^/[l]^{\tB_{n-2}}  
}
\] 
\end{figure}

Here the vector spaces $(\tD_i, \tV_i)$ are given by 
\eq
\tD_1 = D'_0,
\quad
\tV_i = V_i \oplus D'_i,
\endeq
where
\eq\label{def:D'}
D'_i = \begin{cases}
\< e_1, \ldots, e_{n-k-i}, f_1, \ldots, f_{k-i}\>
&\tif i \leq k-1, 
\\
\qquad\hspace{2mm} \< e_1, \ldots, e_{n-k-i}\>
&\tif k \leq i \leq n-k-1, 
\\
\qquad\qquad\quad \{0\} &\tif n-k \leq i \leq n-1.
\end{cases}
\endeq
Note that we utilize the following identification with the spaces $D^{(h)}_j$ in \cite{Maf05}:
\eq
\begin{split} 
\<e_i\> \equiv D^{(i)}_{n-k},
\quad
\<f_i\> \equiv D^{(i)}_{k}&\quad\tif n > 2k, 
\\
\<e_i, f_i\> \equiv D^{(i)}_k &\quad\tif n =2k.
\end{split}
\endeq
%

Denote by $\td = (\dim \tD_i)_i, \tv = (\dim \tV_i)_i$ the dimension vectors.
The advantage of manipulating over such modified quivers is that Proposition~\ref{prop:MS} applies, and hence it produces an isomorphism between the Nakajima quiver variety $M(\td, \tv)$ and the Slodowy variety $\tcS(\td, \tv)$ for the dimension vectors $\td$ and $\tv$.

Now we identify the linear maps $\tA_i, \tB_i, \tGa_i, \tDe_i$ as block matrices in light of \cite[(9)]{Maf05}. 
For example, we have
\eq\label{eq:TTSS}
\tGa_1
=
\begin{blockarray}{ *{8}{c} }
&  f_b & \dots & e_b \\
\begin{block}{ c @{\quad} ( @{\,} *{7}{c} @{\,} )}
V_1& \TT_{0,V}^{f,b}&\dots & \TT_{0,V}^{e,b}\\
f_a & \TT_{0,f,a}^{f,b}&\dots & \TT_{0,f,a}^{e,b}\\
\vdots & \vdots & \ddots & \vdots \\
e_a  & \TT_{0,e,a}^{f,b}&\dots& \TT_{0,e,a}^{e,b}  \\
\end{block}
\end{blockarray}
\normalsize
~~~~~~, 
\quad 
\tDe_1=
\small 
\begin{blockarray}{ *{8}{c} }
& V_1 &   f_b& \dots &e_b\\
\begin{block}{ c @{\quad} ( @{\,} *{7}{c} @{\,} ) }
f_a & \SS^V_{0,f,a}& \SS_{0,f,a}^{f,b}&\dots & \SS_{0,f,a}^{e,b}\\
\vdots & \vdots & \vdots & \ddots & \vdots \\
 e_a & \SS^V_{0,e,a} & \SS_{0,e,a}^{k,b}&\dots& \SS_{0,e,a}^{e,b}  \\
\end{block}
\end{blockarray}
\normalsize 
~~~~~~,
\endeq
\eq
\tA_1
=
\small 
\begin{blockarray}{ *{8}{c} }
&V_1&  f_b &e_b\\
\begin{block}{ c @{\quad} ( @{\,} *{7}{c} @{\,} ) }
V_2& \mathbb{A}_1& \TT_{1,V}^{f,b} & \TT_{1,V}^{e,b}\\
f_a &\TT_{1,f,a}^{V}& \TT_{1,f,a}^{f,b} & \TT_{1,f,a}^{e,b}\\
e_a &\TT_{1,e,a}^{V} & \TT_{1,e,a}^{f,b}& \TT_{1,e,a}^{e,b}  \\
\end{block}
\end{blockarray}
\normalsize 
~~~~~~, 
\quad 
\tB_1=
\small 
\begin{blockarray}{ *{8}{c} }
& V_2 &  f_b &e_b\\
\begin{block}{ c @{\quad} ( @{\,} *{7}{c} @{\,} ) }
V_1& \mathbb{B}_1 & \SS_{1,V}^{f,b} & \SS_{1,V}^{e,b}\\
f_a & \SS^V_{1,f,a}& \SS_{1,f,a}^{f,b} & \SS_{1,f,a}^{e,b}\\
e_a & \SS^V_{1,e,a} & \SS_{1,e,a}^{f,b}& \SS_{1,e,a}^{e,b}  \\
\end{block}
\end{blockarray}~~~~~~,
\endeq
\normalsize
with respect to the basis vectors indicated above and to the left of each matrix. 
In other words, the variables $\mathbb{A}, \mathbb{B}, \SS, \TT$ are certain linear maps with domains and codomains specified as below, for $\phi, \psi \in \{e,f\}$:
\eq
\begin{split}
&\mathbb{A}_i: V_i \to V_{i+1},
\quad
\mathbb{B}_i: V_i \to V_{i+1},
\\
&
\SS_{i,\phi,a}^V: V_{i+1} \to \<\phi_a\>,
\quad
\SS_{i,V}^{\phi,a}: \<\phi_a\> \to V_i,
\quad
\SS_{i,\phi,a}^{\psi, b}: \<\psi_b\> \to \<\phi_a\>,
\\
&
\TT_{i,\phi,a}^V: V_{i} \to \<\phi_a\>,
\quad
\TT_{i,V}^{\phi,a}: \<\phi_a\> \to V_{i+1},
\quad
\TT_{i,\phi,a}^{\psi, b}: \<\psi_b\> \to \<\phi_a\>.
\end{split}
\endeq
A definition of the transversal element can be found in \cite[Defn.~16]{Maf05}.
In our context it is convenient to rewrite the definition as below: 
let
$\pi_{D'_i}$ be the projection onto $D'_i$ (recall \eqref{def:D'}), 
let $\tA_0 = \tGa_1, \tB_0 = \tDe_1$, 
and let $(x_i, y_i, [x_i, y_i])$ be the fixed $\fsl_2$-triple on $\fsl(D'_i)$ uniquely determined by
\eq\label{def:sl2}
\begin{aligned}
x_i(e_h) &= 
	\begin{cases}
	e_{h-1} &\tif 1< h \leq n-k-i, 
	\\
	\:\:\: 0 &\otw,
	\end{cases}
&&
y_i(e_h) = 
	\begin{cases}
	h(n-k-i-h)e_{h+1} &\tif 1\leq h < n-k-i, 
	\\
	\qquad \quad 0 &\otw,
	\end{cases}
\\
x_i(f_h) &= 
	\begin{cases}
	f_{h-1} &\tif 1< h \leq k-i, 
	\\
	\:\:\: 0 &\otw,
	\end{cases}
&&
y_i(f_h) = 
	\begin{cases}
	h(k-i-h) f_{h+1} &\tif 1\leq h < k-i, 
	\\
	\qquad \quad 0 &\otw.
	\end{cases}    
\end{aligned}
\endeq
An admissible quadruple $(\tA, \tB, \tGa, \tDe)$ in $\widetilde{S}_{n-k,k}$  is called {\em transversal} if the following conditions hold, for $0\leq i \leq n-2$:
\eq\label{eq:M1a}
\left[
\left.\pi_{D'_i} \tB_i \tA_i\right|_{D'_i} -x_i, y_i
\right] = 0,
\endeq
\eq\label{eq:M1b}
\begin{aligned}
\TT^{f,b}_{i,f,a} = \TT^{e,b}_{i,e,a}  &= 0,        
&\tif b > a+1;
&&\SS^{f,b}_{i,f,a} = \SS^{e,b}_{i,e,a}  &= 0,        
&\tif b > a;  
	\\
\TT^{f,b}_{i,f,a} = \TT^{e,b}_{i,e,a}  &= \textup{id},          
&\tif b = a+1;
&&\SS^{f,b}_{i,f,a} = \SS^{e,b}_{i,e,a}  &= \textup{id},          
&\tif b = a;  
	\\
\TT^{e,b}_{i,f,a}  &= 0,         
&\tif b \geq a+1; 
&&\TT^{f,b}_{i,e,a}  &= 0,         
&\tif b \geq a+1+2k-n; 
	\\
\SS^{e,b}_{i,f,a}  &= 0,         
&\tif b \geq a; 
&&\SS^{f,b}_{i,e,a}  &= 0,         
&\tif b \geq a+2k-n;
	\\
\TT^{V}_{i,j,a}        &= 0;         
&&&   
\SS^{j,b}_{i,V}        &= 0;
	\\
\TT^{j,b}_{i,V}    &= 0,         
&\tif b \neq 1;
&&\SS^{V}_{i,j,a} &=0
&\tif a \neq j-i.
\end{aligned}
\endeq
Denote the subspace in $\widetilde{S}_{n-k,k}$ consisting of transversal (hence admissible) and stable elements by 
\eq\label{def:cT+}
\fT^+_{n-k,k} = \{(\tA, \tB, \tGa, \tDe)\in \widetilde{S}_{n-k,k}\mid\eqref{eq:M1a}, \eqref{eq:M1b} \eqref{eq:M2}, \eqref{eq:M3} \},
\endeq
where the relations other than the transversal ones are
\begin{align}
\textup{(admissibility)}~&\tB_i\tA_i = \tA_{i-1} \tB_{i-1} + \tGa_i \tDe_i
\quad \mbox{for } 1\leq i\leq n-1, 
\label{eq:M2}
\\
\textup{(stability)}~&\Im \: \tA_{i-1} + \sum_{j\geq i} \Im \: \tGa_{j\to i} = \tV_i
\quad \mbox{for } 1\leq i\leq n-1,
\label{eq:M3}
\end{align}
where $\tGa_{j\to i}$ is defined similarly as \eqref{def:Gaij}.

\rmk
The system of equations \eqref{eq:M1b} is not the easiest to work with. 
For example, it implies that the map $\tGa_1$ must be of the following form:
\eq\label{eq:TTSS2}
\begin{tikzpicture}[baseline={(0,0)}, scale = 0.8]
\draw (1.2, 1.8) node {\scalebox{1.5}{$0$}};
\draw (-4.5, -.5) node {\scalebox{1.5}{$\TT_{0,e,\bullet}^{e,\bullet}$}};
\draw (7.2, 1.8) node {\scalebox{1.5}{$0$}};
\draw (-4.7, -2.7) node {\scalebox{1.5}{$\TT_{0,f,\bullet}^{e,\bullet}$}};
\draw (0, -2) node {\scalebox{1.5}{$0$}};
\draw (4.5, -.5) node {\scalebox{1.5}{$\TT_{0,e,\bullet}^{f,\bullet}$}};
\draw (8.2, -1.7) node {\scalebox{1.5}{$0$}};
\draw (4.5, -2.7) node {\scalebox{1.5}{$\TT_{0,f,\bullet}^{f,\bullet}$}};
  \draw (0,0) node {$
  \tGa_1
=
  \small 
\begin{blockarray}{ *{12}{c} }
&  e_1 & e_2 & \dots &\dots &\dots & e_{n-k} & f_1 & \dots & \dots & f_k 
\\
\begin{block}{ c @{\quad} ( @{\,} cccccc|ccccc @{\,} )}
V_1& \TT_{0,V}^{e,1}& 0 & \dots & \dots  & \dots& 0 & \TT_{0,V}^{f,1} & \dots & \dots & 0
\\
\cline{1-11}
e_1 & \TT_{0,e,1}^{e,1}& 1 &  & & &  &  &  &   & 
\\
\vdots & & \ddots & \ddots &  &  & &0 &  &   & 
\\
\vdots & &  & \ddots & \ddots &  & & & \ddots &   & 
\\
\vdots & &  &  & \ddots & \ddots &  & & & \ddots  & 
\\
e_{n-k-1} &  & &&& \TT_{0,e,n-k-1}^{e,n-k-1}&1 &  & && 0
\\
\cline{1-11}
f_1 & \TT_{0,f,1}^{e,1} &0 &&&&&\TT_{0,f,1}^{f,1} &1&&
\\
\vdots &  & \ddots &\ddots &&&& &\ddots&\ddots
\\
f_{k-1}
&&&\TT_{0,f,k-1}^{e,k-1}&0&&&&&\TT_{0,f,k-1}^{f,k-1}&1
\\
\end{block}
\end{blockarray}
\normalsize
~~~~.
$};
\end{tikzpicture}
\endeq
One can then solve for all the unknown variables $\TT_{0,\bullet,\bullet}^{\bullet,\bullet}$ using stability and admissibility conditions, together with the first transversality condition \eqref{eq:M1a}.
In general the solutions are very involved (see \cite[Lemma~18]{Maf05}.
We will use the proposition below to show that the solutions are actually very simple in our setup.
\endrmk
\begin{prop}\label{prop:Phi}
Let $(A, B, \Gamma, \Delta) \in \Lambda_{n-k,k}$. 
\begin{enumerate}[(a)]
\item There is a unique element $(\tA, \tB, \tGa, \tDe) \in \fT_{n-k,k}$ such that
\eq\label{eq:uniq}
\mathbb{A}_i = A_i,
\quad
\mathbb{B}_i = B_i,
\quad
\Gamma_k = \TT^{f,1}_{k-1,V},
~
\Gamma_{n-k} = \TT^{e,1}_{n-k-1,V},
\quad
\Delta_k = \SS^{V}_{k-1,f,1},
~
\Delta_{n-k} = \SS^{V}_{n-k-1,e,1}.
\endeq
\item The assignment in part (a) restricts to a $\GL(V)$-equivariant isomorphism $\Phi : \Lambda^+_{n-k,k} \overset{\sim}{\to} \fT^+_{n-k,k}$.
\item The map $\Phi$ induces an isomorphism $\Phi_M:M_{n-k,k} \overset{\sim}{\to} p_{n-k,k}(\fT^+_{n-k,k})$.
\end{enumerate}
\end{prop}
\proof
This is a special case of \cite[Lemmas~18, 19]{Maf05}.
\endproof
Now we are in a position to define Maffei's isomorphism $\tphi: M_{n-k,k} \overset{\sim}{\to} \tcS_{n-k,k}$.
Recall $\Lambda^+(-,-)$ from \eqref{def:L+}, $p_{d,v}$ from \eqref{def:p}, $M(-,-)$ from \eqref{def:M}, $\tphi(\td,\tv)$ from Proposition~\ref{prop:MS}, $\fT^+_{n-k,k}$ from \eqref{def:cT+}.
Finally, $\tphi_{n-k,k}$ is defined such that the lower right corner of the diagram below commute:
\eq\label{def:tphi}
\begin{tikzcd}
\Lambda^+(\td,\tv) 
\ar[r, "{\widetilde{p}~=~p_{\td,\tv}}", twoheadrightarrow] 
& 
M(\td,\tv) 
\ar[r, "{\tphi~=~\tphi(\td,\tv)}", "\sim"'] 
&
\tcS(\td,\tv) 
\\
\fT^+_{n-k,k}
\ar[r, twoheadrightarrow] \ar[u, hookrightarrow]
& 
\widetilde{p}(\fT^+_{n-k,k})
\ar[r, "\sim"'] \ar[u, hookrightarrow]
&
\tphi\circ\widetilde{p}(\fT^+_{n-k,k}).
\ar[u, hookrightarrow]
\\
\Lambda^+_{n-k,k}
\ar[r, "p_{n-k,k}", twoheadrightarrow]
\ar[u, "\Phi", "\sim"']
&
M_{n-k,k}
\ar[r, "\tphi_{n-k,k}", dashrightarrow, "\sim"']
\ar[u, "\Phi_M", "\sim"']
&
\tcS_{n-k,k}
\ar[u, equal]
\end{tikzcd}
\endeq




 

\begin{prop} \label{prop:MS2}
There map $\tphi_{n-k,k}: M_{n-k,k} \to \tcS_{n-k,k}$ defined above is an isomorphism of algebraic varieties.
\end{prop}
 \proof
 This is a special case of \cite[Thm. 8]{Maf05} by setting $N = n$, $r = (1, \ldots, 1)$, and $x\in\cN$ of Jordan type $(n-k,k)$.  
 \endproof


\section{Components of Springer fibers of type A} \label{sec:(n-k,k)}
For each cup $a\in B_{n-k,k}$, our strategy to single out the irreducible component $K^a \subset \cB_x \subset \tcS_{n-k,k}$ requires the following ingredients:
\begin{enumerate}
\item construction of a subset $\fT^a \subset \fT^+_{n-k,k}$ such that $\widetilde{p}(\fT^a) \simeq K^a$.
\item construction of a subset $\Lambda^a \subset \Lambda^+_{n-k,k}$ so that $\Phi(\Lambda^a) \simeq \fT^a$, which implies
$\Phi_M (p(\Lambda^a)) \simeq \widetilde{p}(\fT^a) \simeq K^a$.
\end{enumerate}
In light of \eqref{def:tphi}, we have
\eq\label{def:tphi2}
\begin{tikzcd}
\Lambda^+(\td,\tv) 
\ar[r, "{\widetilde{p}}", twoheadrightarrow] 
& 
M(\td,\tv) 
\ar[r, "{\tphi}", "\sim"'] 
&
\tcS(\td,\tv) 
\\
\fT^a
\ar[r, twoheadrightarrow] \ar[u, hookrightarrow]
& 
\widetilde{p}(\fT^a)
\ar[r, "\sim"', "\textup{Prop. }\ref{prop:a1}"] \ar[u, hookrightarrow]
&
K^a.
\ar[u, hookrightarrow]
\\
\Lambda^a
\ar[r, "p", twoheadrightarrow]
\ar[u, "\sim"', "\textup{Prop. }\ref{prop:a2}"]
&
p(\Lambda^a)
\ar[u, "\Phi_M", "\sim"']
&
\end{tikzcd}
\endeq
Recall that a cup diagram is uniquely determined by the configuration of its cups; that is, once the placement of the cups has been decided, rays emanate from the rest of the nodes.
Hence, for our construction of $\fT^a$ and $\Lambda^a$, we use only the information about the cups.
For completeness we also give a characterization for the ray relation on the quiver representation side.

\subsection{Irreducible components via quiver representations} 
Given a cup diagram $a = \{\{i_t, j_t\}\}_t\in B_{n-k,k}$,
we assume that $i_t < j_t$ for all $t$ and then denote the set of all vertices connected to the left (resp.\ right) endpoint of a cup in $a$ by
\eq
V_l^a = \{i_t \mid 1\leq t \leq k\},
\quad
V_r^a = \{j_t \mid 1\leq t \leq k\}.
\endeq
Define the endpoint-swapping map by
\eq
\sigma: V_l^a\to V_r^a,
\quad
i_t \mapsto j_t.
\endeq
Given $i\in V_l^a$, denote the ``size'' of the cup $\{i, \sigma(i)\}$ by
\eq
\delta(i)=\frac{1}{2}(\sigma(i)-i+1). 
\endeq
For instance, a minimal cup connecting neighboring vertices has size $\delta(i) = \frac{(i+1) -i + 1}{2} = 1$, and a cup containing a single minimal cup nested inside has size $2$.      

For short we set, for $0\leq p<q \leq n$,
\eq\label{eq:to}
\tB_{q\to p}
=
\tB_{p}\tB_{p+1}\cdots \tB_{q-1}: \tV_q \to \tV_p
,
\quad
\tA_{p \to q}
=
\tA_{q-1}\tA_{q-2}\cdots \tA_{p}: \tV_p \to \tV_q.
\endeq
Now we define
\eq\label{defi:M_n^a}
\fT^a
=\{(\tA,\tB,\tGa,\tDe)\in \fT^+_{n-k,k}
\mid
\ker \tB_{i+\delta(i)-1 \to i-1} 
= \ker \tA_{\sigma(i)-\delta(i) \to \sigma(i)} 
\tforall i\in V_l^a
\}.
\endeq 
The kernel condition in \eqref{defi:M_n^a} can be visualized in Figure~\ref{fig:Ma} below:

\begin{figure}[ht!]
\caption{The paths in the kernel condition of \eqref{defi:M_n^a}.}
\label{fig:Ma}
\[ 
\xymatrix@C=25pt@R=9pt{
	&   
	& 
	&
	\ar@{=}[d]
{\tV_{\sigma(i)-\delta(i)}}  
	\ar@/^/[r]^{\tA_{\sigma(i)-\delta(i)}}
	& 
{\tV_{i+\delta(i)}}  
	\ar@/^/[r]^{\tA_{\sigma(i)-\delta(i)+1}}
	&
	\cdots
   \ar@/^/[r]^{\tA_{\sigma(i)-1}}
	& 
{\tV_{\sigma(i)}.} 
\\
{\tV_{i-1}}  
	&   
	\ar@/^/[l]^{\tB_{i-1}}
{\tV_{i}}  
	& 
	\ar@/^/[l]^{\tB_{i}}
	\cdots
	&
	\ar@/^/[l]^{\tB_{i+\delta(i)-2}}
{\tV_{i+\delta(i)-1}}  
	& 
	&
	& 
}
\] 
\end{figure}

Note that for a minimal cup connecting neighboring vertices $i$ and $i+1$ the relations in \eqref{defi:M_n^a} take the simple form 
\eq
\begin{split}
\ker \tDe_1 = \ker \tA_1 
&\quad\tif i =1;
\\
\ker \tB_{i-1} = \ker \tA_i 
&\quad\tif 2 \leq i \leq n-2;
\\
\ker \tB_{n-2} = \tV_{n-1} = \CC 
&\quad\tif i=n-1.
\end{split}
\endeq 
\begin{prop}\label{prop:a1}
For $a \in B_{n-k,k}$,
we have an equality
$
\widetilde{p}(\fT^a) =  \tphi\inv(K^a).
$
\end{prop}
\proof
Thanks to Proposition~\ref{prop:MS}, it suffices to show that, for any $i\in V_l^a$ and
$(\tA, \tB, \tGa, \tDe) \in \fT^a$,
the kernel condition 
\eq\label{eq:cup_condition_quiver_side_rewritten} 
\ker \tB_{i+\delta(i)-1 \to i-1} 
= \ker \tA_{\sigma(i)-\delta(i) \to \sigma(i)}
\endeq 
is equivalent to the Fung/Stroppel--Webster cup relation (see Proposition~\ref{prop:known_results_about_irred_comp}~(b)(i))
\begin{equation}\label{eq:equivalent_condition}
(\tDe_1\tGa_1)^{-\delta(i)}
\ker \tGa_{1\to i-1}
=\ker \tGa_{1 \to \sigma(i)}.
\end{equation}
Note that the left-hand side of \eqref{eq:equivalent_condition} can be rewritten as follows:
\begin{align*}
(\tDe_1\tGa_1)^{-\delta(i)}
(\ker \tGa_{1 \to i-1}) 
	&= \ker (\tGa_{1 \to i-1} (\tDe_1\tGa_1)^{\delta(i)}) 
	\\
	&= \ker (\tA_{1 \to i-1} (\tB_1\tA_1)^{\delta(i)}  \tGa_1)	
	\\
	&= \ker (\tB_{i+\delta(i)-1 \to i-1}\tGa_{1\to i+\delta(i)-1}),
\end{align*}
where the second equality follows from applying $\delta(i)$ times the admissibility condition $\tB_1\tA_1=\tGa_1\tDe_1$;
while the third equality follows from applying
the admissibility condition $\tA_{t-1}\tB_{t-1} = \tB_t\tA_t$ repeatedly from $t=2$ to $t=i+\delta(i)-1 = \sigma(i)-\delta(i)$. 
Therefore, the cup relation 
\eqref{eq:equivalent_condition}
is equivalent to another kernel condition below
\eq\label{eq:final_equivalent_condition}
\ker \tB_{i+\delta(i)-1 \to i-1}\tGa_{1\to i+\delta(i)-1}
=\ker \tGa_{1 \to \sigma(i)}.
\endeq
In particular, the kernels are equal for the two maps in Figure~\ref{fig:cupeq} given by dashed and solid arrows, respectively:

\begin{figure}[ht!]
\caption{The kernel condition that is equivalent to \eqref{eq:equivalent_condition}.}
\label{fig:cupeq}
\[ 
\xymatrix@C=25pt@R=25pt{
{\tD_{i-1}}  
	\ar@/^/[d]^{\tGa_1}
	\ar@/_/@{.>}[d]_{\tGa_1}
\\
{\tV_{1}}
	\ar@/^/[r]^{\tA_{1}}
	\ar@/_/@{.>}[r]_{\tA_{1}}  
&
{\tV_{2}}
	\ar@/^/[r]^{\tA_{2}}
	\ar@/_/@{.>}[r]_{\tA_{2}}  
&
\dots  
	\ar@/^/[r]
	\ar@/_/@{.>}[r]  
	&
	\ar@{=}[d]
{\tV_{\sigma(i)-\delta(i)}}  
	\ar@/^/[r]^{\tA_{\sigma(i)-\delta(i)}}
	& 
	\cdots
   \ar@/^/[r]^{\tA_{\sigma(i)-1}}
	& 
{\tV_{\sigma(i)}} 
\\
&
{\tV_{i-1}}  
	&
	\ar@/^/@{.>}[l]^{\tB_{i}}   
	\cdots
	&
	\ar@/^/@{.>}[l]^{\tB_{i+\delta(i)-2}}
{\tV_{i+\delta(i)-1}}  
	& 
	& 
}
\] 
\end{figure}

Note that~\eqref{eq:cup_condition_quiver_side_rewritten} evidently implies~\eqref{eq:final_equivalent_condition}. 
By the stability conditions on $\tV_1, \ldots, \tV_{i+\delta(i) -1}$ we see that the maps $\tGa_1, \tA_1, \ldots, \tA_{i+\delta(i)-2}$ are all surjective, and so is its composition $\tGa_{1\to i+\delta(i)-1}$. 
Thus, \eqref{eq:final_equivalent_condition} implies~\eqref{eq:cup_condition_quiver_side_rewritten}, and we are done.
\endproof
\subsection{Springer fibers via quiver representations} \label{sec:hell}

For $1\leq i \leq n-1$, we define an isomorphic copy of $D'_i$ (see \eqref{def:D'}) with a shift of index by $t$
as
\eq
D'_i[t] = \<f_{i+t}\mid f_i \in D'_i\>  \oplus \< e_{i+t} \mid e_i \in D'_i\>.
\endeq
By a slight abuse of notation, we denote by $\Gamma_{\to i}$ the assignment given by
\eq
e_a \mapsto \Gamma_{n-k\to i}(e) \in V_i,
\quad
f_b \mapsto \Gamma_{k\to i}(f) \in V_i.
\endeq
We define the obvious composition by
\eq\label{def:Deij}
\Delta_{j\to i} = 
\begin{cases} 
\Delta_i B_i  \cdots B_{j-1}  &\tif j \geq i, 
\\
\Delta_i A_{i-1} \cdots A_j  &\tif j \leq i,
\end{cases}
\endeq
and then define $\Delta_{i\to }$ similarly.
\begin{lemma}\label{lem:hell}
If $(\tA, \tB, \tGa, \tDe) \in \fT^+_{n-k,k}$ and $\Phi\inv(\tA, \tB, \tGa, \tDe) = (A,B,\Gamma, \Delta) \in \Lambda^+_{n-k,k}$, then $(\tA, \tB, \tGa, \tDe)$ must be of the form, for $1\leq i \leq n-2$: 
\eq
\tGa_1
=
\begin{blockarray}{ *{8}{c} }
&  D'_0 \setminus D'_1[1] & D'_1[1] \\
\begin{block}{ c @{\quad} ( @{\,} *{7}{c} @{\,} )}
V_1& \Gamma_{\to1}&0\\
D'_1 & 0& I_{n-2} & \\
\end{block}
\end{blockarray}
\normalsize
~~~~~, 
\quad 
\tDe_1=
\small 
\begin{blockarray}{ *{8}{c} }
& V_1 &  D'_1\\
\begin{block}{ c @{\quad} ( @{\,} *{7}{c} @{\,} ) }
D'_1 & 0 & I_{n-2} \\
D'_0 \setminus D'_1& \Delta_{1\to} &   0\\
\end{block}
\end{blockarray}
\normalsize 
~~~~~,
\endeq
\eq
\tA_i
=
\begin{blockarray}{ *{8}{c} }
& V_i & D'_{i} \setminus D'_{i+1}[1] & D'_{i+1}[1] \\
\begin{block}{ c @{\quad} ( @{\,} *{7}{c} @{\,} )}
V_{i+1}& {A}_i&\Gamma_{\to i+1} & 0\\
D'_{i+1} & 0 &0 &I   \\
\end{block}
\end{blockarray}
\normalsize
~~~~~,
\quad
\tB_i=
\small 
\begin{blockarray}{ *{8}{c} }
& V_{i+1} &   D'_{i+1}\\
\begin{block}{ c @{\quad} ( @{\,} *{7}{c} @{\,} ) }
V_i & {B}_i & 0 \\
D'_{i+1}& 0 &   I   \\
D'_i \setminus D'_{i+1} & \Delta_{i+1 \to} &0\\
\end{block}
\end{blockarray}
\normalsize 
~~~~~.
\endeq
Moreover, the following equation hold, for $2\leq i \leq n-1$,
\eq\label{eq:T}
\Delta_{i \to } \Gamma_{\to i} = 0.
\endeq
\end{lemma}
\proof
By Proposition~\ref{prop:Phi} we know that $(\tA,\tB,\tGa,\tDe)$ must be the unique element in $\fT^+_{n-k,k}$ that satisfies \eqref{eq:uniq}, which indeed are satisfied from the construction.
What remains to show is that the formulas above do define an element in $\fT^+_{n-k,k}$.

From construction, we have
\eq\label{eq:tDG}
\tGa_1 \tDe_1=
\small 
\begin{blockarray}{ *{8}{c} }
& V_{1} & D'_1\setminus D'_2[1] & D'_{2}[1]\\
\begin{block}{ c @{\quad} ( @{\,} *{7}{c} @{\,} ) }
V_1 & 0 & \Gamma_{\to1} & 0 
\\
D'_2& 0 & 0&  I   \\
D'_1 \setminus D'_2 &0 &0&0\\
\end{block}
\end{blockarray}
\normalsize 
~~~~,
\quad
\tDe_1 \tGa_1=
\small 
\begin{blockarray}{ *{8}{c} }
&  D'_0 \setminus D'_1[1]& D'_{1}[1]\\
\begin{block}{ c @{\quad} ( @{\,} *{7}{c} @{\,} ) }
D'_1& 0&  I   \\
D'_0 \setminus D'_1& \Delta_{1\to}\Gamma_{\to 1} &0\\
\end{block}
\end{blockarray}
\normalsize 
~~~~~,
\endeq
\eq\label{eq:tAtB}
\tA_i \tB_i=
\small 
\begin{blockarray}{ *{8}{c} }
& V_{i+1} & D'_{i+1} \setminus D'_{i+2}[1]& D'_{i+2}[1]\\
\begin{block}{ c @{\quad} ( @{\,} *{7}{c} @{\,} ) }
V_{i+1} & A_iB_i & \Gamma_{\to i+1}  & 0 
\\
D'_{i+2}& 0& 0&  I   \\
D'_{i+1} \setminus D'_{i+2} & \Delta_{i+1\to} &0&0\\
\end{block}
\end{blockarray}
\normalsize 
~~~~~~,
\endeq
\eq\label{eq:tBtA}
\tB_i \tA_i=
\small 
\begin{blockarray}{ *{8}{c} }
& V_i & D'_i\setminus D'_{i+1}[1]& D'_{i+1}[1]\\
\begin{block}{ c @{\quad} ( @{\,} *{7}{c} @{\,} ) }
V_i& B_i A_i & \Gamma_{\to i}  & 0\\
D'_{i+1}& 0&0&  I   \\
D'_i \setminus D'_{i+1} &\Delta_{i\to} &\Delta_{i+1\to}\Gamma_{\to i+1}&0\\
\end{block}
\end{blockarray}
\normalsize 
~~~~~~,
\endeq
where $I$ represents the identity map of appropriate rank.
The admissibility conditions \eqref{eq:M2} follow from comparing the entries in \eqref{eq:tAtB} -- \eqref{eq:tBtA} together with the original admissibility condition \eqref{eq:L2} and \eqref{eq:T}. 

By the original stability condition \eqref{eq:L1}, the blocks in the first row for $\tA_i$ has full rank, and hence \eqref{eq:M3} follows.

From \eqref{eq:tBtA} we see that
\eq
\left.\pi_{D'_i} \tB_i \tA_i\right|_{D'_i} -x_i
=
\small 
\begin{blockarray}{ *{8}{c} }
&  D'_i \setminus D'_{i+1}[1] & D'_{i+1}[1]\\
\begin{block}{ c @{\quad} ( @{\,} *{7}{c} @{\,} ) }
D'_{i+1}&0&  0   \\
D'_i \setminus D'_{i+1} &\Delta_{i+1\to}\Gamma_{\to i+1}&0\\
\end{block}
\end{blockarray}
\normalsize 
~~~~~~, 
\endeq
and thus \eqref{eq:M1a} holds.
Finally, a straightforward verification such as \eqref{eq:TTSS2} shows that \eqref{eq:M1b} are satisfied.
\endproof
We can now combine Lemma~\ref{lem:hell} and the Nakajima-Maffei isomorphism to describe the complete flag assigned to each  quiver representation. 
\begin{cor}\label{cor:Fi}
If $(x, F_\bullet) = \tphi_{n-k,k}(p_{n-k,k}(A,B,\Gamma,\Delta))$ for some $(A,B,\Gamma,\Delta) \in \Lambda^+_{n-k,k}$,
then
\[
F_i = \ker \begin{blockarray}{ *{8}{c} }
 & D''_0 & \dots & D''_{t-1} & \dots & D''_{i-1} \\
\begin{block}{ c @{\quad} ( @{\,} *{7}{c} @{\,} )}
V_{i}& 
A_{1\to i}\Gamma_{\to 1}
&
\dots
&
A_{t\to i}\Gamma_{\to t}
&
\dots
&\Gamma_{\to i} \\
\end{block}
\end{blockarray}
\normalsize
~~~~~,
\]
where $D''_{t-1}$, for $1\leq t \leq i$, is the space (depending on $i$) described below:
\eq
D''_{t-1} = D'_{t-1}[t-1] - D'_{t}[t]
=
\begin{cases}
\<e_{t}, f_{t}\> &\tif t \leq k, 
\\
\hspace{3mm} \<e_{t}\> &\tif k+1 \leq t \leq n-k, 
\\
\hspace{3mm} \{0\} &\textup{otherwise}.
\end{cases}
\endeq
\end{cor}
\proof
By Proposition~\ref{prop:MS} we know that the spaces $F_i$ are determined by the kernels of the maps $\tGa_{1\to i}$.
The assertion follows from a direct computation of $\tGa_{1\to i}$ using Lemma~\ref{lem:hell}, which is,
\eq\label{eq:FiKer}
\tGa_{1\to i}
=
\begin{blockarray}{ *{8}{c} }
 & D''_0 & \dots & D''_{t-1} & \dots & D''_{i-1} & D'_{i}[i] \\
\begin{block}{ c @{\quad} ( @{\,} *{7}{c} @{\,} )}
V_{i}& 
A_{1\to i}\Gamma_{\to 1}
&
\dots
&
A_{t\to i}\Gamma_{\to t}
&
\dots
&\Gamma_{\to i} & 0\\
D'_{i}  &0 &\dots & 0 &
\dots & 0 &I   \\
\end{block}
\end{blockarray}
\normalsize
~~~~~.
\endeq
\endproof
\ex\label{ex:Fi}
The quiver representations in $S_{2,2}$ are described as below:
\begin{equation}\label{eq:22-quiver} 
\xymatrix@C=18pt@R=9pt{
&
\ar@/_/[dd]_<(0.2){\Gamma_2}
D_2 =\<e,f\>
&   
\\   
& &    
\\ 
V_1 =\CC
	\ar@/^/[r]^{A_1}
	& 
	\ar@/^/[l]^{B_1} 
V_2  =\CC^2
	\ar@/_/[uu]_>(0.8){\Delta_2} 
	\ar@/^/[r]^{A_{2}}
	& 
	\ar@/^/[l]^{B_2}
V_3 = \CC. 
	}
\end{equation} 
Let $\tphi_{2,2}\inv(x, F_\bullet) =  p_{2,2}(A,B,\Gamma,\Delta) \in M_{2,2}$.
By Corollary~\ref{cor:Fi}, the  flag $F_\bullet$ is described by,
\[
\begin{split}
F_1 &=  
 \ker \begin{blockarray}{ *{8}{c} }
 & \<f_1, e_1\>  \\
\begin{block}{ c @{\quad} ( @{\,} *{7}{c} @{\,} )}
V_{1}& 
B_1\Gamma_2
\\
\end{block}
\end{blockarray}
\normalsize
~~~~~,
\\
F_2 &= 
 \ker \begin{blockarray}{ *{8}{c} }
 & \<f_1,e_1\> &\<f_2,e_2\>  \\
\begin{block}{ c @{\quad} ( @{\,} *{7}{c} @{\,} )}
V_{2}& 
A_1B_1\Gamma_2
&
\Gamma_2
\\
\end{block}
\end{blockarray}
\normalsize
~~~~~,
\\
F_3 &= 
 \ker \begin{blockarray}{ *{8}{c} }
 & \<f_1,e_1\> &\<f_2,e_2\>  \\
\begin{block}{ c @{\quad} ( @{\,} *{7}{c} @{\,} )}
V_{3}& 
A_2A_1B_1\Gamma_2
&
A_2\Gamma_2
\\
\end{block}
\end{blockarray}
\normalsize
~~~~~. 
\end{split}
\]
In particular, if 
$
A_1 = \begin{pmatrix}1 \\ 0\end{pmatrix} = B_2,
A_2 = \begin{pmatrix} 0 & 1 \end{pmatrix} = B_1,
\Gamma_2 = \begin{pmatrix}1&0\\0&1\end{pmatrix},
\Delta = 0,
$
then
\[
F_1 = \ker \begin{pmatrix} 0 & 1 \end{pmatrix} = \<f_1\>,
\quad
F_2 = \ker \begin{pmatrix} 0 & 1 &1& 0 \\ 0&0&0&1 \end{pmatrix} = \<f_1, e_1 - f_2\>,
\quad
F_3 = \ker \begin{pmatrix} 0 & 0 &0& 1  \end{pmatrix} = \<f_1, e_1, f_2\>.
\]
\endex
\ex
The quiver representations in $S_{3,1}$ are described as below:
\[
 \xymatrix@-1pc{
D_1=\<f\> \ar@/^/[dd]^{\Gamma_1} &&& &   D_3=\<e\>  \ar@/^/[dd]^{\Gamma_3.}  \\ 
 & & \\ 
V_1=\CC\ar@/^/[uu]^{\Delta_1} \ar@/^/[rr]^{A_1} 
& & \ar@/^/[ll]^{B_1}  V_2=\CC  \ar@/^/[rr]^{A_2}   
& & \ar@/^/[ll]^{B_2}  V_3=\CC  \ar@/^/[uu]^{\Delta_3}\\ 
 }
\] 
If $\tphi_{3,1}\inv(x, F_\bullet) =  p_{3,1}(A,B,\Gamma,\Delta) \in M_{3,1}$, then the flag $F_\bullet$ can be described by
\[
\begin{split}
F_1 &= 
 \ker \begin{blockarray}{ *{8}{c} }
 & \<f_1\> &\<e_1\>  \\
\begin{block}{ c @{\quad} ( @{\,} *{7}{c} @{\,} )}
V_{1}& 
\Gamma_1
&
B_1B_2\Gamma_3
\\
\end{block}
\end{blockarray}
\normalsize
~~~~~,
\\
F_2 &
= \ker \begin{blockarray}{ *{8}{c} }
 & \<f_1\> & \<e_1\>  & \<e_2\> \\
\begin{block}{ c @{\quad} ( @{\,} *{7}{c} @{\,} )}
V_{2}&
A_1\Gamma_1
& 
A_1B_1B_2\Gamma_3
&
B_2\Gamma_3
\\
\end{block}
\end{blockarray}
\normalsize
~~~~~,
\\
F_3 & =  \ker \begin{blockarray}{ *{8}{c} }
 & \<f_1\>& \<e_1\>  & \<e_2\> & \<e_3\>\\
\begin{block}{ c @{\quad} ( @{\,} *{7}{c} @{\,} )}
V_{3}& 
A_2A_1\Gamma_1
&
A_2A_1B_1B_2\Gamma_3
&
A_2B_2\Gamma_3
&
\Gamma_3
\\
\end{block}
\end{blockarray}
\normalsize
~~~~~.
\end{split}
\]
In particular, if 
$
A_1 = 1, A_2 =0,
B_1 = 0, B_2 =1,
\Gamma_1 = 1, \Gamma_3 = 1,
\Delta_1 = 0 = \Delta_3,
$
then
\[
F_1 = \ker \begin{pmatrix} 1 & 0 \end{pmatrix} = \<e_1\>,
\quad
F_2 = \ker \begin{pmatrix} 1 & 0 & 1  \end{pmatrix} = \<e_1, f_1 - e_2\>,
\quad
F_3 = \ker \begin{pmatrix} 0 & 0 &0& 1  \end{pmatrix} = \<e_1, f_1, e_2\>.
\]
\endex
\rmk
The previous example demonstrates that Corollary~\ref{cor:Fi} provides an efficient way to compute the corresponding complete flag in the Slodowy variety, while generally it is very implicit to apply Maffei's isomorphism $\tphi$.
It can also be seen that $\ker \tGa_{1\to i}$ is not  {necessarily} a direct sum of the kernels of the blocks.
\endrmk

Define $\Lambda^\cB_{n-k,k}$ to be the subset of $\Lambda^+_{n-k,k}$ so that
\eq
p_{n-k,k}(\Lambda^\cB_{n-k,k}) = \tphi\inv(\cB_{n-k,k}).
\endeq
In other words, $\Lambda^\cB_{n-k,k}$ is the incarnation of the Springer fiber we wish to study via the corresponding Nakajima quiver variety, which is characterized first by Lusztig \cite[Prop. 14.2(a)]{Lu91}, \cite[Lemma 2.22]{Lu98}.
Below we demonstrate an elementary proof in the two-row case using Lemma~\ref{lem:hell}, which is stated without proof in \cite[Rmk.~24]{Maf05}.

\begin{cor}\label{cor:Spr}
If $x\in \cN$ has Jordan type $(n-k,k)$, then
\[
\Lambda^\cB_{n-k,k} = \{(A,B,\Gamma,\Delta) \in \Lambda^+_{n-k,k}\mid\Delta = 0 \}.
\]
\end{cor}
\proof
If $(A,B,\Gamma,\Delta) \in \Lambda^\cB_{n-k,k}$ then $\Phi(A,B,\Gamma, \Delta) = (\tA, \tB, \tGa, \tDe)$ is exactly of the form described in Lemma~\ref{lem:hell},
 in addition to the fact that $\tDe_1\tGa_1 = x_0$ thanks to Proposition~\ref{prop:MS}.
Hence, we can extend \eqref{eq:T} by including
\eq\label{eq:T2}
\Delta_{1\to}\Gamma_{\to 1} = 0.
\endeq   
We can now show by induction on $i$ that for $j \in \{k,n-k\}$ and $0\leq i \leq j$,  
\eq\label{eq:0claim}
\Delta_{i\to j} = 0,
\endeq
which implies that $\Delta = 0$.

For the base case $i=0$, we use the stability condition \eqref{eq:L2} for $V_1$, which is, 
$\Im \Gamma_{\to 1} = V_1$.
Therefore, by \eqref{eq:T2} we have
\eq
0 = \Im \Delta_{1 \to }\Gamma_{\to 1} 
= \Delta_{1\to }(V_1).
\endeq
For the inductive step, we will deal with two cases: $i<k$ or $k\leq i < n-k$, assuming \eqref{eq:0claim} holds for $0,1, \ldots, i-1$. 
For the first case $i<k$, we use the stability condition \eqref{eq:L2} for $V_i$, i.e.,
\eq\label{eq:Si}
\Im A_{i-1} + \Im \Gamma_{k\to i} + \Im \Gamma_{n-k\to i} = V_i.
\endeq
Now by the inductive hypothesis we see that
\eq\label{eq:indhyp}
0 = \Delta_{i-1\to j} = \Delta_{i \to j} A_{i-1}.
\endeq
By \eqref{eq:Si} we have, for $j \in \{k,n-k\}$, 
\eq
\begin{split}
\Delta_{i\to j}(V_i) 
&=
\Delta_{i\to j}(\Im A_{i-1} + \Im \Gamma_{k\to i} + \Im \Gamma_{n-k\to i})
\\
&= 
\Im \Delta_{i \to j} A_{i-1}
+
\Im \Delta_{i \to j} \Gamma_{n-k\to i}
+
\Im \Delta_{i \to j} \Gamma_{k\to i}  
= 0,
\end{split}
\endeq
where the last equality follows from \eqref{eq:indhyp} and \eqref{eq:T}. 
Thus $\Delta_{k} = 0$. 

For the second case $k\leq i \leq n-k$ we have
$\Im A_{i-1} + \Im \Gamma_{n-k \to i} = V_i$.
Similarly,
\eq
\Delta_{i\to n-k}(V_i) = 
\Im \Delta_{i \to n-k} A_{i-1}
+
\Im \Delta_{i \to n-k} \Gamma_{n-k\to i} 
= 0,
\endeq
which leads to $\Delta_{n-k} = 0$. We are done.
\endproof
\subsection{Maffei's immersion and components of Springer fibers} 
Now we can prove the final piece of the main theorem using Lemma~\ref{lem:hell}.
Given a cup diagram $a\in B_{n-k,k}$, define 
\eq\label{defi:Springer_component_in_quiver}
\Lambda^a
=\{
(A,B,\Gamma,\Delta) \in \Lambda^\cB_{n-k,k}
\mid
\ker B_{i+\delta(i)-1 \to i-1} 
= \ker A_{\sigma(i)-\delta(i) \to \sigma(i)} 
\tforall i\in V_l^a
\},
\endeq
where the maps $A_{i\to j}, B_{j \to i}$ are defined similarly as their tilde versions in \eqref{eq:to}.
\begin{prop}\label{prop:hellA}
\label{prop:a2}
For $a \in B_{n-k,k}$,
we have an equality
$
\Phi(\Lambda^a) = \fT^a.
$
\end{prop}
\proof
Let $\Phi\inv(\tA, \tB, \tGa, \tDe) = (A,B, \Gamma, \Delta) \in \Lambda^a$.
From Corollary~\ref{cor:Spr} we see that $\Delta$ must be zero.
The proposition follows as long as we show that
$\ker \tB_{i+\delta(i)-1 \to i-1} 
= \ker \tA_{\sigma(i)-\delta(i) \to \sigma(i)}$ 
is equivalent to $\ker B_{i+\delta(i)-1 \to i-1} 
= \ker A_{\sigma(i)-\delta(i) \to \sigma(i)}$. 
For simplicity, let us use the shorthand $a = i-1 < b = i+\delta(i)-1 = \sigma(i) - \delta(i) < c = \sigma(i)$.

Using Lemma~\ref{lem:hell}, we obtain that,
for $a<b<c$, 
\eq
\tB_{b\to a}=
\small 
\begin{blockarray}{ *{8}{c} }
& V_{b} &   D'_{b}\\
\begin{block}{ c @{\quad} ( @{\,} *{7}{c} @{\,} ) }
V_a & {B}_{b\to a} & 0 \\
D'_{b}& 0 &   I   \\
D'_a \setminus D'_b & 0 &0\\
\end{block}
\end{blockarray}
\normalsize 
~~~~~,
\endeq
\eq
\tA_{b\to c}
=
\begin{blockarray}{ *{8}{c} }
& V_b & D''_b & \dots & D''_t & \dots & D''_{c-1} & D'_{c}[c-b] \\
\begin{block}{ c @{\quad} ( @{\,} *{7}{c} @{\,} )}
V_{c}& {A}_{b\to c}&
A_{b+1\to c}\Gamma_{\to b+1}
&
\dots
&
A_{t+1\to c}\Gamma_{\to t+1}
&
\dots
&\Gamma_{\to c} & 0\\
D'_{c} & 0 &0 &\dots & 0 &
\dots & 0 &I   \\
\end{block}
\end{blockarray}
\normalsize
~~~~~,
\endeq
where $D''_i$ is the space (depending on fixed $b<c$) described below:
\eq\label{def:D''}
D''_t = D'_{t}[t-b] \setminus D'_{t+1}[t-b+1]
=
\begin{cases}
\<e_{t-b+1}, f_{t-b+1}\> &\tif t \leq k, 
\\
\hspace{6mm} \<e_{t-b+1}\> &\tif k+1 \leq t \leq n-k, 
\\
\hspace{10mm} \{0\} &\textup{otherwise}.
\end{cases}
\endeq

Since $\tB_{b\to a}$ acts as an identity map on $D'_b$, its kernel must lie in $V_b$.
Moreover, $\ker \tB_{b\to a} = \ker B_{b\to a}$. 

Assume that $\ker \tB_{b \to a} 
= \ker \tA_{b \to c}$.
It follows that $\ker \tA_{b \to c} \subset V_b$.
In other words,
$D'_b \setminus D'_{c-1}[c-b]$ must not lie in the kernel, and hence $\ker A_{b\to c} =\ker \tA_{b\to c} = \ker \tB_{b \to a} = \ker B_{b\to a}$.

On the other hand, assuming $\ker B_{b \to a} 
= \ker A_{b \to c}$, we need to show that
$
D'_b \setminus D'_{c-1}[c-b] \not\in \ker
\tA_{b \to c}.
$
In other words, for $b \leq t \leq c-1$,
any composition of maps of the following form must be nonzero: 
\eq\label{eq:a23}
\begin{array}{cc}
\xymatrix@C=25pt@R=25pt{
&&
{D_{k}}  
	\ar@/^/[d]^{\Gamma_{k}}
\\
{V_{t+1}}
	\ar@/_/[r]_{A_{t+1 \to c}}  
&
{V_{c}}
	\ar@/_/[l]_{B_{c\to t+1}}  
&
	\ar@/_/[l]_{B_{k \to c}}	
{V_{k}}  
}
&
\xymatrix@C=25pt@R=25pt{
&&
{D_{n-k}}  
	\ar@/^/[d]^{\Gamma_{n-k}}
\\
{V_{t+1}}
	\ar@/_/[r]_{A_{t+1 \to c}}    
&
{V_{c}}
	\ar@/_/[l]_{B_{c\to t+1}}    
&
	\ar@/_/[l]_{B_{n-k \to c}}		
{V_{n-k}}  
}
\\
\\
\tif t+1 \leq k, 
&
\tif t+1 \neq n-k. 
\end{array}
\endeq

Note that the spaces $D''_t$ are nonzero only when $t \leq n-k$, and hence the maps $A_{t+1 \to c}\Gamma_{\to t+1}$ in \eqref{eq:a23} only exist when $c \leq n-k$.
Therefore, the proposition is proved for $c > n-k$, and we may assume now $c \leq n-k$.

We first prove \eqref{eq:a23} for the base case $t=c-1$ by contradiction. 
Our strategy is to construct nonzero vectors $v_i \in V_i \cap \Im A_{1\to i}$ for $1\leq i \leq c$.
If this claim holds, then by admissibility conditions from $V_1$ to $V_i$, we have
\eq
B_{i-1}(v_i) = 
B_{i-1}A_{1 \to i}(v_{1}) =
A_{1\to i-1}B_1A_1(v_1) = 0,
\endeq
and hence there is a nonzero vector $v_b \in \ker B_{b\to a}$.
On the other hand, $v_b \not\in \ker A_{b\to c}$ since $A_{b\to c}(v_b) = v_c \neq 0$, a  contradiction.
Now we prove the claim. Suppose that $\Gamma_{n-k \to c} = 0$.
\eq
\Gamma_{n-k \to i} = B_{c\to i}\Gamma_{n-k \to c}= 0, \quad \tforall i < c.
\endeq
By the stability condition on $V_1$, we have 
\eq
V_1 = \Im \Gamma_{k\to1} + \Im \Gamma_{n-k\to 1} = \Gamma_{k\to1}(f),
\endeq
 and hence the vector $\phi_i = \Gamma_{k\to i}(f) \in V_i$ are all nonzero for $i \leq k$.
Define $v_i= A_{1\to i}(\phi_1)$ for all $i$.
The stability condition on $V_2$ now reads
\eq
V_2 = \Im A_1 + \Im \Gamma_{k\to2} + \Im \Gamma_{n-k\to 2} = \<A_1(\phi_1)\> + \<\phi_2\>.
\endeq
Since $V_2$ is 2-dimensional, the vector $v_2 = A_1(\phi_1)$ must be nonzero. 
An easy induction shows that, for $2\leq i\leq k$, the vector $v_i$ is nonzero.
For $k +1\leq i \leq c$, the stability condition on $V_i$ is then
\eq
V_i = \Im A_{i-1} + \Im \Gamma_{n-k \to i} = A_{i-1}(V_{i-1}).
\endeq
Since both $\dim V_{i} = \dim V_{i-1} = k$, the map $A_{i-1}$ is of full rank, and hence $v_i \neq 0$ for $k+1\leq i \leq c$.
Therefore, we have seen that the assumption that $\Gamma_{n-k \to c} = 0$ leads to a contradiction, and hence $\Gamma_{n-k \to c} \neq 0$.
A similar argument shows that $\Gamma_{k \to c}\neq 0$.
The base case is proved.

Next, we are to show that \eqref{eq:a23} holds for $b\leq t < c-1$.
We write for short $h = c - t - 1$ to denote the size of the ``hook'' in the map $A_{t+1 \to c}\Gamma_{n-k \to t+1}$.
For example, as shown in the figure below, the maps $\Gamma_{\to c}$ have hook size 0, the maps $A_{c-1}\Gamma_{\to c-1}$ have hook size 1, and so on:
\[ 
\begin{array}{ccc}
\xymatrix@C=25pt@R=25pt{
&
{D_{j}}  
	\ar@/_/[d]_{\Gamma_{j}}
\\
{V_{c}}
&
	\ar@/_/[l]	
{V_{j}}  
}
&
\xymatrix@C=25pt@R=25pt{
&&
{D_{j}}  
	\ar@/_/[d]_{\Gamma_{j}}
\\
{V_{c-1}}
	\ar@/_/[r]_{A_{c-1}}
&
{V_{c}}
&
	\ar@/_/[ll]_{B_{j\to c-1}}
{V_{j}}  
}
&
\xymatrix@C=25pt@R=25pt{
&&
{D_{j}}  
	\ar@/_/[d]_{\Gamma_{j}}
\\
{V_{c-2}}
	\ar@/_/[r]_{A_{c-1}A_{c-2}}  
&
{V_{c}}
&
	\ar@/_/[ll]_{B_{j\to c-2}}
{V_{j}.}  
}
\\
\\
h=0
&
h=1
&
h=2
\end{array}
\] 
Note that $h$ is strictly less than the size $c-b$ of the cup.
Our strategy is to construct nonzero vectors $v_{i} \in V_{i} \cap \Im A_{h+1\to i}$ for $h+1\leq i \leq c$.
If this claim holds, then by admissibility conditions from $V_{h+1}$ to $V_i$ and an induction on $i$, we have
\eq
B_{i \to i-h-1}(v_i) = 
B_{i \to i-h-1} A_{i-1}( v_{i-1}) =  
A_{i-2} B_{i-1 \to i-h-2}( v_{i-1}) = 0.
\endeq
Note that the initial case holds since
$B_{h+1 \to 0}(v_{h+1}) = 0$.
Hence, there is a nonzero vector $v_b \in \ker B_{b\to b-h-1} \subset \ker B_{b\to a}$.
On the other hand, $v_b \not\in \ker A_{b\to c}$ since $A_{b\to c}(v_b) = v_c \neq 0$, a  contradiction.
We can now prove the claim.
Suppose first that $A_{t+1 \to c}\Gamma_{n-k \to t+1} = 0$.
By the admissibility conditions from $V_{t+2}$ to $V_{n-k+h-1}$, we have
\eq\label{eq:admt}
0 = A_{t+1 \to c}\Gamma_{n-k \to t+1} = 
B_{n-k+h\to c}A_{n-k \to n-k+h}\Gamma_{n-k},
\endeq
which can be visualized from the figure below by equating the two maps $D_{n-k}\to V_c$ represented by composing solid arrows and dashed arrows, respectively:
\[ 
\xymatrix@C=25pt@R=25pt{
&&
{D_{n-k}}  
	\ar@/^/@{.>}[d]^{\Gamma_{n-k}}
	\ar@/_/[d]_{\Gamma_{n-k}}
\\
{V_{t+1}}
	\ar@/_/[r]  
&
{V_{c}}
	\ar@/_/[l]  
&
	\ar@/^/@{.>}[l]
	\ar@/_/[l]	
{V_{n-k}}  
	\ar@/^/@{.>}[r]
	& 
	\ar@/^/@{.>}[l]
	{V_{n-k+h}.} 
}
\] 
For all $1 \leq i \leq t+1$, we will show now any map $D_{n-k} \to V_i$ with exactly a size $h$ ``hook'' is zero.
Precisely speaking, the admissibility conditions from $V_{i+1}$ to $V_{n-k+h-1}$ imply that
\eq
A_{i \to i+h}\Gamma_{n-k\to i}
=
B_{n-k+h \to i+h}A_{n-k \to n-k+h}\Gamma_{n-k},
\endeq 
which can be visualized as the figure below:
\[ 
\xymatrix@C=25pt@R=25pt{
&&
{D_{n-k}}  
	\ar@/^/@{.>}[d]^{\Gamma_{n-k}}
	\ar@/_/[d]_{\Gamma_{n-k}}
\\
V_i
	\ar@/_/[r]  
&
{V_{i+h}}
	\ar@/_/[l]  
&
	\ar@/^/@{.>}[l]
	\ar@/_/[l]	
{V_{n-k}}  
	\ar@/^/@{.>}[r]
& 
	\ar@/^/@{.>}[l]
	{V_{n-k+h}.}
}
\] 
It follows that $A_{i \to i+h}\Gamma_{n-k\to i} = 0$ since it is a composition of a zero map in \eqref{eq:admt}.
Since the hook size $h$ is less than the cup size $c-b$, which is less or equal to the total number $k$ of cups, we have
$h < k$ and so 
\eq
\dim V_h = h, 
\quad
\dim V_{h+1} = h+1.
\endeq
We claim that 
\eq\label{eq:claim1_t} 
\Im \Gamma_{k \to h+1} \neq 0.
\endeq
By the stability condition on $V_1$, we have 
\eq
V_1 = \Im \Gamma_{k\to 1} + \Im \Gamma_{n-k \to 1}.
\endeq
For the dimension reason, either $\Gamma_{k\to 1}$ or $\Gamma_{n-k\to 1}$ is nonzero. 
If $\Gamma_{k\to 1} \neq 0$ then $\Gamma_{k\to h+1} \neq 0$, and the claim follows.
If $\Gamma_{n-k\to 1} \neq 0$,
we define, for $1\leq i \leq l\leq n-k$, 
\eq
\epsilon_i = \epsilon^{(i)}_i = \Gamma_{n-k \to i}(e) \neq 0, 
\quad
\epsilon^{(i)}_l = A_{i\to l}(\epsilon_i).
\endeq
An easy induction shows that if $\epsilon^{(i)}_l = 0$ for some $1\leq i \leq l \leq h$ then $\Gamma_{k\to l} \neq 0$, which proves the claim. 
So we we now assume that $\epsilon^{(i)}_l \neq 0$ for all $1\leq i \leq l \leq h$.
Note that
\eq
\epsilon^{(1)}_{1+h}
 =
A_{1\to 1+h} \Gamma_{n-k \to 1}(e) = 0,
\endeq
and hence $\rk A_h = h-1$.
Now the  stability condition on $V_{h+1}$ implies that 
\eq
\dim V_{h+1} = \rk A_h + \rk \Gamma_{k\to h+1} + \rk \Gamma_{n-k\to h+1}, 
\endeq
and hence $\rk \Gamma_{k\to h+1} =1$.
The claim \eqref{eq:claim1_t} is proved.
Moreover, the vectors $\phi_i = \Gamma_{k\to i}(f) \in V_i$ are all nonzero for $h+1 \leq i \leq k$.
Define $v_i = A_{h+1 \to i}(\phi_{h+1})$ for all $i>h$.

The stability condition on $V_{h+2}$ now reads
\eq\label{eq:dimt}
V_{h+2} = \begin{cases}
\Im A_{h+1} + \Im \Gamma_{k\to h+2} + \Im \Gamma_{n-k\to h+2} 
&\tif h+2 \leq k, 
\\
\qquad \Im A_{h+1} + \Im \Gamma_{n-k\to h+2} 
&\tif h+2 > k.
\end{cases}
\endeq
In either case, a dimension argument similar to the one given in the base case $t=c-1$ shows that the vector $v_{h+2} = A_{h+1}(\phi_{h+1})$ must be nonzero since
\eq
\epsilon^{(i)}_{l}
 =
A_{i\to l} \Gamma_{n-k \to i}(e) = 0 
\quad
\textup{for}
\quad
l \geq i+h.
\endeq 
For $h+2\leq i\leq c$, a dimension argument using \eqref{eq:dimt} shows that
$v_i \neq 0$,
which leads to a contradiction, and hence $A_{t+1 \to c}\Gamma_{n-k \to t+1} \neq 0$.
A similar argument shows that $A_{t+1 \to c}\Gamma_{k \to t+1} \neq 0$.
The proposition is proved.
\endproof
\thm\label{thm:main1}
Recall $\Lambda^a$ from \eqref{defi:Springer_component_in_quiver}.
For any cup diagram $a \in B_{n-k,k}$,  we have an equality
\eq
p_{n-k,k}(\Lambda^a) = \tphi\inv(K^a).
\endeq
As a consequence, $\tphi\inv(\cB_{n-k,k}) = \bigcup_{a \in B_{n-k,k}} p_{n-k,k}(\Lambda^a)$.
\endthm
\proof
We have
\eq
\begin{aligned}
\tphi\inv(K^a)&=\widetilde{p}(\fT^a)\quad &&\textup{by Proposition}~\ref{prop:a1}
\\
&=p_{n-k,k}(\Phi\inv(\fT^a))\quad &&\textup{by Proposition}~\ref{prop:Phi}(b)(c)
\\
&=p_{n-k,k}(\Lambda^a) \quad  &&\textup{by Proposition}~\ref{prop:a2}.
\end{aligned}
\endeq
\endproof

\subsection{The ray condition}
For completeness, in this section we characterize the ray condition $F_i = x^{-\frac{1}{2}(i-\rho(i))}(x^{n-k-\rho(i)} F_n ),$ on the quiver representation side.
\prop\label{prop:ray}
Let $(x, F_\bullet) = \tphi(A,B,\Gamma,0) \in K^a \subset \tcS_{n-k,k}$.
Then the ray condition is equivalent to
\eq
\begin{cases}
\:\: B_iA_i = 0
&\tif c(i) \geq 1, 
\\
\Gamma_{n-k\to i} = 0
&\tif c(i) = 0,
\end{cases}
\endeq
where
$c(i) = \frac{i-\rho(i)}{2}$ is the total number of cups to the left of $i$.
\endprop
\proof
Write $\rho = \rho(i)$ and $c = c(i)$ for short.
By Corollary~\ref{cor:Fi}, the ray relation is equivalent to that
\eq\label{eq:rayblock}
\ker(A_{n-k-\rho+1 \to n} \Gamma_{\to n-k-\rho+1}|\ldots|A_{n-k \to n}\Gamma_{n-k})
=
\ker(A_{c+1 \to i}\Gamma_{\to c+1}|\ldots|\Gamma_{\to i}).
\endeq
Note that there are $\rho$ blocks on the left hand side of \eqref{eq:rayblock} and each block is a zero map; while there are $i -c = \rho + c \geq \rho$ blocks on the right hand side. 
Hence, the defining relations, by an elementary case-by-case analysis, are
\eq
\begin{cases}
A_{c+\rho \to i} \Gamma_{n-k\to c+\rho} = 0,
\quad
A_{c+\rho+1 \to i} \Gamma_{n-k\to c+\rho+1} \neq 0 
&\tif c \geq 1, 
\\
\qquad \qquad \qquad 
A_{c+\rho \to i} \Gamma_{n-k\to c+\rho} = 0 
&\tif c = 0.
\end{cases}
\endeq
Note that by definition, $i = \rho$ when $c = 0$.
We are done.
\endproof

\section{Springer fibers for classical types}
\subsection{Springer fibers and Slodowy varieties of type D}
From now on, let $n = 2m$ be an even positive integer.
Let $\pi^\textup{D}$ be the subset of all partitions of $2m$ whose even parts occur even times, i.e., 
\eq
\pi^\textup{D}
=\{\lambda= (\lambda_i)_i \vdash 2m \mid 
\#\{i \mid \lambda_i = j\} \in 2\ZZ
\tforall j\in 2\ZZ
\}.
\endeq
It is known (\cite{Wi37}) that $\pi^\textup{D}$ is in bijection with the $O_{n}(\CC)$-orbits of the type D nilpotent cone $\cN^\textup{D}$.
Now we fix $\ld$ to be a two-row partition in $\pi^\textup{D}$, and hence it is of the following form:
\eq\label{eq:partitionD}
\ld = (m,m),
\quad
\textup{or}
\quad
\ld = (n-k, k) \in (2\ZZ+1)^2.
\endeq
For each $\ld$ of the form as in \eqref{eq:partitionD}, 
we define a $n$-dimensional $\CC$-vector space $V_\ld$ with (ordered) basis
\eq\label{eq:basisVld}
\{ e^\ld_1, e^\ld_2, \ldots, e^\ld_{n-k}, f^\ld_1, f_2^\ld, \ldots f^\ld_k \},
\endeq
and a non-degenerate symmetric bilinear form $\beta_\ld: V_\ld \times V_\ld \to \CC$, 
 whose associated matrix, under the ordered basis \eqref{eq:basisVld}, is
\eq\label{eq:beta}
M^\ld=
\begin{cases}
\small
\begin{blockarray}{ *{8}{c} }
 & \{e^\ld_i\} & \{f^\ld_i\} \\
\begin{block}{ c @{\quad} ( @{\,} *{7}{c} @{\,} )}
\{e^\ld_i\}& 0 & J_m 
\\
\{f^\ld_i\}& J^t_m & 0
\\
\end{block}
\end{blockarray}~~~ 
\normalsize
&\:\:\tif n-k = k; 
\\
\small
\begin{blockarray}{ *{8}{c} }
 & \<e^\ld_i\> & \<f^\ld_i\> \\
\begin{block}{ c @{\quad} ( @{\,} *{7}{c} @{\,} )}
\<e^\ld_i\>& J_{n-k} & 0 
\\
\<f^\ld_i\>& 0 & J_k
\\
\end{block}
\end{blockarray}~~~
\normalsize
&\:\:\tif n-k > k,
\end{cases}
\quad
\textup{where}
\quad
J_i = 
\small
\begin{pmatrix}
&&&1
\\
&&-1&
\\
&\iddots&&
\\
(-1)^{i-1}
\end{pmatrix}
\normalsize.
\endeq
Usually we omit the superscripts when there is no ambiguity. Note that in Section~\ref{sec:irred} we shall see the need to keep the superscripts.

For each subspace $W$ of $V_\ld$ we let $W^\perp =W^\perp(\beta_\ld)$ be its orthogonal complement in $V_\ld$ with respect to $\beta_\ld$.
$W$ is called {\em isotropic} if $W \subseteq W^\perp$.
Denote the orthogonal Lie algebra corresponding to $\beta_\ld$ by
\eq\label{def:so}
\fso_{n}(\CC; \beta_\ld) = \{ X \in \fgl_n \mid M^\ld X = -X^tM^\ld \}.
\endeq
Let $\cBD = \cBD(\beta_\ld)$ be the flag variety of $O_{2m}(\CC)$ with respect to $\beta_\ld$, namely,
\eq
\cB^\textup{D}
= \{F_\bullet \in\cB \mid F_i = F_{n-i}^\perp \tforall  i\}.
\endeq
\rmk
It is intentional to use $O_{n}(\CC)$ rather than the special orthogonal group so that $\cBD$ has two connected components. 
\endrmk
Let $\tcND$ be the cotangent bundle of $\cBD$. Explicitly, we have
\eq
\tcND = T^*\cBD = \{(u,F_\bullet)\in \cND \times \cBD \mid u (F_i) \subseteq F_{i-1} \tforall i \}.
\endeq
The type D Springer resolution $\mu_\textup{D}: \tcND \to \cND$ is given by $(u, F_\bullet) \mapsto u$.
Given $x \in \cND$, the associated Springer fiber of type D is defined as the subvariety 
\eq
\cBD_x = \{ F_\bullet \in \cBD
\mid
x F_i \subseteq F_{i-1}
\tforall i
\}. 
\endeq
The type D Springer fiber only depends (up to isomorphism) on the $O_n$-orbit containing $x$. 
For $x\in \cND$,
we denote the
type D Slodowy transversal slice 
of (the $O_n$-orbit of) 
$x$ by the variety 
\eq
\cSD_x =\{ u\in \cND \mid [u-x,y]=0\}, 
\endeq
where $(x,y,h)$ is a $\fsl_2$-triple in $\cND$.  
Let the type D Slodowy variety associated to $x$ be defined as
\eq
\tcSD_{x}=\mu_\textup{D}\inv(\cSD_{x}) = \{ (u, F_\bullet) \in \cND\times \cBD \mid [u-x,y]=0,  \:\:   u (F_i) \subseteq F_{i-1} \textup{ for all }i \}.
\endeq  
If $x \in \cND$ is of Jordan type $(n-k,k)$, we write
\eq
\cBD_{n-k,k}:= \cBD_{x},
\quad
\cSD_{n-k,k} := \cSD_{x},
\quad
\tcSD_{n-k,k} := \tcSD_{x}.
\endeq
\subsection{Irreducible components of type D Springer fibers}\label{sec:mcup}
In order to describe the irreducible components of $\cBD_{n-k,k}$, we define the notion of a  marked cup diagram.

\begin{definition}
A {\it marked cup diagram} is a cup diagram (see \eqref{eq:cup}) in which each cup or ray may be decorated with a single marker satisfying the following rules:
\begin{enumerate}[($M1$)]
\item 
The vertices on the axis are labeled by $1, 2, \ldots, m$,
\item
A marker can be connected to the right border  of the rectangular region by a path which does not intersect any other cup or ray.  
\end{enumerate}
Given a cup diagram $\da\in \BD_{n-k,k}$,
 denote the sets of all vertices connected to the left (resp.\ right) endpoint of a marked cup in $\da$ by $X_l^{\da}$ (resp.\ $X_r^{\da}$); while the sets of vertices connected to the left (resp.\ right) endpoint of an unmarked cup in $\da$ are denoted by
$V_l^{\da}$ (resp.\ $V_r^{\da}$).
The endpoint-swapping map $\sigma$ and the size formula $\delta$ naturally extend to $X_l^{\da} \sqcup V_l^{\da}$.

Now we define an auxiliary  set $\tB_{n-k,k}$ containing $B_{n-k,k}$ such that the cups and rays in any $a \in \tB_{n-k,k}$ can be decorated by markers.
A marked cup diagram on $m$ vertices 
is obtained from folding a centro-symmetric cup diagram on $n = 2m$ vertices in the following sense.


\alg\label{alg:fold}
Given $a\in \tB_{n-k,k}$ that has cups crossing the axis of reflection, in the following we demonstrate how to produce two diagrams $a', a^-  \in \tB_{n-k,k}$. 
\begin{enumerate}
\item If $a$ has exactly one cup connecting that crosses the axis of reflection, set $a'$ to be the diagram obtained from $a$ by replacing the cup by two (unmarked) rays connected to the endpoints of the cup, respectively; while $a^-$ is obtained similarly but with markers.
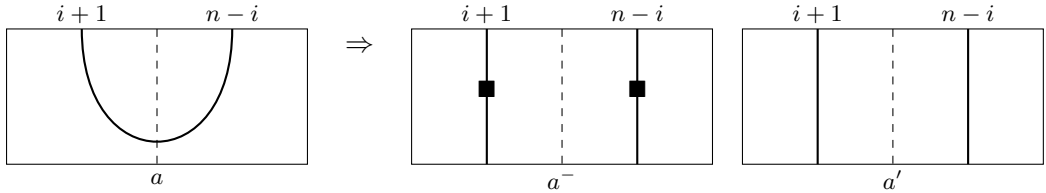
\[
\begin{tikzpicture}[baseline={(0,-.3)}]
\draw (-.5,0) -- (3.5,0) -- (3.5,-1.8) -- (-.5,-1.8) -- cycle;
\draw[dashed] (1.5,0)--(1.5,-1.8);
\begin{footnotesize}
\node at (0.5,.2) {$i+1$};
\node at (2.5,.2) {$n-i$};
\node at (1.5, -2) {$a$};
\end{footnotesize}
\draw[thick] (0.5,0) .. controls +(0,-2) and +(0,-2) .. +(2,0);
\end{tikzpicture}
\quad
\Rightarrow
\quad
\begin{tikzpicture}[baseline={(0,-.3)}]
\draw (-.5,0) -- (3.5,0) -- (3.5,-1.8) -- (-.5,-1.8) -- cycle;
\draw[dashed] (1.5,0)--(1.5,-1.8);
\begin{footnotesize}
\node at (0.5,.2) {$i+1$};
\node at (2.5,.2) {$n-i$};
\node at (.5, -0.8) {$\blacksquare$};
\node at (2.5, -0.8) {$\blacksquare$};
\node at (1.5, -2) {$a^-$};
\end{footnotesize}
\draw[thick] (0.5,0) -- (0.5, -1.8);
\draw[thick] (2.5,0) -- (2.5, -1.8);
\end{tikzpicture}
\quad
\begin{tikzpicture}[baseline={(0,-.3)}]
\draw (-.5,0) -- (3.5,0) -- (3.5,-1.8) -- (-.5,-1.8) -- cycle;
\draw[dashed] (1.5,0)--(1.5,-1.8);
\begin{footnotesize}
\node at (0.5,.2) {$i+1$};
\node at (2.5,.2) {$n-i$};
\node at (1.5, -2) {$a'$};
\end{footnotesize}
\draw[thick] (0.5,0) -- (0.5, -1.8);
\draw[thick] (2.5,0) -- (2.5, -1.8);
\end{tikzpicture}
\]
\item Otherwise, set $a'$ to be the diagram obtained from $a$ by replacing the innermost two (nested) cups that cross the axis by two side-by-side cups without markers; while $a^-$ is obtained similarly but with markers. 
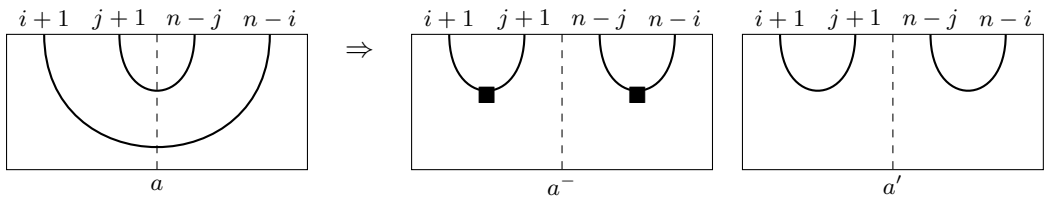
\[
\begin{tikzpicture}[baseline={(0,-.3)}]
\draw (-.5,0) -- (3.5,0) -- (3.5,-1.8) -- (-.5,-1.8) -- cycle;
\draw[dashed] (1.5,0)--(1.5,-1.8);
\begin{footnotesize}
\node at (0,.2) {$i+1$};
\node at (1,.2) {$j+1$};
\node at (2,.2) {$n-j$};
\node at (3,.2) {$n-i$};
\node at (1.5, -2) {$a$};
\end{footnotesize}
\draw[thick] (1,0) .. controls +(0,-1) and +(0,-1) .. +(1,0);
\draw[thick] (0,0) .. controls +(0,-2) and +(0,-2) .. +(3,0);
\end{tikzpicture}
\quad
\Rightarrow
\quad
\begin{tikzpicture}[baseline={(0,-.3)}]
\draw (-.5,0) -- (3.5,0) -- (3.5,-1.8) -- (-.5,-1.8) -- cycle;
\draw[dashed] (1.5,0)--(1.5,-1.8);
\begin{footnotesize}
\node at (0,.2) {$i+1$};
\node at (1,.2) {$j+1$};
\node at (2,.2) {$n-j$};
\node at (3,.2) {$n-i$};
\node at (.5, -0.8) {$\blacksquare$};
\node at (2.5, -0.8) {$\blacksquare$};
\node at (1.5, -2) {$a^-$};
\end{footnotesize}
\draw[thick] (0,0) .. controls +(0,-1) and +(0,-1) .. +(1,0);
\draw[thick] (2,0) .. controls +(0,-1) and +(0,-1) .. +(1,0);
\end{tikzpicture}
\quad
\begin{tikzpicture}[baseline={(0,-.3)}]
\draw (-.5,0) -- (3.5,0) -- (3.5,-1.8) -- (-.5,-1.8) -- cycle;
\draw[dashed] (1.5,0)--(1.5,-1.8);
\begin{footnotesize}
\node at (0,.2) {$i+1$};
\node at (1,.2) {$j+1$};
\node at (2,.2) {$n-j$};
\node at (3,.2) {$n-i$};
\node at (1.5, -2) {$a'$};
\end{footnotesize}
\draw[thick] (0,0) .. controls +(0,-1) and +(0,-1) .. +(1,0);
\draw[thick] (2,0) .. controls +(0,-1) and +(0,-1) .. +(1,0);
\end{tikzpicture}
\]
\end{enumerate}
\endalg
\begin{definition}\label{def:unfold}
Let $\leq$ be the partial order on $\tB_{n-k,k}$ induced by $a' \leq a$ and $a^- \leq a$ using transitivity and reflexivity for all $a \in \tB_{n-k,k}$ and $a', a^-$ as in Algorithm~\ref{alg:fold}.
For a marked cup diagram $\da$, let $\ddot{a} \in \tB_{n-k,k}$ be the diagram obtained by placing the mirror image of $\da$ to the right of $\da$.
A diagram $\da$ is said to {\em unfold} to $b \in B_{n-k,k}$ if $\ddot{a} \leq b$.
\end{definition}

We write $\BD_{n-k,k}$ to denote the set of all marked cup diagrams on $m$ vertices with exactly $\lfloor\frac{k}{2}\rfloor$ cups.

\end{definition}  

\begin{ex}
Below we describe in below all six marked cup diagrams $\da_i (1\leq i \leq 6)$ in $\BD_{3,3}$ having $\lfloor\frac{3}{2}\rfloor = 1$ cup:
For the centro-symmetric cup $a_{135} = \{\{1,2\}, \{3,4\}, \{5,6\}\}$, we have
\[
a_{135} = \fcupd,
\quad
a'_{135} = \fcupdp = \ddot{a}_1,
\quad
a^-_{135} = \fcupdm = \ddot{a}_2,
\]
where
\[
\da_1 = \mcupa,
\quad
\da_2 = \mcupd.
\]
For $a_{124} = \{\{1,6\}, \{2,3\}, \{4,5\}\}$, we have
\[
a_{124} = \fcupf,
\quad
a'_{124} = \fcupfp = \ddot{a}_3,
\quad
a^-_{124} = \fcupfm = \ddot{a}_4,
\]
where
\[
\da_3 = \mcupb,
\quad
\da_4 = \mcupf.
\]
Finally, for $a_{123} = \{\{1,6\}, \{2,5\}, \{3,4\}\}$, we have
\[
\begin{split}
&a_{123} = \fcupc,
\quad
a'_{123} = a_{124},
\quad
a^-_{123} = \fcupcm,
\\
&(a^-_{123})'=\fcupcmp= \ddot{a}_5,
\quad
(a^-_{123})^- =\fcupcmm= \ddot{a}_6,
\end{split}
\]
where
\[
\da_5 = \mcupe,
\quad
\da_6 = \mcupc.
\]
In this case, $\da_3$ unfolds not only to $a_{124}$ but also to $a_{123}$ because $\ddot{a}_3 \leq a_{124} \leq a_{123}$.

Here we use dashed line as the right border to emphasize that it is the axis of reflection, as well as that the markers must be accessible from the dashed line by a path that does not intersect any other rays or cups.
\end{ex}
\prop
There exists a bijection between the irreducible components of the Springer fiber $\cBD_{n-k,k}$ and the set $\BD_{n-k,k}$ of marked cup diagrams.
\endprop
\proof
It follows from combining \cite[Lemma~5.12]{ES16}, \cite[II.9.8]{Spa82} 
and \cite[Lemmas~3.2.3, 3.3.3]{vL89}.
\endproof
\rmk
Note that people only knew there is a bijection. That being said, given an irreducible component $K$ of $\cBD_{n-k,k}$, it is unclear whether or not  there is the most natural marked cup diagram $\da$ assigned to this irreducible component $K$.
For the rest of the paper, we will construct a subvariety $K^{\da} \subseteq \cBD_{n-k,k}$
for each marked cup diagram $\da \in \BD_{n-k,k}$ (see \eqref{def:LDn-kk}), and prove that they are indeed irreducible components in Section~\ref{sec:irred}.
\endrmk
In the examples below we demonstrate a direct computation to determine irreducible components for the base case $\ld =(1,1)$ or $(2,2)$.  
\ex\label{ex:Ka11}
Let $n=2, k=1$. We fix a basis $\{e_1,f_1\}$ of $\CC^2$ so that $x = 0$. 
According to Algorithm~\ref{alg:fold}, the marked cup diagrams in $\BD_{1,1}$ all come from the folding of $a=\cupaaa$, and hence they are
\[
\da = \mcupaaa~,~ \dot{b} = \mcupaab.
\]
The (type A) irreducible component $K^a$ is the entire Springer fiber $\cB_{1,1} \simeq \{F_1 = \<\ld e_1+\mu f_1\>\}$. By imposing the isotropy condition $F_1^\perp = F_1$ with respect to the symmetric bilinear form $\beta_{1,1}$ satisfying that $\beta_{1,1}(e_1, f_1) = 1, \beta_{1,1}(e_1,e_1) = \beta_{1,1}(f_1,f_1)=0$, 
we see that
\[
0 = \beta_{1,1}(\ld e_1+ \mu f_1 , \ld e_1+ \mu f_1) = 2\ld\mu.
\]
Therefore, it is either $\ld=0$ or $\mu=0$. Hence, there are only two isotropic flags in $\cBD_{1,1}$:
\[
(0 \subset \<e_1\> \subset \CC^2)
,
\quad
(0 \subset \<f_1\> \subset \CC^2).
\] 
We shall denote $K^{\da} = \{(0 \subset \<e_1\> \subset \CC^2)\}$ since it satisfies the ray relation in type A plus that $\da$ is a type A ray. 
We then have no choice but to denote $K^{\dot{b}} = \{(0 \subset \<f_1\> \subset \CC^2)\}$.
We remark that the description of $F_1$ here eventually leads to the marked ray relation in  Theorem~\ref{thm:main2a}(iii). 
\endex
\ex\label{ex:Ka22}
Let $n=4, k=2$. We fix a basis $\{e_1, e_2, f_1, f_2\}$ of $\CC^4$ so that $x$ is determined by
$e_2 \mapsto e_1 \mapsto 0, f_2 \mapsto f_1 \mapsto 0$.
Define
\[
a_{12} = \cupaa~,~ a_{13}= \cupab~.
\]
According to Algorithm~\ref{alg:fold}, the corresponding marked cup diagrams are
\[
\da_{13} = \mcupab~,~\da_{12} = \mcupaa.
\]
The (type A) irreducible components in $\cB_{2,2}$ are 
\[
\begin{split}
K^{a_{13}} 
&= \{F_\bullet ~|~ x\inv F_0 = F_2, x\inv F_2 = F_4, \dim F_i = i\}
\\
&= \{(0 \subset F_1 \subset \<e_1, f_1\> \subset F_3 \subset \CC^4) ~|~ \dim F_1 = 1, \dim F_3=3\},
\\
K^{a_{12}} 
&= \{F_\bullet ~|~ x^{-2} F_0 = F_4, x\inv F_1 = F_3, \dim F_i = i\}
\\
&= \{ F_\bullet ~|~ x\inv F_1 = F_3, \dim F_i = i\}. 
\end{split}
\]
By imposing the isotropy condition with respect to the matrix
\[
\begin{blockarray}{ *{8}{c} }
& e_1 &   e_2 & f_1 & f_2\\
\begin{block}{ c @{\quad} ( @{\,} *{7}{c} @{\,} ) }
e_1 & 0 & 0 & 0& 1 \\
e_2 & 0 & 0 & -1 & 0   \\
f_1 & 0 & -1 & 0 &0\\
f_2 & 1 & 0 & 0 &0\\
\end{block}
\end{blockarray}~~~~~,
\]
we see that the isotropic flags in $\cBD_{2,2}$ os a disjoint union $K_1\sqcup K_2$, where
\eq\label{eq:K13}
\begin{split}
K_1
&=\{ (0 \subset \<\ld e_1 + \mu f_1\> \subset \<e_1, f_1\> \subset \<e_1,f_1, \ld e_2 + \mu f_2\>\subset \CC^4) \in \cBD_{2,2}\}
\\
&= \{F_\bullet \in \cBD_{2,2}~|~ x F_2 = 0\},
\end{split}
\endeq
\eq\label{eq:K12}
\begin{split}
K_2
&=\{(0 \subset \<\ld e_1 + \mu f_1\> \subset \<\ld e_1 + \mu f_1, \ld e_2 + \mu f_2\> \subset \<e_1,f_1, \ld e_2 + \mu f_2\>\subset \CC^4) \in \cBD_{2,2}\}
\\
&= \{ F_\bullet \in \cBD_{2,2}~|~ x F_2 = F_1\}.
\end{split}
\endeq
We assume for now that $K_1$ and $K_2$ are irreducible (a proof will be given in Lemma~\ref{lem:n<=4}).
Since the condition in $K_1$ coincides with the cup relation in type A, we call this irreducible component $K^{\da_{13}}$ because  $\da_{13}$ is a type A cup.
We then have no choice but to identify $K_2 \equiv K^{\da_{12}}$.
We remark that the new relation $x F_2 = F_1$ here eventually leads to the marked cup relation in Theorem~\ref{thm:main2a}(ii). 
\endex

\section{Fixed point subvarieties of quiver varieties}
\subsection{Automorphisms on quiver varieties}
In {the} literature the fixed point subvarieties of Nakajima's quiver varieties are studied  by Henderson--Licata in \cite{HL14} for the type A quivers associated with an explicit involution. 
For quivers associated with the symmetric pairs (or Satake diagrams), the corresponding fixed-point subvarieties  are studied by Li in \cite{Li19}. 
While it is difficult to compare the two automorphisms in an explicit way,
both automorphism restrict to an isomorphism between the fixed-point subvariety and a Slodowy variety of type D which is isomorphic to one of type C. 

In \cite{HL14}, Henderson--Licata consider any diagram automorphism $\theta$ which is an admissible automorphism in the sense of Lusztig, see \cite[\S 12.1.1]{L93}.
For type A quiver varieties, such a diagram automorphism defines a variety automorphism 
$\Theta = \Theta(\theta, \sigma_k): M_{n-k,k} \to M_{n-k,k}$,
which also depends on the choice of an involution $\sigma_k: D_k \to D_k$.
With a suitable choice of $\sigma_k$, there is an isomorphism $M^\Theta_{n-k,k} \simeq \tcSD_{n-k,k} \sqcup \tcSD_{n-k,k}$.

On the other hand, Li has constructed in \cite{Li19} a family of automorphisms which works in a more general scenario. The automorphism $\sigma = a \circ S_{\omega} \circ \tau$ is composed of a diagram automorphism $a$, a reflection functor $S_{\omega}$ for some Weyl group element $\omega$, and an isomorphism $\tau$ which deals with the so-called formed spaces which account for the orthogonality in our setup.
By choosing $a=1, \omega = w_\circ$, which is the longest element, he exhibits an isomorphism
$M^\sigma_{n-k,k} \simeq \tcS^{\mathfrak{o}}_{n-k,k} \equiv \tcSD_{n-k,k} \sqcup \tcSD_{n-k,k}$.

Moreover, in both cases the isomorphisms restrict to isomorphisms between the Springer fibers.
In this section we investigate to what extent they preserve the components of Springer fibers. 

\subsection{Henderson--Licata's fixed point subvarieties}
We use identifications $V_i \equiv V_{n-i}$ for all $i$ and $D_i \equiv D_{n-i}$ for all $i\neq k$ (they are referred {to} as isomorphisms $\varphi_i$ and $\sigma_i$ in \cite[\S 3.2]{HL14}).
For all $k$, let $\sigma_k$ be the automorphism on $D_k$ determined by 
\[
\sigma_k(e) = e,
\quad
\sigma_k(f) = \begin{cases}
\:\:\: f &\tif k \in2\ZZ, 
\\
-f&\tif k \in2\ZZ+1.
\end{cases}
\]
Define the $\Theta$-action by
\eq\label{eq:theta}
\begin{split}
&
\Theta(\Delta_i) = \Delta_{n-i},
\quad
\Theta(\Gamma_i) = \begin{cases}
\:\:\:  \Gamma_{n-i} &\tif i \neq k, 
\\
\Gamma_{k} \circ \sigma_k\inv &\tif i=k,
\end{cases}
\\
&
\Theta(A_i) = B_{n-1-i},
\quad
\Theta(B_i) = A_{n-1-i}.
\end{split}
\endeq
\rmk\label{rmk:fix} 
By \cite[\S 3.2]{HL14}, for all $s \in \Lambda^+_{n-k,k}$, $\Theta$ sends the orbit containing $s$ to the orbit containing $\Theta(s)$, or, $\Theta(p(s)) = p(\Theta(s))$, where $p$ is the projection map in \eqref{def:p}.  
Therefore, an element $[s]  = [(A,B,\Gamma, 0)] \in M^\cB_{n-k,k}$ is fixed under $\Theta$ if and only if there exists a $g = (g_i)_i \in \GL(V)$ such that
\eq\label{eq:fix}
A_i = g_{i+1} B_{n-1-i} g_i\inv,
\quad
B_i = g_i A_{n-1-i} g_{i+1}\inv,
\quad
\Gamma_i = 
\begin{cases}
g_i\Gamma_{n-i} &\tif i \neq k, 
\\
g_k\Gamma_{k} \sigma_k &\tif i=k.
\end{cases}
\endeq
\begin{prop}\label{prop:HL}
Let $\Theta$ be defined as in \eqref{eq:theta}, and let $n, k$ be such that $(n-k,k)$ is a type D partition (see \eqref{eq:partitionD}),
then
there is an isomorphism
$M^\Theta_{n-k,k} \simeq \tcSD_{n-k,k}\sqcup \tcSD_{n-k,k}$. 
\end{prop}
\proof
See \cite[(5.2), Thm.~5.3]{HL14}.
\endproof

\begin{remark}
In \cite[Thm.~5.3]{HL14}, it is also given an isomorphism for the type C Slodowy's variety if the condition \cite[(5.1)]{HL14} is satisfied. 
Here, we only discuss {the} type D result. 
\end{remark}

\subsection{Two Jordan blocks of equal size}
The purpose of this subsection is to obtain a potential marked cup relation by working out a small rank example using fixed point subvarieties.
With type D in mind, we need $n=2k$. Furthermore, 
the action of $\Theta$ on the quiver representations in $\Lambda^\cB_{k,k}$ can be visualized as: 
\[ 
\xymatrix@C=18pt@R=9pt{
&
& &    
\ar[dd]_<(0.2){\Gamma_k}
D_k
& 
\\ & 
\\ 
&
{V_1}  
	\ar@/^/[r]^{A_1} 
	& 
	\ar@/^/[l]^{B_1}
	\cdots
	\ar@/^/[r]^{A_{k-1}} 
	&   
	\ar@/^/[l]^{B_{k-1}} 
{V_k}  
	\ar@/^/[r]^{A_{k}}
	&   
	\ar@/^/[l]^{B_k}
	\cdots
	\ar@/^/[r]^{A_{2k-2}}
	&
	\ar@/^/[l]
{V_{2k-1}} 
\ar@/^/[l]^{B_{2k-2}}  
}
\overset{\Theta}{\mapsto}
\xymatrix@C=18pt@R=9pt{
& &    
\ar[dd]_<(0.2){\Gamma_k}
D_k
  &
\\  
\\ 
{V_{2k-1}}  
	\ar@/^/[r]^{B_{2k-2}} 
	& 
	\ar@/^/[l]^{A_{2k-2}}
	\cdots
	\ar@/^/[r]^{B_{k}} 
	&   
	\ar@/^/[l]^{A_{k}} 
{V_k}  
	\ar@/^/[r]^{B_{k-1}}
	&   
	\ar@/^/[l]^{A_{k-1}}
	\cdots
	\ar@/^/[r]^{B_{1}}
	&
	\ar@/^/[l]
{V_{1}.} 
\ar@/^/[l]^{A_{1}}  
}
\] 
Note that by Corollary~\ref{cor:Spr} all $\Delta$ maps are zero, and hence omitted in our notation/picture.
In this section we demonstrate how we obtained the marked cup relation by working out a few non-trivial examples.
\ex
Let $n=4=2k$.
Consider $\Lambda^\cB_{2,2} = \{ (A,B,\Gamma,\Delta) \in \Lambda^+_{2,2} \mid \Delta = 0\}$, where the quiver representations are described below: 

\begin{equation}\label{eq:22-quiver} 
\xymatrix@C=18pt@R=9pt{
&
\ar@/_/[dd]_<(0.2){\Gamma_2}
D_2 
&   
\\   
& &    
\\ 
V_1 
	\ar@/^/[r]^{A_1}
	& 
	\ar@/^/[l]^{B_1} 
V_2  
	\ar@/_/[uu]_>(0.8){\Delta_2=0} 
	\ar@/^/[r]^{A_{2}}
	& 
	\ar@/^/[l]^{B_2}
V_3. 
	}
\end{equation} 
Define
\[
a_{12} = \cupaa~,~ a_{13}= \cupab~,~
\da_{12} = \mcupaa~,~\da_{13} = \mcupab~.
\]
By \eqref{defi:Springer_component_in_quiver}, we have
\begin{align}
\Lambda^{a_{12}}&=\{(A,B,\Gamma,0) \in \Lambda^+_{2,2}\mid \ker B_1 = \ker A_2\},
\\
\Lambda^{a_{13}}&=\{(A,B,\Gamma,0) \in \Lambda^+_{2,2}\mid A_1 = 0,  B_2 = 0\}.
\end{align}
By Remark~\ref{rmk:fix}, the fixed-points in $(M^{a_{12}})^\Theta$ 
and $(M^{a_{13}})^\Theta$
are described below: $[s] \in (M^{a_{12}})^\Theta$ if and only if
\eq
\begin{split}
&s=(A,B,\Gamma,0) \in \Lambda^{+}_{2,2},
\quad \ker B_1 = \ker A_2, \textup{ there is a }g=(g_i)_i \in \GL(V) \textup{ such that }
\\
&
\Gamma_2 = g_2 \Gamma_2,
\quad
A_1 = g_2 B_2 g_1\inv,
\quad
A_2 = g_3 B_1 g_2\inv,
\quad
B_1 =g_1 A_2 g_2\inv,
\quad 
B_2 = g_2 A_1 g_3\inv,
\end{split}
\endeq
while $[s] \in (M^{a_{13}})^\Theta$ if and only if
\eq
\begin{split}
s &\in \Lambda^{+}_{2,2},
\quad
A_1 = 0 = B_2,
\quad
\textup{ there is a }g \in \GL(V) \textup{ such that }
\\
\Gamma_2 &= g_2 \Gamma_2,
\quad
A_2 = g_3 B_1 g_2\inv,
\quad
B_1 =g_1 A_2 g_2\inv.
\end{split}
\endeq
Using Proposition~\ref{prop:MS}, the component 
$K^{\da_{13}}= \tphi(p((\Lambda^{a_{13}})^\Theta))$  
consists of pairs $(x, F_\bullet)$,
where $F_i = \ker \tGa_{1\to i}$, and
\[
x = \tDe_1\tGa_1 = \begin{pmatrix}0&0&1&0 \\ 0&0&0&1 \\ 0&0&0&0\\0&0&0&0\end{pmatrix}.
\]
From Example~\ref{ex:Fi}, the isotropic flag $F_\bullet$ is described by,
\eq
F_1 =  \ker (B_1 \Gamma_2 ), 
\quad
F_2 =  \ker (A_1B_1 \Gamma_2|  \Gamma_2), 
\quad
F_3  = \ker (A_2A_1B_1 \Gamma_2 |  A_2\Gamma_2). 
\endeq
Since $F_1$ is a 1-dimensional space which is killed by $x$, it must be spanned by a nonzero vector of the form $\ld e_1 + \mu f_1$ for some $\ld, \mu \in \CC$.  
For $F_3$, we have $\ker A_2 A_1 B_1 \Gamma_2 = \<e,f\>$ since the admissibility conditions imply that
\eq
A_2 (A_1 B_1) \Gamma_2
=
(A_2 B_2) A_2 \Gamma_2
= 0.
\endeq
Since the first block is a zero map, the space $\ker ( 0 |  A_2\Gamma_2)$ is a direct sum
$\<e_1, f_1\> \oplus \iota_2(\ker A_2\Gamma_2)$, where $\iota_t$ is the subscript-decorating linear map sending $e, f$ to $e_t, f_t$, respectively.
Since $A_2 = g_3B_1 g_2\inv$ and $\Gamma_2 = g_2 \Gamma_2$,
we have 
\eq
\ker A_2\Gamma_2 = \ker g_3 B_1\Gamma_2 = \ker B_1\Gamma_2 = \<\ld e + \mu f\>,
\endeq 
and hence
$F_3 = \<e_1, f_1, \ld e_2+ \mu f_2\>$.

For $F_2$, there are two cases:
$A_1 = B_2=0$ if $p_{2,2}(A,B,\Gamma,0)\in (M^{a_{13}})^\Theta$, or 
$A_1 = g_2 B_2 g_3\inv \neq 0$ if $p_{2,2}(A,B,\Gamma,0) \in (M^{a_{12}})^\Theta \setminus (M^{a_{13}})^\Theta$.
If $A_1 = 0$, then $\ker A_1B_1\Gamma_2 = \<e,f\>$ and hence $F_2 = \<e_1, f_1\>$.

If $A_1 \neq 0$, then $A_1$ is injective since its domain is one-dimensional, and hence 
\eq
\ker A_1B_1\Gamma_2 
= \ker B_1\Gamma_2
= \ker (g_1 A_2 g_2\inv) g_2 \Gamma_2
= \ker A_2\Gamma_2 
= \<\ld e + \mu f\>.
\endeq
Moreover, we have $\ker\Gamma_2 = \ker A_1 B_1 \Gamma_2$
since $\ker \Gamma_2 \subset \ker A_1 B_1 \Gamma_2$ and $\dim\ker \Gamma_2 = \dim\ker A_1 B_1 \Gamma_2 =1$.
Therefore, $F_2 = \<\ld e_1 + \mu f_1, \ld e_2 + \mu f_2\>$.

Comparing this example with Example~\ref{ex:Ka22}, we see that
\eq
\tphi((M^{a_{13}})^\Theta) = K^{\da_{13}},
\quad
\tphi((M^{a_{12}})^\Theta) = K^{\da_{13}}
\sqcup
K^{\da_{12}}.
\endeq
Here we label $K^{\da_{13}}$ by a{n} unmarked cup since it is characterized by
\eq
A_1 = B_2 = 0 \Leftrightarrow \ker B_2 = \ker A_3,
\endeq
which is the type A cup relation connecting the vertices 2 and 3, while $K^{\da_{12}}$ is characterized by the relation
\eq
A_1, B_2 \neq 0 \Leftrightarrow \ker B_1\Gamma_2 = \ker A_1B_1 \Gamma_2.
\endeq
\endex
\subsection{Two Jordan blocks of unequal sizes}
The purpose of this subsection is to obtain a potential marked ray relation by working out a small rank example using fixed point subvarieties.
Thanks to Corollary~\ref{cor:Spr}, the quiver representations in $\Lambda^\cB_{n-k,k}$ are of the form
\[ 
\xymatrix@C=18pt@R=9pt{
\dim D_i&0&0&\cdots &1&0&\cdots &1&\cdots &0&0
\\
&
& &  &   
\ar[dd]_<(0.2){\Gamma_k}
D_k
& &  &  
\ar[dd]_<(0.2){\Gamma_{n-k}}
D_{n-k}
 & & &&
\\ & & & & & & &  && 
\\ 
&
{V_1}  
	\ar@/^/[r]^{A_1} 
	& 
	\ar@/^/[l]^{B_1}
{V_2}
	\ar@/^/[r]^{A_2}
	&
	\ar@/^/[l]^{B_2}
	\cdots   
	\ar@/^/[r]
	& 
	\ar@/^/[l]
{V_k}  
	\ar@/^/[r]^{A_{k}}
	&   
{V_{k+1}}  
	\ar@/^/[l]^{B_k}
	\ar@/^/[r]
	&
	\ar@/^/[l]
	\cdots
	\ar@/^/[r]
	&  
{V_{n-k}} 
	\ar@/^/[r]^{A_{n-k}}
	\ar@/^/[l]
	& 
	\ar@/^/[l]^{B_{n-k}}
	\cdots
	\ar@/^/[r]
	&
	\ar@/^/[l]
{V_{n-2}}  
	\ar@/^/[r]^{A_{n-2}}
	& 
{V_{n-1}.} 
\ar@/^/[l]^{B_{n-2}}  
\\
\dim V_i & 1&2&\ldots&k&k&\ldots&k&\ldots&2&1
}
\] 
Below we demonstrate how the leftmost ray is characterized by working out the case $(3,1)$.

\ex\label{ex:31}
Let $n=4, k=1$.
Consider $\Lambda^\cB_{3,1} = \{ (A,B,\Gamma,\Delta) \in \Lambda^+_{3,1} \mid \Delta = 0\}$, where the quiver representations are described below (with fixed basis elements $v_1, w_i, v_3, e, f$):

\begin{equation}\label{eq:22-quiver} 
\xymatrix@C=18pt@R=9pt{
\ar[dd]_<(0.2){\Gamma_1}
\<f\>
&
&   
\ar[dd]^<(0.2){\Gamma_3.}
\<e\>
\\   
& &    
\\ 
\CC
	\ar@/^/[r]^{A_1}
	& 
	\ar@/^/[l]^{B_1} 
\CC 
	\ar@/^/[r]^{A_{2}}
	& 
	\ar@/^/[l]^{B_2}
\CC
	}
\end{equation} 
Define
\[
a_{1} = \cupfa~,~ a_{2}= \cupfb~,~ a_3 = \cupfc.
\]
By \eqref{defi:Springer_component_in_quiver}, we have
\begin{align}
\Lambda^{a_{1}}&=\{(A,B,\Gamma,0) \in \Lambda^+_{3,1}\mid 0 =  A_1\},
\\
\Lambda^{a_{2}}&=\{(A,B,\Gamma,0) \in \Lambda^+_{3,1}\mid \ker B_1 = \ker A_2\},
\\
\Lambda^{a_{3}}&=\{(A,B,\Gamma,0) \in \Lambda^+_{3,1}\mid B_2 = 0\}.
\end{align}
By \eqref{eq:fix}, the fixed-point subvarieties $(M^{a_{i}})^\Theta$ 
are described below: $[s] \in (M^{a_{1}})^\Theta = (M^{a_{3}})^\Theta$ if and only if
\eq
\begin{split}
&s=(A,B,\Gamma,0) \in \Lambda^{+}_{3,1},
\quad 
A_1 = 0 = B_2, 
\textup{ there is a }g=(g_i)_i \in \GL(V) \textup{ such that }
\\
&\Gamma_1 = g_1 \Gamma_3,
\quad
\Gamma_3 = g_3 \Gamma_1,
\quad
A_2 = g_3 B_1 g_2\inv,
\quad
B_1 =g_1 A_2 g_2\inv.
\end{split}
\endeq
Since $A_1 = 0 = B_2$, the stability condition at $V_2$ is never satisfied and hence $(M^{a_1})^\Theta = \varnothing = (M^{a_3})^\Theta$.
Meanwhile, $[s] \in (M^{a_{2}})^\Theta$ if and only if
\eq\label{eq:|U|}
\begin{split}
&s=(A,B,\Gamma,0) \in \Lambda^{+}_{3,1},
\quad \ker B_1 = \ker A_2, \textup{ there is a }g=(g_i)_i \in \GL(V) \textup{ such that }
\\
&\Gamma_1 = g_1 \Gamma_3,
\quad
\Gamma_3 = g_3 \Gamma_1,
\quad
A_1 = g_2 B_2 g_1\inv,
\quad
A_2 = g_3 B_1 g_2\inv,
\quad
B_1 =g_1 A_2 g_2\inv,
\quad 
B_2 = g_2 A_1 g_3\inv.
\end{split}
\endeq
From Example~\ref{ex:Fi}, the isotropic flag $F_\bullet$ is described by,
\eq
\begin{split}
&F_1 =  \ker (\Gamma_1| B_1B_2\Gamma_3),
\\
&F_2 =  \ker (A_1 \Gamma_1|  A_1B_1B_2\Gamma_3| B_2\Gamma_3), 
\\
&F_3  = \ker (A_2A_1\Gamma_1 |  A_2A_1B_1B_2\Gamma_3| A_2B_2\Gamma_3|\Gamma_3).
\end{split}
\endeq
For $F_1$, by the stability condition on $V_1$, we have $\Gamma_1$ is surjective and so $\ker \Gamma_1 = 0$, while the cup relation $\ker B_1 = \ker A_2$ implies that
\eq
\ker B_1 B_2 \Gamma_3 = \ker A_2 B_2 \Gamma_3 = \ker 0 = \<e\>,
\endeq
and hence $F_1 = \<e_1\>$. For $F_3$, by a similar argument we see that $\ker A_2 A_1 \Gamma_1 = \ker B_1 A_1 \Gamma_1 = \ker 0 = \<f\>$, $\ker \Gamma_3 = 0$, and hence
$F_3 = \<f_1,e_1,e_2\>$.

For $F_2$, we have
\eq
\ker A_1\Gamma_1 = \ker (g_2 B_2 g_1\inv) g_1 \Gamma_3 = \ker B_2 \Gamma_3.
\endeq
By the dimension reason, $\ker A_1\Gamma_1 = \ker B_2\Gamma_3 = 0$, while $\ker (A_1\Gamma_1 | B_2\Gamma_3)$ is 1-dimensional, and is spanned by the vector $f + \ld e$ for some $\ld \neq 0$.

By \eqref{eq:|U|}, we have
\eq
\Gamma_1 = g_1\Gamma_3 = g_1g_3\Gamma_1,
\quad
\Gamma_3 = g_3\Gamma_1 = g_3 g_1 \Gamma_3,
\endeq
and hence $g_1, g_3 \in \CC^\times$ are inverses to each other. 
Moreover,
\eq
A_1 = g_2 B_2 g_1\inv = g_2(g_2 A_1 g_3\inv) g_1\inv = g_2^2 A_1,
\endeq
and hence $g_2 \in \CC^\times$ is an involution, or $g_2 = \pm1$. Then
\eq
0 = A_1\Gamma_1(f) + \ld B_2\Gamma_3(e)
= g_2  B_2 \Gamma_3(e) + \ld B_2 \Gamma_3(e).
\endeq
That is, $-\ld$ is an eigenvalue of $g_2$ and so $\ld \in \{\pm1\}$. Therefore, 
$\tphi((M^{\da_2})^\Theta)$ splits into the following two connected components:
\eq
\begin{split}
&\{(0 \subset \<e_1\> \subset \<e_1, f_1 + e_2\> \subset \<f_1, e_1, e_2\> \subset \CC^4)\},
\\
&\{(0 \subset \<e_1\> \subset \<e_1, f_1 - e_2\> \subset \<f_1, e_1, e_2\> \subset \CC^4)\}.
\end{split}
\endeq
\endex
\section{Components of Springer fibers of type D}
\subsection{The branching rule}
Given a type D cup diagram $\da\in \BD_{n-k,k}$, if $i$ is the left endpoint of a cup, i.e., $i \in V_l^{\da} \cup X_l^{\da}$, set
\eq\textstyle
m(i) = i + \delta(i) -1 = \left\lfloor \frac{i+\sigma(i)}{2}\right\rfloor.
\endeq
Let
$\Lambda^{\da}$ be the subset of $\Lambda^\cB_{n-k,k}$ consisting of quadruples $(A,B,\Gamma,0)$ satisfying the following relations, if $n-k = k$:
\eq\label{def:LDkk}
\begin{split}
&\ker B_{m(i) \to i-1} 
= \ker A_{m(i) \to \sigma(i)} 
\tforall i\in V_l^a,
\\
&\ker B_{m(i) \to i-1} 
\neq \ker A_{m(i) \to \sigma(i)} 
\tforall i\in X_l^a,
\\
&A_{\frac{i+1}{2}\to i}\Gamma_{k \to \frac{i+1}{2}} = A_{\frac{i+1}{2}\to i}\Gamma_{k \to \frac{i+1}{2}}\sigma_k
\quad\tif i \textup{ is connected to a marked ray},
\\
&A_{\frac{i+1}{2}\to i}\Gamma_{k \to \frac{i+1}{2}} = -A_{\frac{i+1}{2}\to i}\Gamma_{k \to \frac{i+1}{2}}\sigma_k
\quad\tif i \textup{ is connected to an unmarked ray}, 
\end{split}
\endeq
while when $n-k > k$, the last two relations in \eqref{def:LDkk} are replaced by the following:
\eq\label{def:LDn-kk}
\begin{split}
&A_{k \to k+\rho(i)-1}\Gamma_{k}
=
-B_{n-k \to k+\rho(i)-1}\Gamma_{n-k}
\quad\tif i \textup{ is connected to the rightmost ray, and it's marked},
\\
&A_{k \to k+\rho(i)-1}\Gamma_{k}
=
B_{n-k \to k+\rho(i)-1}\Gamma_{n-k}
\quad\tif i \textup{ is connected to the rightmost ray, and it's unmarked}.
\end{split}
\endeq
We further define
\eq\label{def:MD}
M^{\da} = p_{n-k,k}(\Lambda^{\da}), 
\quad
K^{\da} = \tphi(M^{\da}).
\endeq
Note that $M^{\da}$ also makes sense for $\da \in \tB_{n-k,k}$ (see Section \ref{sec:mcup}).
The rest of the section is dedicated to the proof of the following theorem:
\begin{thm}\label{thm:mainDQ}
Let $a \in B_{n-k,k}$ be a type A cup diagram.
\begin{enumerate}[(a)]
\item If $a$ is symmetric (with respect to the axis of reflection),
then
\[
(M^a)^\Theta = \sqcup_{\da} M^{\da}, 
\] 
where $\da$ runs over all type D cup diagrams in $\BD_{n-k,k}$ which unfold to $a$ in the sense of Definition~\ref{def:unfold}.
\item If $a$ is not symmetric, then $(M^a)^\Theta \subseteq (M^b)^\Theta$ for some symmetric $b \in B_{n-k,k}$.
\end{enumerate}
As a consequence, 
\[
M_{n-k,k}^\Theta = \bigcup_{\da \in \BD_{n-k,k}} M^{\da}.
\]
\end{thm}

\begin{lemma}\label{lem:Dbase1}
Let $a \in \tB_{n-k,k}$.
If there is a marked cup, unmarked cup, or marked ray that does not cross the axis of reflection, then the relation it imposed in $(M^a)^\Theta$ is equivalent to the relation imposed by its mirror image.
\end{lemma}
\proof
Assume that there is a (unmarked) cup connecting vertices $i, j$ on the left half of the diagram. The cup relation is $\ker A_{m(i) \to j} = \ker B_{m(i) \to i-1}$.
Using Remark~\ref{rmk:fix}, we obtain
\eq
\begin{split}
\ker A_{m(i)\to j} &= \ker A_{j-1} A_{j-2} \cdots A_{m(i)}
\\
&= \ker (g_j B_{n-j} g_{j-1}\inv)(g_{j-1} B_{n-j-1} g_{j-2}\inv)\cdots (g_{m(i)+1} B_{n-m(i)-1} g_{m(i)}\inv)
\\
&= \ker B_{n-m(i)\to n-j} g_{m(i)}\inv.
\end{split}
\endeq 
On the other hand, we obtain
\eq
\begin{split}
\ker B_{m(i)\to i-1} &= \ker B_{i-1} B_i \cdots B_{m(i)-1}
\\
&= \ker (g_{i-1} A_{n-i} g_{i}\inv)(g_{i} A_{n-i+1} g_{i+1}\inv)\cdots (g_{m(i)-1} A_{n-m(i)} g_{m(i)}\inv)
\\
&= \ker A_{n-m(i)\to n-i+1} g_{m(i)}\inv.
\end{split}
\endeq 
Hence, this cup relation is equivalent to another cup relation $\ker A_{n-m(i)\to n-i+1} = \ker B_{n-m(i)\to n-j}$ corresponding to the cup connecting vertices $n-j+1, n-i+1$, which is the mirror image of the original cup.
Similarly, a marked cup relation is equivalent to the marked cup relation for its mirror image; so is the relation imposed by the leftmost ray.
\endproof

\begin{cor}\label{cor:Dbase}
\begin{enumerate}[(a)]
\item
If $a\in \tB_{n-k,k}$ is symmetric (with respect to the axis of reflection) and has no cups that cross the axis, then $(M^a)^\Theta =M^{\da}$, where $\da\in \BD_{n-k,k}$ is the marked cup diagram obtained by cropping the right half of $a$.
\item
If $a \in B_{n-k,k}$ is not symmetric, then $(M^a)^\Theta \subseteq (M^b)^\Theta$ for some symmetric $b \in B_{n-k,k}$.
\end{enumerate}
\end{cor}
\proof
Part (a) follows immediately from Lemma~\ref{lem:Dbase1}.
For Part (b), pick up to $\lfloor\frac{k}{2}\rfloor$ cups $\{\{i_t, j_t\}\mid 1\leq t \leq r\}$ so that these cups that lie entirely on one half of $a$, and they are not mirror images of each other. 
By Lemma~\ref{lem:Dbase1} they impose relations that are equivalent to the ones imposed by the cups $\{\{n+1-j_t, n+1-i_t\}\mid 1\leq t \leq r\}$ assuming \eqref{eq:fix} holds.
Note that the remaining vertices $\{1,\ldots, n\} \setminus \{i_t,j_t, n+1-i_t, n+1-j_t\}_t$ show up in pairs. 
One can then define $b \in B_{n-k,k}$ using on the $2r$ cups altogether with $k-2r$ symmetric cups connecting the remaining vertices from center to outside.
Finally, using \eqref{eq:fix}, the symmetric cups added impose only relations on $g_m$ and hence $(M^a)^\Theta \subseteq (M^b)^\Theta$.
\endproof
\ex
Let $a = \cupfa \in B_{3,1}$.
Since $k = 1$, we cannot pick cups $\{i_t, j_t\}$ otherwise we have $2r > k$.
Hence, $b = \cupfb \in B_{3,1}$ and indeed we have $(M^a)^\Theta = \varnothing \subseteq (M^b)^\Theta$ as in Example~\ref{ex:31}.

If $a = \{\{1,4\}, \{2,3\}, \{5,6\}\} \in B_{3,3}$, we can choose $r=1$, $\{i_1, j_1\} = \{2,3\}$ with mirror image $\{4,5\}$. 
Finally we add enough cups to make sure $b\in B_{3,3}$ and so $b = \{\{2,3\}, \{4,5\}, \{1,6\}\}$.
Note that $b$ is not unique -- one can start with choosing $\{i_1, j_1\} = \{5,6\}$ and the resulting cup diagram is $b=\{\{1,2\},\{3,4\},\{5,6\}\}$.
\endex
\begin{lemma}\label{lem:Dind}
If $a \in \tB_{n-k,k}$ has cups that cross the axis of reflection, then
$(M^a)^\Theta = (M^{a'})^\Theta \sqcup (M^{a^-})^\Theta$. 
\end{lemma}
\proof
It suffices to show that  $(M^{a})^\Theta \setminus (M^{a'})^\Theta = (M^{a^-})^\Theta$.
There are two cases to be done -- if $a$ has at least two cups that cross the axis, or exactly one such cup. 

Assume that $a$ has exactly one cup crossing the axis, say the cup represented by $\{i+1, n-i\}$ for some $i<m$.
Note that the other cups lying entirely on one side of the diagram show up in pairs:
if $\{j, \sigma(j)\}$ represents a cup in $a$ that lies entirely on one side, 
then $\{n+1-\sigma(j), n+1-j\}$ represents its counterpart in $a$, and the relations imposed by the two cups, respectively, are equivalent due to Remark~\ref{rmk:fix}.

Let $(x,F_\bullet) = \tphi(A,B,\Gamma,0)$. 
By Proposition~\ref{prop:HL}, the flag $F_\bullet$ is isotropic, 
and hence is determined by the above cup relations modulo $F_i$ (and its counterpart $F_{n-i}$), which is, by Corollary~\ref{cor:Fi},
$
F_i = \ker ( A_{1\to i}\Gamma_{\to 1}| \ldots | A_{i-1}\Gamma_{\to i-1}|\Gamma_{\to i}).
$ 
Note that cup relation for $\{i+1, n-i\}$, $\ker A_{m\to n-i} = \ker B_{m\to i}$, is equivalent to
\eq
\ker A_{m\to n-i} = \ker A_{m\to n-i} g_m\inv.
\endeq
On the other hand, 
by Corollary~\ref{cor:Fi} we have
\eq
F_i = \ker ( A_{1\to i}\Gamma_{\to 1}| \ldots | A_{i-1}\Gamma_{\to i-1}|\Gamma_{\to i}).
\endeq
By Algorithm~\ref{alg:fold}, it remains to prove that $(M^a)^\Theta$ splits into two exclusive cases described by
\eq
\begin{cases}
\qquad  
\Gamma_{k \to \frac{i+1}{2}} = \pm \Gamma_{k \to \frac{i+1}{2}}\sigma_k
&\tif n-k = k, 
\\
A_{k \to k+\rho-1}\Gamma_{k}
=
\pm B_{n-k \to k+\rho-1}\Gamma_{n-k}
&\tif n-k > k.
\end{cases}
\endeq
For the case $n-k = k = m$, note that in $a$ there is no cups and hence $F_i$ contains $\<e_t, f_t \mid 1\leq t \leq \frac{i-1}{2}\>$.
It remains to investigate the possible vectors that span the 1-dimensional kernel of $A_{\frac{i+1}{2}\to i}\Gamma_{k\to \frac{i+1}{2}}$.
If $\ld e + \mu f \in \ker A_{\frac{i+1}{2}\to i}\Gamma_{k\to \frac{i+1}{2}}$, then by Remark~\ref{rmk:fix} we have
\eq
\begin{split}
\ld e + \mu f &\in \ker A_{i-1} \cdots A_{\frac{i+1}{2}} B_{\frac{i+1}{2}}\Gamma_{k}
= \ker g_i B_{n-i} \cdots B_{n-1-\frac{i+1}{2}}A_{n-1-\frac{i+1}{2}} \cdots \Gamma_k \sigma_k
\\
&
=\ker B_{n-\frac{i+1}{2}\to n-i}A_{k\to n-\frac{i+1}{2}}\Gamma_{k}\sigma_k,  
\end{split}
\endeq
and thus, by combining the admissibility conditions and the cup relation, we have
\eq
\begin{split}
\ld e - \mu f &\in \ker B_{n-\frac{i+1}{2}\to n-i}A_{k\to n-\frac{i+1}{2}} \cdots \Gamma_{k}
=
\ker A_{m\to n-i}A_{* \to m}\Gamma_{k\to *} 
\\
&
=\ker B_{m\to i}A_{* \to m}\Gamma_{k\to *} 
= \ker B_{n-* \to i} A_{k\to n-*}\Gamma_k
\\
&=\ker A_{\frac{i+1}{2}\to i}\Gamma_{k\to \frac{i+1}{2}}.
\end{split}
\endeq
Therefore, it splits into two cases $\ld = 0$ or $\mu = 0$, which correspond to whether $A_{\frac{i+1}{2}\to i} \Gamma_{k \to \frac{i+1}{2}}$ is equal to $\pm A_{\frac{i+1}{2}\to i} \Gamma_{k \to \frac{i+1}{2}}$.

For the case $n-k > k$, a similar analysis reduces the discussion to investigate the 1-dimensional kernel of the map 
\eq
(A_{k \to c+1}\Gamma_{k}| B_{n-k \to i-c}\Gamma_{n-k}) :
\<f_{c+1}, e_{i-c}\> \to V_i = \CC^k.
\endeq
A simpler deduction (due to the lack of the non-trivial involution $\sigma_k$) leads to that 
$A_{k \to c+1}\Gamma_{k} = \pm B_{n-k \to i-c}$ (cf. Example~\ref{ex:31}).

Assume now $a$ has more than one symmetric cup. We pick the two innermost nested cups connecting $j+1, n-j$, and $i+1, n-i$, respectively, such that $i<j < m$.
Write for short $p = \lfloor\frac{j+2}{2}\rfloor > q = \lfloor\frac{i+1}{2}\rfloor$.
Note that $F_{j+1} = \ker (A_{1\to {j+1}}\Gamma_{\to 1}| \ldots | \Gamma_{\to j+1})$. 
We have $\dim\ker A_{p+q\to j+1}\Gamma_{\to p+q} \neq 2$, otherwise
$2(p+q) \leq \dim F_{j+1} = j+1 \in \{2p-1, 2p-2\}$, which is absurd.

Hence, there are two exclusive cases: whether $\dim\ker A_{p+q\to j+1}\Gamma_{\to p+q} = 0$ or 1.
By a tedious analysis using admissibility conditions and the cups relations in $a$, one recover the second condition in \eqref{def:LDkk} assuming $\dim\ker A_{p+q\to j+1}\Gamma_{\to p+q} = 1$, and vice versa.
\endproof

\subsection{Two Jordan blocks of equal sizes}

We have the following result which generalizes~\cite{Fun03} and~\cite{SW12} to type D (see also Proposition~\ref{prop:known_results_about_irred_comp}).
\begin{theorem}\label{thm:main2a}
Let $x \in \cND$ be a nilpotent element of Jordan type $(k,k)$ as in \eqref{eq:partitionD} and $\da\in \BD_{k,k}$.
Then $K^{\da}\subseteq \cBD_{k,k}$ consists of the pairs $(x, F_\bullet)$ which satisfy the following conditions imposed by the diagram $\da$:
\begin{enumerate}[(i)]
\item If vertices $i < j \leq m$ are connected by a cup without a marker, then 
\[
F_{j}=x^{-\frac{1}{2}(j-i+1)}F_{i-1}.
\]
\item If vertices $i < j \leq m$ are connected by a marked cup, then 
\[
x^{\frac{1}{2}(j-i+1)}F_j+F_{i-1} = F_i \hspace{1em} \text{and} \hspace{1em} 
x^{\frac{1}{2}(n-2j)}F_j^\perp = F_j .
\]
\item If vertex $i$ is connected to a marked ray, then 
\[
F_i=\langle e_1,\ldots,e_{\frac{1}{2}(i-1)},f_1,\ldots,f_{\frac{1}{2}(i+1)}\rangle.
\]
\item If vertex $i$ is connected to a ray without a marker, then 
\[
F_i=\langle e_1,\ldots,e_{\frac{1}{2}(i+1)},f_1,\ldots,f_{\frac{1}{2}(i-1)}\rangle.
\]
\end{enumerate}
\end{theorem}
\begin{proof}
Part (i) is exactly the type A result. For part (ii), 
combining Algorithm~\ref{alg:fold}, \eqref{def:LDkk} and part (i), we know that 
\eq\label{eq:mcup-unfold}
F_{j}\neq x^{-\frac{1}{2}(j-i+1)}F_{i-1},
\quad
F_{n+1-j}= x^{-(m+1-j)}F_{j-1},
\quad
F_{n+1-i}= x^{-(m+1-i)}F_{i-1}.
\endeq
By elementary linear algebra, $F_{j}\neq x^{-\frac{1}{2}(j-i+1)}F_{i-1}$ is equivalent to $x^{\frac{1}{2}(j-i+1)}F_j+F_{i-1} = F_i$; while the latter two relations in \eqref{eq:mcup-unfold} are equivalent to 
$x^{\frac{1}{2}(n-2j)}F_j^\perp = F_j$ thanks to the fact that the subdiagram of the unfolded diagram $a$ from vertex $i$ to $n+1-i$ only contains unmarked cups.
Parts (iii) and (iv) follow from that
$\Gamma_{k \to \frac{i+1}{2}} = \pm \Gamma_{k \to \frac{i+1}{2}} \sigma_k$ is equivalent to $\ker \Gamma_{k \to \frac{i+1}{2}} = \<e\>$ or $\<f\>$ due to our choice of $\sigma_k$.
\end{proof}
\rmk
The readers  may notice that the marked cup relation in Theorem~\ref{thm:main2a}(ii) is a pair of equations instead of a single equation. 
In an earlier stage of this work, we thought it can be simplified using one single equation, but it turns out that there is no obvious way to do it to our knowledge.

\endrmk
\begin{ex}\label{ex:UImU}
Let $(x, F_\bullet) \in K^{\da})$, where $\da \in \BD_{5,5}$ is the cup diagram below:
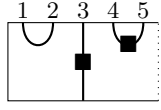
\[ 
\begin{tikzpicture}[baseline={(0,-.3)}, scale = 0.8]
\draw (3.75,0) -- (1.25,0)  -- (1.25,-1.3) --  (3.75,-1.3);
\draw[dotted] (3.75,-1.3) -- (3.75,0);
\begin{footnotesize}
\node at (1.5,.2) {$1$};
\node at (2.5, -.65) {$\blacksquare$};
\node at (2,.2) {$2$};
\node at (3.25, -.35) {$\blacksquare$};
\node at (2.5,.2) {$3$};
\node at (3,.2) {$4$};
\node at (3.5,.2) {$5$};
\end{footnotesize}
\draw[thick] (2.5,0) -- (2.5, -1.3);
\draw[thick] (1.5,0) .. controls +(0,-.5) and +(0,-.5) .. +(.5,0);
\draw[thick] (3,0) .. controls +(0,-.5) and +(0,-.5) .. +(.5,0);
\end{tikzpicture}
\]
Since vertices $1, 2$ are connected by a cup without a marker, we have 
\[
F_{2}=x^{-1}F_{0} = \<e_1, f_1\>.
\]
Next, vertex $3$ is connected to a marked ray, so 
\eq
F_3=\langle e_1, e_2, f_1\rangle.
\endeq
Finally, vertices $4, 5$ are connected by a marked cup, and thus
\eq\label{eq:M45}
xF^5+F_3 = F_4, \quad F_5^\perp = F_5.
\endeq
The most general form for $F_4$ is $F_4 = \<e_1,f_1,f_2, \lambda e_2+\mu f_3\>$ for some $(\lambda, \mu) \in \CC^2\setminus (0,0)$. By \eqref{eq:M45}, $F_5$ must contains $\lambda e_3 + \mu f_4$ and hence
\[
F_5 =\<e_1,f_1,f_2, \lambda e_2+\mu f_3, \lambda e_3+\mu f_4\>.
\]
\end{ex}
\subsection{Two Jordan blocks of unequal sizes}
For this section we assume $n-k > k$.
Recall that $\rho(i) \in \ZZ_{>0}$ counts the number of rays (including itself) to the left of $i$ , and $c(i) = \frac{i-\rho(i)}{2}$ is the total number of cups to the left of $i$.
\begin{theorem}\label{thm:main2b}
Let $x \in \cND$ be a nilpotent element of Jordan type $(n-k,k)$ as in \eqref{eq:partitionD} and $\da\in \BD_{n-k,k}$.
Then $K^{\da}\subseteq \cBD_{n-k,k}$ consists of the pairs $(x, F_\bullet)$ which satisfy (i)--(ii) of Theorem~\ref{thm:main2a} and the following conditions imposed by the diagram $\da$:
\begin{enumerate}[(i)]
\setcounter{enumi}{3}
\item If vertex $i$ is connected to a marked ray, then 
\[
F_i=\langle e_1,\ldots,e_{i-c(i)-1},f_1,\ldots,f_{c(i)}, {f_{c(i)+1}+e_{i-c(i)}}\rangle.
\]
\item If vertex $i$ is connected to the rightmost ray without a marker, then 
\[
F_i=\langle e_1,\ldots,e_{i-c(i)-1},f_1,\ldots,f_{c(i)}, {f_{c(i)+1}-e_{i-c(i)}}\rangle.
\]
\item If vertex $i$ is connected to an unmarked ray that is not the rightmost, then 
\[
F_i=\langle e_1,\ldots,e_{i-c(i)},f_1,\ldots,f_{c(i)}\rangle.
\]
\end{enumerate}
\end{theorem}
\begin{proof}
It suffices to prove part (iii) -- (iv) regarding the new relations for the rightmost ray.
Write $\rho = \rho(i), c = c(i) = \frac{i-\rho}{2}$ for short.
By Corollary~\ref{cor:Fi}, we have $F_i = \ker(A_{1\to i}\Gamma_{\to 1}| \ldots | \Gamma_{\to i})$, and we are to prove that
\eq
\begin{split}
&\ker(A_{c+1\to i}\Gamma_{k\to c+1}) = 0 = \ker(A_{i-c\to i}\Gamma_{n-k\to i-c}),
\\
&\ker(A_{c+1\to i}\Gamma_{k\to c+1}| A_{i-c\to i}\Gamma_{n-k\to i-c}) = \CC.
\end{split}
\endeq
The former line follows from the admissibility conditions and \eqref{eq:fix}, while the latter line follows from 
$A_{c+1\to i}\Gamma_{k\to c+1} = \pm A_{i-c\to i}\Gamma_{n-k\to i-c}$,
which is a consequence of \eqref{def:LDn-kk}.
Therefore, the additional basis vector is $f_{c+1} \pm e_{i-c}$.
\end{proof}
\begin{ex}\label{ex:||U}
Let $F_\bullet \in K^{\da}$, where $\da \in \BD_{5,3}$ is the cup diagram below:
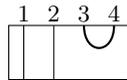
\[
\begin{tikzpicture}[baseline={(0,-.3)}, scale = 0.8]
\draw  (1.75,0) -- (-.25,0)  -- (-.25,-.9)--(1.75,-.9);
\draw[dotted] (1.75,0) -- (1.75,-.9);
\begin{footnotesize}
\node at (0,.2) {$1$};
\node at (.5,.2) {$2$};
\node at (1,.2) {$3$};
\node at (1.5,.2) {$4$};
\end{footnotesize}
\draw[thick] (1,0) .. controls +(0,-.5) and +(0,-.5) .. +(.5,0);
\draw (0,0) -- (0,-.9);
\draw (0.5,0) -- (0.5,-.9);
\end{tikzpicture}
\]
Firstly, vertex $1$ is connected to a unmarked ray that is not the rightmost, so 
\[
F_1=\langle e_1 \rangle.
\]
Secondly, although vertex $2$ is connected to a ray without a marker, we use a different rule since it is the rightmost ray. Hence, 
\[
F_2=\langle e_1, f_1 - e_2\rangle.
\]
Finally, vertices $2, 3$ are connected by a cup without a marker, we obtain that 
\[
F_{3}=x\inv F_{1} = \<e_1, e_2, f_1\>.
\]
\end{ex}

\section{Irreducibility}\label{sec:irred}
In this section we prove the irreducibility of $K_{\da}$ for $\da \in B^\textup{D}_{n-k,k}$ via an induction on $n$. 
\subsection{The base case}
Firstly we prove the case in which $n=2m \leq 4$.
\lemma\label{lem:n<=4}
Let $\da\in \BD_{\ld}$ for some $\ld = (n-k,k)$, $n\leq 4$. Then $K_{\da}$ is irreducible.
\endlemma
\proof
The four irreducible components are described explicitly in Examples~\ref{ex:Ka11}--\ref{ex:Ka22}. Following the notation therein, both $K^{\da}, K^{\dot{b}}$ are (geometric) points, and hence it suffices to show that $K^{\da_{12}}, K^{\da_{13}}$ are irreducible, where
\[
\da_{13} = \mcupab~,~\da_{12} = \mcupaa.
\]
We prove this by showing that they are $\CP^1$-bundles, i.e., for $i=2,3$, we are to construct $(K^{\da_{1i}}, \CP^1, \pi_i, \{\textrm{pt}\})$ such that the diagram below commutes:
\eq\label{eq:n<=4comm}
\begin{tikzcd}
\pi_i\inv(\CP^1) \ar[r,"\phi_i", "\simeq"'] \ar[d,"\pi_i"]
& \CP^1\times \{\textup{pt}\} \ar[ld, "\textup{proj}_1"]
\\
\CP^1
&
\end{tikzcd}
\endeq
Here $\phi_i$ is a homeomorphism, proj$_1$ is the projection onto the first factor.
We define the projection map $\pi_i$ by
\eq
\pi_i:K^{\da_{1i}}\to \CP^1,
\quad
( 0 \subset F_1=\<\ld e_1+\mu f_1\> \subset \ldots \subset\CC^n) \mapsto [\ld:\mu].
\endeq
We see from Example~\ref{ex:Ka22} that $\pi_i\inv(\CP^1) = K^{\da_{1i}}$.
Now we define $\phi_i$ by 
\eq
\phi_i:K^{\da_{1i}}\to \CP^1 \times \{\textup{pt}\},
\quad
( 0 \subset F_1=\<\ld e_1+\mu f_1\> \subset \ldots \subset\CC^n) \mapsto ([\ld:\mu], \textup{pt}).
\endeq
It is easy to check that its inverse is given by
$\phi_i\inv([\ld:\mu], \textup{pt}) = F_\bullet$, where
\eq
F_1 = \<\ld e_1 + \mu f_1\>,
\quad
F_3 = \<e_1, f_1, \ld e_2 + \mu f_2\>,
\quad 
F_2 = \begin{cases}
\<e_1, f_1\> &\tif i=3;
\\
\<\ld e_1 + \mu f_1,\ld e_2 + \mu f_2\> &\tif i=2.
\end{cases}
\endeq
It is routine to check that $\phi_i$ is {bicontinuous} as well as that \eqref{eq:n<=4comm} commutes. We are done.
\endproof
\subsection{Iterated $\CP^1$-bundles}
\begin{definition}
A space $X$ is called an {\em iterated $\CP^1$-bundle of length $\ell$} if there exists spaces
$X=X_1, X_2, \ldots, X_\ell, X_{\ell+1} = \textup{pt}$ and maps $\pi_i:X_i \to \CP^1$ such that $(X_i, \CP^1, \pi_i, X_{i+1})$ is a fiber bundle for $1\leq i \leq \ell$.
A point is considered an iterated $\CP^1$-bundle of length 0.
\end{definition}
The goal of this section is to prove the following theorem.
\begin{thm}\label{thm:main3}
Let $\da\in \BD_{n-k,k}$ be a marked cup diagram with $\ell$ cups. 
Then
$K^{\da}$ is an iterated $\CP^1$-bundle of length $\ell$. 

As a corollary, $\{K^{\da}~|~\da\in \BD_{n-k,k}\}$ forms the complete list of irreducible components of the two-row Springer fiber $\cBD_{n-k,k}$ of type D. 
\end{thm}
We have proved the special case $n\leq 4$ of Theorem~\ref{thm:main3} in Lemma~\ref{lem:n<=4}. 
Next, we will use an induction on $n=2m$.
For $m \geq 3$, below is an exhaustive list of possible ray or cup that connects the first vertex in the marked cup diagram $\da \in \BD_{n-k,k}$:
\eq\label{eq:list}
\icupa~,
\quad
\icupb~,
\quad
\icupc~,
\quad
\icupd~
\quad
(1\leq t \leq \textstyle\lfloor\frac{m}{2}\rfloor)
\endeq
Here $a_1$ represents the subdiagram of $\dot{a}$ to the right of the ray or cup connected to vertex 1; while $a_2$ represents the subdiagram of $\dot{a}$ enclosed by the cup connected to vertex 1.

Our plan is to show that:
\begin{itemize}
\item Certain $K^{\da}$ is the trivial fiber bundle $K^{\dot{b}} \times K^{\dot{c}}$, where  $K^{\dot{b}}$ and $K^{\dot{c}}$ are both iterated $\CP^1$-bundles of shorter lengths.
\item Any remaining $K^{\da}$ is isomorphic to an irreducible component of a smaller $n$.
\end{itemize}
It then follows from \cite[Lemma~8.11]{S12} that $K^{\dot{b}} \times K^{\dot{c}}$ is also an iterated $\CP^1$-bundle.
For our proof, we have to rearrange \eqref{eq:list} and split them into three major cases (with possibly subcases) as below:
\begin{enumerate}[~]
\item Case I: Vertex 1 is connected to $2t$ via an unmarked cup, where $2t < m$.
\item Case II: Vertex 1 is connected to a cup, and is not case I.
Note that in case II, the subdiagram $a_2$ only consists of cups.
\item Case III: Vertex 1 is connected to a ray. We will see that it needs to have three subcases.
\end{enumerate}

\subsection{Formed spaces}
In each case we will discuss isomorphisms between {\em formed spaces}, i.e., vector spaces equipped with bilinear forms. 
For $\ld = (\ld_1, \ld_2) \vdash n$, we denote by $V_\ld$ the formed space $\CC^n = \<e_1^\ld, \ldots, e^\ld_{\ld_1}, f^\ld_1, \ldots, f^\ld_{\ld_2}\>$ equipped with the bilinear form $\beta_\ld$ (see \eqref{eq:basisVld}--\eqref{eq:beta}).

For $\ld = (\ld_1,\ld_2), \mu = (\mu_1,\mu_2)$ such that $\ld_1 \geq \mu_1, \ld_2 \geq \mu_2$, we define  a projection
\eq\label{def:P}
P^\ld_\mu: V_\ld \to V_\mu,
\quad
e^\ld_i \mapsto \begin{cases}
e^\mu_i &\tif 1\leq i \leq \mu_1;
\\
0 &\textup{otherwise},
\end{cases}
\quad
f^\ld_i \mapsto \begin{cases}
f^\mu_i &\tif 1\leq i \leq \mu_2;
\\
0 &\textup{otherwise}.
\end{cases}
\endeq
We also define the inverse map of its restriction on $\<e^\ld_i, f^\ld_j ~|~ 1\leq i \leq \mu_1, 1\leq j \leq \mu_2\>$ by
\eq
P^\mu_\ld: V_\mu \to V_\ld,
\quad
e^\mu_i \mapsto e^\ld_i,
\quad
f^\mu_j \mapsto f^\ld_j.
\endeq 
Now we consider a formed subspace $W \subseteq V_\ld$ that is isotropic (i.e., $W \subseteq W^\perp$).
Let  $\{x^\ld_i\}_{i\in I}$ and $\{x^\ld_j\}_{j\in J}$ be bases of $W$ and $W^\perp$, respectively, such that $I \subsetneq J$.
We write $\bar{x}^\ld_j = x^\ld_j + W$, and hence
$\{\bar{x}^\ld_j\}_{j \in J\setminus I}$ forms a basis of $W^\perp/W$. 
Denote the quotient map by $\psi = \psi(W)$ by
\eq\label{eq:psi}
\psi:W^\perp \to W^\perp/W.
\endeq
Moreover,  $W^\perp/W$ is a formed space equipped with the induced bilinear form $\bar{\beta}_\ld$ given by $\bar{\beta}_\ld(x+W,y+W) = \beta_\ld(x,y)$ for all $x,y \in W^\perp$.

In the following we give explicit isomorphisms between certain formed spaces that will be used throughout Section~\ref{sec:irred}.
\begin{lemma}\label{lem:Q}
Let $\ld = (n-k,k)$.
\begin{enumerate}[(a)]
\item
Let $W$ be the formed subspace of $V_\ld$ spanned by $e_i^\ld,  f_i^\ld (1\leq i \leq t)$ such that $2t<m$.
Then 
\[
W^\perp = \< e_i^\ld,  f_j^\ld ~|~ 1\leq i \leq n-k-t, 1\leq j \leq k-t \>.
\] 
Moreover, there is a formed space isomorphism $Q^{\textup{I}}:W^\perp/W \to V_\nu$, for $\nu = \ld - (2t,2t)$, given by
\eq\label{eq:Qa}
\bar{e}^\ld_{t+i} \mapsto \begin{cases}
\sqrt{-1} e^\nu_i &\tif t \in 2\ZZ+1;
\\
e^\nu_i &\textup{otherwise},
\end{cases}
\quad
\bar{f}^\ld_{t+i} \mapsto \begin{cases}
\sqrt{-1} f^\nu_i &\tif t \in 2\ZZ+1;
\\
f^\nu_i &\textup{otherwise},
\end{cases}
\endeq
\item Assume that $\ld = (m,m)$ and $W = \<ce_1^\ld+ d f_1^\ld\>$ for some $(c,d) \in \CC^2\setminus \{(0,0)\}$. Then 
\[
W^\perp = \<c e_m^\ld+ d f_m^\ld, f^\ld_i ~|~ 1\leq i \leq m-1\>.
\]
Moreover, the assignments below all define formed space isomorphisms between $W^\perp/W$ and $V_\nu$, for $\nu = (m-1,m-1)$:
\begin{subequations}
\begin{align}
&Q^{\III}_1:&
\sqrt{-1} \bar{e}^\ld_{i+1} &\mapsto e^\nu_i,
&
\sqrt{-1} \bar{f}^\ld_{i} &\mapsto  f^\nu_i,
&\tif d=0;\label{eq:III-1}
\\
&Q^{\III}_2:&
\bar{f}^\ld_{i+1} &\mapsto f^\nu_i,
&
 \bar{e}^\ld_{i} &\mapsto  e^\nu_i,
&\tif c=0;\label{eq:III-2}
\\
&Q^{\II}_1:&
\sqrt{-1} (\overline{e^\ld_{i+1}+\textstyle\frac{d}{c} f^\ld_{i+1}}) &\mapsto e^\nu_i,
&
\sqrt{-1} \bar{f}^\ld_{i} &\mapsto  f^\nu_i,
&\tif c\neq 0, m\in2\ZZ;\label{eq:II-1}
\\
&Q^{\II}_2:&
(\overline{\textstyle\frac{c}{d} e^\ld_{i+1}+ f^\ld_{i+1}}) &\mapsto f^\nu_i,
&
\bar{e}^\ld_{i} &\mapsto  e^\nu_i,
&\tif d\neq 0, m\in2\ZZ.\label{eq:II-2}
\end{align}
\end{subequations}
\item Assume that $n-k > k$ and $W = \<e_1^\ld\>$. Then
\[
W^\perp = \<e_i^\ld, f^\ld_j ~|~ 1\leq i \leq n-k-1, 1\leq j \leq k\>.
\]
Moreover,  the assignments below all define formed space isomorphisms between $W^\perp/W$ and $V_\nu$, for $\nu = (n-k-2,k)$:
\begin{subequations}
\begin{align}
&Q^{\III}_3:&
\sqrt{\textstyle\frac{-1}{2}}(\overline{e^\ld_{i+1}+ f^\ld_{i}}) &\mapsto e^\nu_i,
&
\sqrt{\textstyle\frac{-1}{2}}(\overline{e^\ld_{i+1}- f^\ld_{i}}) &\mapsto  f^\nu_i,
&\tif n-k-2=k; \label{eq:III-3}
\\
&Q^{\III}_4:&
\sqrt{-1}\bar{e}^\ld_{i+1} &\mapsto e^\nu_i,
&
\sqrt{-1}\bar{f}^\ld_{i} &\mapsto  f^\nu_i,
&\tif  n-k-2>k. \label{eq:III-4}
\end{align}
\end{subequations}
\end{enumerate}
\end{lemma}
\proof
It follows from a direct computation using \eqref{eq:basisVld}--\eqref{eq:beta}.
For part (a), the form $\bar{\beta}_\ld$ is associated to the matrix obtained from $M_\ld$ by deleting columns and rows corresponding to $e^\ld_i, f^\ld_i$ for $1\leq i \leq t$.
Therefore, a naive projection $P^\ld_\nu$ does the job when $t$ is even; while in the odd case one needs to multiply the new basis elements by $\sqrt{-1}$ to make the forms to be compatible.

For part (b), we note first that $B_f = \{ ce^\ld_2+df^\ld_2, \ldots, ce^\ld_m+df^\ld_m, f^\ld_1, \ldots, f^\ld_{m-1}\}$ is an ordered basis of $W^\perp/W$ when $c \neq 0$; 
while $B_e = \{  e^\ld_1, \ldots, e^\ld_{m-1}, ce^\ld_2+df^\ld_2, \ldots, ce^\ld_m+df^\ld_m\}$ is an ordered basis of $W^\perp/W$ when $d \neq 0$.
In either case, the matrices associated to $\bar{\beta}_\ld$ with respect to $B_f, B_e$, respectively, are
\eq
\small
\begin{blockarray}{ *{8}{c} }
 & \{ce^\ld_{i+1}+df^\ld_{i+1}\} & \{f^\ld_i\} \\
\begin{block}{ c @{\quad} ( @{\,} *{7}{c} @{\,} )}
\{ce^\ld_{i+1}+df^\ld_{i+1}\}& A_{m-1} & -cJ_{m-1} 
\\
\{f^\ld_i\}& -cJ^t_{m-1} & 0
\\
\end{block}
\end{blockarray}
\quad
\quad
\begin{blockarray}{ *{8}{c} }
 & \{e^\ld_i\} & \{ce^\ld_{i+1}+df^\ld_{i+1}\} \\
\begin{block}{ c @{\quad} ( @{\,} *{7}{c} @{\,} )}
\{e^\ld_i\}& 0 & dJ_{m-1} 
\\
\{ce^\ld_{i+1}+df^\ld_{i+1}\}& dJ^t_{m-1} & A_{m-1}
\\
\end{block}
\end{blockarray}~~~ 
\normalsize
\endeq
Here $A_{m-1}$ is the zero matrix if $m$ is even; when $m=2r-1$ is odd, there only possible nonzero entry is
\eq
(A_{m-1})_{r,r} = \beta_\ld(ce^\ld_r+df^\ld_r, ce^\ld_r+df^\ld_r) = (-1)^{r-1}2cd.
\endeq
In other words, when $m$ is odd, $Q$ is an isomorphism of formed spaces if either $c=0$ or $d=0$.

For part (c), if $n-k-2 > k$, then the matrix of $\bar{\beta}_\ld$ is obtained from $M_\ld$ by deleting the rows and columns corresponding to $e_1^\ld$ and $e_m^\ld$, and hence one only needs to deal with the sign change of the upper left block. 
For the other case $n-k-2=k$, note first that $k$ must be odd so that \eqref{eq:partitionD} is satisfied. Next, we consider the ordered basis $B= \{e^\ld_2+f^\ld_1, \ldots, e^\ld_{k+1}+f^\ld_k, e^\ld_2-f^\ld_1, \ldots, e^\ld_{k+1}-f^\ld_k\}$. The matrix associated to $\bar{\beta}_\ld$ under $B$ is then
\eq
\small
\begin{blockarray}{ *{8}{c} }
 & \{e^\ld_{i+1}+f^\ld_i\} & \{e^\ld_{i+1}-f^\ld_i\} \\
\begin{block}{ c @{\quad} ( @{\,} *{7}{c} @{\,} )}
\{e^\ld_{i+1}+f^\ld_i\}& 0 & -2J_{k} 
\\
\{e^\ld_{i+1}-f^\ld_i\}& -2J^t_{k} &0
\\
\end{block}
\end{blockarray}~~~ 
\normalsize
\endeq 
We are done after a renormalization.
\endproof
For any $\mu = (\mu_1, \mu_2)$, let $x_\mu$ be the nilpotent operator given by
\eq
e^\mu_{\mu_1} \mapsto e^\mu_{\mu_1-1} \mapsto \ldots \mapsto e^\mu_1 \mapsto 0,
\quad
f^\mu_{\mu_2} \mapsto \ldots \mapsto f^\mu_1 \mapsto 0.
\endeq
Let $W$ be any formed subspace as in Lemma~\ref{lem:Q}. 
Denote by
$\bar{x}_\ld$ be the induced nilpotent operator determined by
\eq\label{eq:barx}
\bar{e}^\ld_{i} \mapsto \bar{e}^\ld_{i-1},
\quad  
\bar{f}^\ld_{i} \mapsto \bar{f}^\ld_{i-1},
\endeq
For each such $W$, define an integer 
\eq
\ell = \begin{cases}
2t &\tif W = \<e_i, f_i ~|~ 1\leq i \leq t\>, 2t < m;
\\
1 &\textup{otherwise}.
\end{cases}
\endeq
We further denote by $\Omega = \Omega(W)$ the map
$\Omega: \BD(\beta_\ld) \to \BD(\beta_\nu), F_\bullet \mapsto F''_\bullet$, where
\eq\label{eq:Omega}
F''_i = Q(F_{\ell+i}/W),
\quad (1\leq i \leq m-\ell)
\endeq
while $F''_{m} , \ldots F''_{n-2\ell}$ are determined by the isotropy condition with respect to $\beta_\nu$.
\begin{cor}\label{cor:Q}
Retain the notations in Lemma~\ref{lem:Q}. Let $Q$ be one of the following formed space isomorphisms: $Q^{\textup{I}}, Q^{\II}_j (j=1,2), Q^{\III}_l (1\leq l \leq 4)$. Then
 $Q \bar{x}_\ld = x_\nu Q$.
Moreover, $\Omega(\cBD_{\bar{x}_\ld}) \subseteq \cBD_{x_\nu}$.
\end{cor}
\proof
The former part  is a direct consequence of Lemma~\ref{lem:Q} thanks to the explicit construction provided in \eqref{eq:Qa}--\eqref{eq:III-4}. For the latter part, take any $F_\bullet \in \cB_{\bar{x}_\ld}$, we have
\[
 x_\nu Q(F_i/W)
=
Q \bar{x}_\ld (F_i/W)
\subseteq
Q (F_{i-1}/W).
\]
\endproof
\subsection{Case I}\label{sec:CaseI}
In this section we assume that
\eq\label{eq:caseI}
n =2m\geq 6,
\quad
\ld = (n-k,k),
\quad
\da\in \BD_{\ld},
\quad
1 = \sigma(2t) \in V^{\da}_l,
\quad
2t < m.
\endeq
In words, we consider the marked cup diagram $\da$ on $m \geq 3$ vertices with $\lfloor\frac{k}{2}\rfloor$ cups such that vertex 1 and $2t<m$ are connected by an unmarked cup, as below:
 \eq
\da= \icupd
 \quad
 (2t<m). 
\endeq
We are going to show that $K^{\da}$ is homeomorphic to the trivial bundle $K^{\dot{b}} \times K^{\dot{c}}$ for some $\db \in \BD_{\mu_1, \mu_2}, \dc \in \BD_{\nu_1, \nu_2}$ such that $\mu_2, \nu_2 < k$ so that the inductive hypothesis applies.
Note that if $2t=m$, the construction in this section will not work and has to be discussed in Section~\ref{sec:CaseII}
\begin{definition}\label{def:abc}
Assume that \eqref{eq:caseI} holds. We define two marked cup diagrams $\db \in \BD_\mu, \dc \in \BD_\nu$ as follows:
\begin{enumerate}
\item $\db$ is obtained by cropping $\da$ on the first $2t$ vertices. Note that by construction, $\db$ consists of only unmarked cups and hence $\mu =(2t,2t) \vdash 4t$.
\item $\dc$ is obtained by cropping $\da$ on the last $m-2t$ vertices with a shift on the indices.
Note that $\nu =(n-2t-k,k-2t) \vdash n-4t$ by a simple bookkeeping on the number of vertices and cups.
\end{enumerate}
Namely, we have
\[
\da= \icupd
\quad
\Rightarrow
\db =
\begin{tikzpicture}[baseline={(0,-.3)}, scale = 0.8]
\draw (2.75,0) -- (0.75,0) -- (0.75,-.9) -- (2.75,-.9);
\draw[dotted] (2.75,-.9) -- (2.75,0);\begin{footnotesize}
\node at (2.5,.2) {$2t$};
\node at (1,.2) {$1$};
\node at (1.75, -.3) {$a_1$};
\end{footnotesize}
\draw[thick] (1,0) .. controls +(0,-1) and +(0,-1) .. +(1.5,0);
\end{tikzpicture}
\quad
\textup{and}
\quad
\dc = 
\begin{tikzpicture}[baseline={(0,-.3)}, scale = 0.8]
\draw (2.25,0) -- (0.75,0) -- (0.75,-.9) -- (2.25,-.9);
\draw[dotted] (2.25,-.9) -- (2.25,0);\begin{footnotesize}
\node at (2.25,.2) {$m-2t$};
\node at (1,.2) {$1$};
\node at (1.5, -.3) {$a_2$};
\end{footnotesize}
\end{tikzpicture}
\]
\end{definition}

\lemma\label{lem:Kc}
Let $\da, \dc$ be defined as in Definition~\ref{def:abc}. 
Recall $\Omega$ from \eqref{eq:Omega} using $Q= Q^\textup{I}$ from \eqref{eq:Qa}.
Then $\Omega(K^{\da}) \subseteq K^{\dc}$.
\endlemma
\proof
We first the the unmarked cup relations hold in $\Omega(K^{\da})$.
In other words, for any $F_\bullet \in K^{\da}$ and any pairs of vertices $(i,j)$ connected by an unmarked cup in $\dc$,
\eq\label{eq:newcup}
Q(F_{\ell+j}/W) = x_\nu^{\frac{-1}{2}(j-i+1)} Q(F_{\ell+i-1}/W),
\quad
\textup{for all}
\quad 
i \in V^{\dc}_l.
\endeq
Note that $(i-\ell, j-\ell)$ must be connected by an unmarked cup, and hence
\eq\label{eq:oldcup}
F_{j-\ell} = x_\ld^{\frac{-1}{2}(j-i+1)} F_{i-\ell},
\quad
\textup{for all}
\quad 
i \in V^{\dc}_l.
\endeq
Thus, \eqref{eq:newcup} follows from combining \eqref{eq:oldcup} and the former part of Corollary~\ref{cor:Q}. 

A ray connected to vertex $i$ in $\dc$ corresponds to the five flag conditions as in Theorem~\ref{thm:main2a}(iii)--(iv) and Theorem~\ref{thm:main2b}(iii)--(v).
Since $\db$ contains exactly $t$ unmarked cups, $F_{2t} = \<e_1, \ldots, e_t, f_1, \ldots, f_t\>$ and so 
$F''_i = Q(F_{\ell+i}/F_{2t})$. A case by case analysis shows that the new ray relations hold.
\endproof
\lemma\label{lem:Kb}
Let $\da, \db$ be defined as in Definition~\ref{def:abc}.
Then the map below is well-defined: 
\eq
\pi_{a,b}: K^{\da} \to K^{\db},
\quad
F_\bullet \mapsto F'_\bullet,
\endeq
where $F'_i = P^\ld_{\mu}(F_i)$ if $1 \leq i \leq 2t$, and that $F'_{2t+1} , \ldots F'_{4t-1}$ are uniquely determined by $F'_1, \ldots, F'_{2t}$ under the isotropy condition with respect to $\beta_\mu$. 
\endlemma
\proof.
We split the proof into three steps: firstly we show that $F'_\bullet$ is isotropic under $\beta_\mu$. Secondly, we show that $F'_\bullet$ sits inside the Springer fiber $\cBD_{x_\mu}$. 
Finally, we show that $F'_\bullet$ lies in the irreducible component $K^{\db}$.

\begin{enumerate}[~]
\item Step 1: 
It suffices to show that $F'_{2t}$ is isotropic. Since the first $2t$ vertices in $\da$ are all connected by unmarked cups, $F_{2t} = \<e^\ld_1, \ldots, e^\ld_t, f^\ld_1, \ldots, f^\ld_t\>$. Hence,
\eq
F'_{2t} = \<e^\mu_1, \ldots, e^\mu_t, f^\mu_1, \ldots, f^\mu_t\>.
\endeq
It then follows directly from \eqref{eq:beta} that $(F'_{2t})^\perp = F'_{2t}$.
\item Step 2: 
It suffices to check that $x_\mu F'_i \subseteq F'_{i-1}$ for all $i$. 
Note that a direct computation shows that
\eq\label{eq:Pxmu}
P_\mu^\ld x_\ld = x_\mu P_\mu^\ld.
\endeq
Hence, when $i\leq 2t$, we have
\eq
x_\mu F'_i  = x_\mu P^\ld_\mu(F_i) = P_\mu^\ld x_\ld (F_i) 
\subseteq 
P_\mu^\ld(F_{i-1}) = F'_{i-1}.
\endeq
Next, since now $x_\mu F'_{i+1} \subseteq F'_i$ for $i<2t$ , we have
$F'_{i+1} \subseteq x_\mu\inv F'_i$, and hence
\eq
F'_{4t-i-1}= (F'_{i+1})^\perp \supseteq (x_\mu\inv F'_i)^\perp = x_\mu F'_{4t-i}.
\endeq
\item Step 3: we need to verify the conditions
\eq\label{eq:newcup}
P^\ld_\mu(F_{\sigma(i)}) = x_\mu^{\frac{-1}{2}(\sigma(i)-i+1)} (P^\ld_\mu(F_{i-1})),
\quad
\textup{for all}
\quad 
i \in V^{\db}_l.
\endeq
Note that $V^{\db}_l \subseteq V^{\da}_l$, hence the unmarked cup relations in $\da$ hold, i.e.,
\eq\label{eq:oldcup}
F_{\sigma(i)} = x_\ld^{\frac{-1}{2}(\sigma(i)-i+1)} F_{i-1},
\quad
\textup{for all}
\quad 
i \in V^{\db}_l.
\endeq
Thus, \eqref{eq:newcup} follows from applying $P^\ld_\mu$ to \eqref{eq:oldcup},
thanks to \eqref{eq:Pxmu}. 
\end{enumerate}
\endproof
Finally, we are in a position to show that $K^{\da}$ is irreducible.
\begin{prop}\label{prop:caseI}
Retain the notations of Lemma~\ref{lem:Kc}--\ref{lem:Kb}. The assignment
$F_\bullet \mapsto (\pi_{a,b}(F_\bullet), \Omega(F_\bullet))$ defines a homeomorphism between $K^{\da}$ and the trivial fiber bundle $K^{\db} \times K^{\dc}$.
\end{prop}
\proof
It suffices to check that the given assignment is a homeomorphism with its inverse map given by
$(F'_\bullet, F''_\bullet) \mapsto F_\bullet$, where
\eq
F_i = 
\begin{cases}
P^\mu_\ld(F'_i) &\tif 1\leq i \leq 2t;
\\
\psi\inv(Q\inv(F''_{i-2t})) &\tif 2t+1 \leq i \leq m;
\\
F_{n-i}^\perp &\tif m+1 \leq i \leq n.
\end{cases}
\endeq
That $F_\bullet \in K^{\da}$ almost follows from construction, except for that we need to check if $F_{2t} \subset F_{2t+1}$, which follows from that $\psi\inv (Q\inv(F''_1))$ is a $2t+1$-dimensional space that contains $F_{2t}$.
\endproof
\subsection{Case II}\label{sec:CaseII}
In this section we consider  one of the following diagrams:
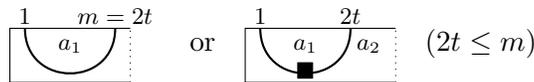
\[
\begin{tikzpicture}[baseline={(0,-.3)}, scale = 0.8]
\draw (2.75,0) -- (0.75,0) -- (0.75,-.9) -- (2.75,-.9);
\draw[dotted] (2.75,-.9) -- (2.75,0);\begin{footnotesize}
\node at (2.5,.2) {$m=2t$};
\node at (1,.2) {$1$};
\node at (1.75, -.3) {$a_1$};
\end{footnotesize}
\draw[thick] (1,0) .. controls +(0,-1) and +(0,-1) .. +(1.5,0);
\end{tikzpicture}
\quad
\textup{or}
\quad
\icupc 
\quad
(2t \leq m)
\]
Note that for the latter subcase the marked cup $(1,2t)$ has to be accessible from the axis of reflection, and so there is no rays in the subdiagram $a_2$. Moreover, there is no rays in the entire diagram $\da$, and thus $\ld= (m,m)$.
That is, in this section we are working with the following assumptions:
\eq\label{eq:caseII}
n =2m\geq 6,
\quad
\ld = (m,m),
\quad
\da\in \BD_{\ld},
\quad
\begin{array}{l}
\textup{either }
1 = \sigma(m) \in V^{\da}_l
\\
\textup{or }
1 = \sigma(2t) \in X^{\da}_l
\quad
(2t \leq m).
\end{array}
\endeq
We are going to show that $K^{\da}$ is a $\CP^1$-bundle with typical fiber $K^{\dc}$ for some $\dc \in \BD_{\nu_1, \nu_2}$ such that $\nu_2 < k$ so that the inductive hypothesis applies.
\begin{definition}\label{def:ac}
Assume that \eqref{eq:caseII} holds. Define $\dc \in \BD_\nu$ by performing the following actions:
\begin{enumerate} 
\item Remove the cup connected to vertex 1 together with vertex 1. 
\item Connect a marked ray to vertex $\sigma(1)$.
\item Decrease all indices by one.
\end{enumerate}
Note that $\nu =(m-1,m-1)$ by a simple bookkeeping on the number of vertices and cups.
Namely, we have
\[
\da=
\begin{tikzpicture}[baseline={(0,-.3)}, scale = 0.8]
\draw (2.75,0) -- (0.75,0) -- (0.75,-.9) -- (2.75,-.9);
\draw[dotted] (2.75,-.9) -- (2.75,0);\begin{footnotesize}
\node at (2.5,.2) {$m$};
\node at (1,.2) {$1$};
\node at (1.75, -.3) {$a_1$};
\end{footnotesize}
\draw[thick] (1,0) .. controls +(0,-1) and +(0,-1) .. +(1.5,0);
\end{tikzpicture}
\Rightarrow
\dc=
\begin{tikzpicture}[baseline={(0,-.3)}, scale = 0.8]
\draw (2.75,0) -- (1.25,0) -- (1.25,-.9) -- (2.75,-.9);
\draw[dotted] (2.75,-.9) -- (2.75,0);\begin{footnotesize}
\node at (2.5,.2) {$m-1$};
\node at (2.5, -.65) {$\blacksquare$};
\node at (1.75, -.3) {$a_1$};
\end{footnotesize}
\draw[thick] (2.5,0) -- (2.5,-.9); 
\end{tikzpicture}
\quad
\textup{or}
\quad
\da=
\icupc 
\Rightarrow
\dc=
\begin{tikzpicture}[baseline={(0,-.3)}, scale = 0.8]
\draw (3.25,0) -- (1.25,0) -- (1.25,-.9) -- (3.25,-.9);
\draw[dotted] (3.25,-.9) -- (3.25,0);\begin{footnotesize}
\node at (2.5,.2) {$2t-1$};
\node at (2.5, -.65) {$\blacksquare$};
\node at (1.75, -.3) {$a_1$};
\node at (2.8, -.3) {$a_2$};
\end{footnotesize}
\draw[thick] (2.5,0) -- (2.5,-.9); 
\end{tikzpicture}
\]
\end{definition}
Define  $\pi:K^{\da}\to \CP^1$ by
\eq
F_\bullet = ( 0 \subset F_1=\<\ld e_1+\mu f_1\> \subset \ldots \subset\CC^n) \mapsto [\ld:\mu].
\endeq
Now we check $(K^{\da}, \CP^1, \pi, K^{\dc})$ is a fiber bundle. For the local triviality condition we use the open covering $\CP^1 = U_1 \cup U_2$ where
\eq
U_1 = \{[1:\gamma] ~|~\gamma \in \CC\},
\quad
U_2 = \{[\gamma:1] ~|~\gamma \in \CC\}.
\endeq
\lemma\label{lem:KcII}
Let $\da, \dc$ be defined as in Definition~\ref{def:ac}. 
Recall $\Omega$ from \eqref{eq:Omega} using $Q = Q^{\II}_j$, $j=1,2$ from Lemma~\ref{lem:Q}(b). 
Then the maps below are well-defind: 
\eq
\phi_j:\pi\inv(U_j) \to K^{\dc},
\quad
F_\bullet \mapsto \Omega(F_\bullet). 
\endeq
\endlemma
\proof
The cup relations in subdiagrams $a_1$ and $a_2$ can be verified similarly as in Lemma~\ref{lem:Kc}.
For the marked ray connected to vertex $2t-1$ in $\dc$, the corresponding relation, according to Theorem~\ref{thm:main2a}(iii) and Theorem~\ref{thm:main2b}(iv), is
\eq\label{eq:newray}
F''_{2t-1} = 
\begin{cases}
\<e^\nu_1, \ldots, e^\nu_{t-1}, f^\nu_1, \ldots, f^\nu_t\>;
&\tif 2t=m;
\\
\<e^\nu_1, \ldots, e^\nu_{t-1}, f^\nu_1, \ldots, f^\nu_{t-1}, e^\nu_t+f^\nu_t\>
&\tif 2t<m.
\end{cases}
\endeq  
We may assume now $j=1$ since the other case can be proved by symmetry.
Take $F_\bullet \in \pi\inv(U_1)$ so that $F_1 = \<e_1^\ld + \gamma f^\ld_1\>$. The relation for the cup connecting $1$ and $2t$ is
\eq\label{eq:oldcupII}
\begin{cases}
F_{2t} =\<e^\ld_1, \ldots e^\ld_t, f^\ld_1, \ldots, f^\ld_t\>
&\tif 1 = \sigma(2t) \in V^{\da}_l, 2t=m;
\\
x_\ld^t F_{2t} = F_1, 
&\tif 1 = \sigma(2t) \in X^{\da}_l, 2t=m;
\\
x_\ld^t F_{2t} = F_1, 
\quad 
x^{m-2t} F_{2t}^\perp = F_{2t}
&\tif 1 = \sigma(2t) \in X^{\da}_l, 2t < m.
\end{cases} 
\endeq
A direct application of Lemma~\ref{lem:Q}(b) shows that  \eqref{eq:oldcupII} leads to \eqref{eq:newray} in either case.
\endproof
\begin{prop}\label{prop:caseII}
Let $\da, \dc$ be defined as in Definition~\ref{def:ac}. 
In both cases, the subvariety $K^{\da}$ is a $\CP^1$-bundle with typical fiber $K^{\dc}$.
\end{prop}
\proof
For $j=1,2$, it suffices to check that the diagram below commutes:
\eq\label{eq:IIcomm}
\begin{tikzcd}
\pi\inv(U_j) \ar[r, "{(\pi,\phi_j)}","\simeq"'] \ar[d,"\pi_i"]
& U_j\times K^{\dc} \ar[ld, "\textup{proj}_1"]
\\
U_j
&
\end{tikzcd}
\endeq
It is routine to check that $(\pi,\phi_j)$ is a homeomorphism with its inverse map given by
$([a:b], F''_\bullet) \mapsto F_\bullet$, where
\eq
F_i = 
\begin{cases}
\<ae_1^\ld+bf_1^\ld\> &\tif i=1;
\\
\psi\inv((Q^{\II}_j)\inv(F''_{i-1})) &\tif 2 \leq i \leq m;
\\
F_{n-i}^\perp &\tif m+1 \leq i \leq n.
\end{cases}
\endeq
\endproof
\subsection{Case III}
In this section we assume that
\eq\label{eq:caseIII}
n =2m\geq 6,
\quad
\ld = (n-k,k),
\quad
\da\in \BD_{\ld},
\quad
1 \not\in V^{\da}_l \sqcup X^{\da}_l
\endeq
That is, vertex 1 is connected to a ray.
We will discuss the following subcases:
\begin{enumerate}[~]
\item Case III--1: $\da = \icupa$ in which $a_2$ contains no rays.
\item Case III--2: $\da = \icupb$ in which $a_2$ contains no rays.
\item Case III--3: $\da = \icupa$ in which $a_2$ contains exactly one ray.
\item Case III--4: $\da = \icupa$ in which $a_2$ contains at least two rays.
\end{enumerate}
Let $\dc \in \BD_{\nu}$ be the marked cup diagram obtained from $a_2$ by decreasing all indices by one. 
We are going to show that $K^{\da}$ is isomorphic to $K^{\dc}$ so that the inductive hypothesis applies since $\nu \vdash n-2 < n$.

For case III--$l$ ($1\leq l \leq 4$), by Lemma~\ref{lem:Q} we have a formed space isomorphism $Q^{\III}_l$. Below we summarize the data associated to each subcase:
\eq\label{eq:data}
\begin{array}{|c|c|c|c|c|}
\hline
\textup{Case} & \textup{III--}1& \textup{III--}2& \textup{III--}3& \textup{III--}4
\\
\hline
\ld & (m,m) &(m,m) & (m+1,m-1) & (n-k,k)
\\
W=F_1&\<e^\ld_1\> & \<f^\ld_1\> &\<e^\ld_1\>&\<e^\ld_1\>
\\
\nu & (m-1,m-1) & (m-1,m-1) & (m-1,m-1) & (n-k-2,k)
\\
Q^{\III}_l&\eqref{eq:III-1}&\eqref{eq:III-2}&\eqref{eq:III-3}&\eqref{eq:III-4}
\\
\hline
\end{array}
\endeq
\lemma\label{lem:KcIII}
Let $\da, \dc$ be defined as in Definition~\ref{def:ac}. 
Recall $\Omega$ from \eqref{eq:Omega} using $Q = Q^{\III}_l$, $1\leq l \leq 4$ from Lemma~\ref{lem:Q}(b)--(c). 
Then the maps below are well-defined: 
\eq
\phi_l:  K^{\da} \to K^{\dc},
\quad
F_\bullet \mapsto Q^{\III}_l(F_\bullet).
\endeq
\endlemma
\proof
The lemma can be proved using a routine case-by-case analysis in a similar way as Lemmas \ref{lem:Kc} and \ref{lem:KcII} but is easier.
\endproof
\begin{prop}\label{prop:caseIII}
Retain the notations of  Lemma~\ref{lem:KcIII}.
The map $\phi_l:  K^{\da} \to K^{\dc}$ is an isomorphism.
\end{prop}
\proof
Recall the subspace $W$ from \eqref{eq:data} for each $l$.
It is routine to check that its inverse map is given by $F''_\bullet \mapsto F_\bullet$, where
\eq
F_i = 
\begin{cases}
W &\tif i=1;
\\
\psi\inv((Q^{\III}_l)\inv(F''_{i-1})) &\tif 2 \leq i \leq m;
\\
F_{n-i}^\perp &\tif m+1 \leq i \leq n.
\end{cases}
\endeq
\endproof
\proof[Proof of Theorem~\ref{thm:main3}]
By Lemma~\ref{lem:n<=4}, $K^{\da}$ is an iterated $\CP^1$-bundle of length $\ell$ for $n \leq 4$.
For $n = 2m \geq 6$, an exhaustive list is given in \eqref{eq:list}, and can be divided into three cases.
In either case, $K^{\da}$ is an iterated $\CP^1$-bundle of length $\ell$ thanks to Propositions~\ref{prop:caseI}, \ref{prop:caseII} and \ref{prop:caseIII}.
\endproof
\bibliography{litlist-geom} \label{references}
\bibliographystyle{amsalpha}

\end{document}